\documentclass[letterpaper,11pt]{article}
\usepackage[utf8]{inputenc}
\title{Nonconvex-Nonconcave Min-Max Optimization with a Small Maximization Domain}
\author{
Dmitrii M. Ostrovskii\thanks{School of Mathematics and H. Milton Stewart School of Industrial \& Systems Engineering (ISyE), Atlanta, USA. Email: \texttt{ostrov@gatech.edu}.}
\quad\quad Babak Barazandeh
\quad\quad Meisam Razaviyayn\thanks{Viterbi School of Engineering, University of Southern California, Los Angeles, USA. Email: \texttt{razaviya@usc.edu.}
This work was supported by the NSF CAREER Award CCF2144985 and the AFOSR Young Investigator Program Award FA9550-22-1-0192.}
}
\date{}

\usepackage{fullpage}
\usepackage{microtype}
\usepackage{xfrac}
\usepackage{enumitem}
\usepackage{comment}
\usepackage{color}
\usepackage{amsmath,amsthm,amsfonts,amssymb}
\usepackage[numbers]{natbib}

\allowdisplaybreaks

\usepackage{framed}

\usepackage{xfrac}

\DeclareSymbolFont{boldoperators}{OT1}{cmr}{bx}{n}
\SetSymbolFont{boldoperators}{bold}{OT1}{cmr}{bx}{n}
\edef\bar{\unexpanded{\protect\mathaccentV{bar}}\number\symboldoperators16}

\newcommand{\ApproxMax}{\textsf{\textup{ApproxMax}}}

\usepackage{MnSymbol}

\makeatletter
\newcommand*\rel@kern[1]{\kern#1\dimexpr\macc@kerna}
\newcommand*\wbar[1]{%
  \begingroup
  \def\mathaccent##1##2{%
    \rel@kern{0.8}%
    \overline{\rel@kern{-0.8}\macc@nucleus\rel@kern{0.1}}%
    \rel@kern{-0.2}%
  }%
  \macc@depth\@ne
  \let\math@bgroup\@empty \let\math@egroup\macc@set@skewchar
  \mathsurround\z@ \frozen@everymath{\mathgroup\macc@group\relax}%
  \macc@set@skewchar\relax
  \let\mathaccentV\macc@nested@a
  \macc@nested@a\relax111{#1}%
  \endgroup
}
\makeatother

\newcommand{\NE}{\textsf{NE}}
\newcommand{\FNE}{\textsf{FNE}}

\newcommand{\Coupled}{\textsf{\textup{Coupled}}}
\newcommand{\Naive}{\textsf{\textup{Naive}}}

\renewcommand{\succeq}{\succcurlyeq}

\usepackage{url}
\usepackage[dvipsnames]{xcolor}
\usepackage[colorlinks]{hyperref}
\hypersetup{
	colorlinks=true,
    linkcolor=red,
    citecolor=OliveGreen,
    filecolor=black,
    urlcolor=black,
}

\usepackage{algorithm}
\usepackage{algpseudocode}

\newcommand{\pinv}{\dagger}
\newcommand{\tHVP}{\textsf{time$($HVP$)$}}

\newcommand{\mHVP}{\textsf{mem$($HVP$)$}}

\usepackage{pdfpages}
\usepackage{multirow}
\usepackage{dsfont,makecell}
\usepackage{color, colortbl}

\usepackage{amsopn}
\DeclareMathOperator*{\Argmin}{Argmin}
\DeclareMathOperator*{\Argmax}{Argmax}
\DeclareMathOperator*{\argmin}{argmin}
\DeclareMathOperator*{\argmax}{argmax}

\newtheorem{theorem}{Theorem}[section]
\newtheorem{lemma}{Lemma}[section]
\newtheorem{definition}{Definition}
\newtheorem{proposition}{Proposition}[section]

\newtheorem{assumption}{Assumption}

\algnewcommand{\LineComment}[1]{\State \(\triangleright\) #1}

\newcommand{\Sx}{\mathsf{S}_{X}}
\newcommand{\Sy}{\mathsf{S}_{Y}}

\newcommand{\Sz}{\textsf{\textup S}_{Z}}

\newcommand{\envalias}[2]{%
  \expandafter\let\csname#1\expandafter\endcsname\csname#2\endcsname
  \expandafter\let\csname end#1\expandafter\endcsname\csname end#2\endcsname
}
\envalias{eq}{equation}
\envalias{eq*}{equation*}
\envalias{ald}{aligned}

\newcommand{\N}{\mathds{N}}
\newcommand{\fapx}{\hat f}
\newcommand{\vphiapx}{\hat \vphi}
\newcommand{\x}{{\textsf{\textup x}}}
\newcommand{\y}{{\textsf{\textup y}}}
\newcommand{\XX}{E_{\x}}
\newcommand{\YY}{E_{\y}}

\newcommand{\gx}{\nabla_\x}
\newcommand{\gy}{\nabla_\y}
\newcommand{\hxx}{\nabla^2_{\x{}^2}}
\newcommand{\hyy}{\nabla^2_{\y{}^2}}
\newcommand{\hxy}{\nabla^2_{\x\y}}
\newcommand{\Ty}[1]{\nabla^{#1}_{\y^{#1}}}
\newcommand{\Txy}[2]{\nabla^{#1}_{\x\y^{#2}}}
\newcommand{\Txxy}[2]{\nabla^{#1}_{\x^2\y^{#2}}}
\newcommand{\Dy}{\mathsf{D}}

\newcommand{\rhoy}{\rho_{\y\y\y}}

\newcommand{\diam}{\textup{diam}}

\newcommand{\mbar}{\overline m}
\newcommand{\Lxx}{\lam}

\newcommand{\Lxxp}{\bar\lam}
\newcommand{\mucrit}{\mu_{{\mathsf{cr}}}}

\newcommand{\Lyy}{N}
\newcommand{\Lxy}{\mu}
\newcommand{\Lxyp}{\bar\mu}

\newcommand{\cLyky}{\rho_{k}}

\newcommand{\cLykx}{\sigma_{k}}

\newcommand{\cLykxx}{\tau_{k}}

\newcommand{\by}{\hat y}
\newcommand{\ex}{\veps_\x}
\newcommand{\ey}{\veps_\y}

\newcommand{\gam}{\gamma}
\newcommand{\lam}{\lambda}
\newcommand{\Ry}{\mathsf{R}}

\newcommand{\proofstep}[1]{$\boldsymbol{{#1}^o}$}

\newcommand{\reg}{\textsf{\textup{reg}}}
\newcommand{\vphi}{\varphi}

\newcommand{\sig}{\sigma}

\newcommand{\half}{\frac{1}{2}}

\newcommand{\eps}{\epsilon}
\newcommand{\Gap}{\Delta}

\newcommand{\gamx}{\gam_\x}
\newcommand{\gamy}{\gam_\y}

\renewcommand{\le}{\leqslant}
\renewcommand{\ge}{\geqslant}
\newcommand{\proj}{\Pi}

\newcommand{\nn}{\notag\\}

\newcommand{\ba}{\begin{array}{c}}
\newcommand{\bal}{\begin{array}{l}}
\newcommand{\ea}{\end{array}}

\newcommand{\R}{\mathds{R}}

\newcommand{\bit}{\begin{itemize}}
\newcommand{\eit}{\end{itemize}}

\newcommand{\bvec}{\left(\!\!\!\begin{array}{c} }
\newcommand{\evec}{\end{array}\!\!\!\right)}
\newcommand{\E}{\mathds{E}}

\newcommand{\wt}{\widetilde}

\newcommand{\sign}{\mathop{ \rm sign}}

\newcommand{\lsim}{\lesssim}
\newcommand{\gsim}{\gtrsim}

\newcommand{\barr}{\begin{array}}
\newcommand{\earr}{\end{array}}

\newcommand{\ind}{\mathds{1}}

\newcommand{\prob}{{p}}
\newcommand{\cK}{\mathcal{K}}

\newcommand{\Uniform}{\textup{\textsf{Uniform}}}

\newcommand{\lammax}{\lambda_{\max}}

\newcommand{\wh}{\widehat}
\newcommand{\veps}{\varepsilon}

\newcommand{\lang}{\left\langle}

\newcommand{\rang}{\right\rangle}

\usepackage{algorithm}
\usepackage{algpseudocode}

\usepackage{cleveref}
\crefname{algorithm}{Algorithm}{Algorithms}
\crefname{assumption}{Assumption}{Assumptions}
\crefname{equation}{}{}
\crefname{figure}{Fig.}{Figs.}
\crefname{table}{Table}{Tables}
\crefname{section}{Section}{Sections}
\crefname{subsection}{Section}{Sections}
\crefname{theorem}{Theorem}{Theorems}
\crefname{lemma}{Lemma}{Lemmas}
\crefname{proposition}{Proposition}{Propositions}
\newtheorem{remark}{Remark}
\crefname{definition}{Definition}{Definitions}
\crefname{corollary}{Corollary}{Corollaries}
\crefname{remark}{Remark}{Remarks}
\crefname{example}{Example}{Examples}
\crefname{appendix}{Appendix}{Appendices}

\newcommand{\odima}[1]{{{\color{black}#1}}}

\usepackage{bbm}
\usepackage{tikz}
\usetikzlibrary{decorations.pathreplacing,calc}

\newcommand{\cP}{{\mathcal{P}}}

\newcommand{\cY}{{\cal Y}}

\newcommand{\dom}{{\rm dom}}

\newcommand{\conv}{\textup{conv}}
\newcommand{\clconv}{\overline{\conv}}

\begin{document}
\maketitle

\begin{abstract}
We study the problem of finding approximate first-order stationary points in optimization problems of the form $\min_{x \in X} \max_{y \in Y} f(x,y)$, where the sets $X,Y$ are convex and $Y$ is compact; the objective function $f$ is smooth, but assumed neither convex in $x$ nor concave in $y$. Our approach relies upon replacing the function $f(x,\cdot)$ with its $k^\text{th}$-order Taylor approximation (in $y$) and finding a near-stationary point in the resulting surrogate problem. To guarantee its success, we establish the following result: let the Euclidean diameter of $Y$ be small in terms of the target accuracy $\varepsilon$, namely $O(\varepsilon^{\frac{2}{k+1}})$ for $k \in \mathbb{N}$ and $O(\varepsilon)$ for $k = 0$, with the constant factors controlled by certain regularity parameters of $f$; then any $\varepsilon$-stationary point in the surrogate problem \odima{is}~$O(\varepsilon)$-stationary for the initial problem, and \odima{vice versa.} Moreover, we show that these upper bounds are nearly optimal: the aforementioned reduction provably fails when the diameter of $Y$ is larger. For $0 \le k \le 2$, the surrogate function can be efficiently maximized in $y$; our general approximation result then allows to build efficient algorithms for finding a near-stationary point in nonconvex-nonconcave min-max problems, for which we also provide convergence guarantees.
\end{abstract}

\section{Introduction}
\label{sec:intro}

In the past few years, min-max optimization has become popular among practitioners due to its relevance for machine learning applications -- in particular, when training generative adversarial networks (GANs) \cite{goodfellow2014generative, sanjabi2018convergence, nowozin2016f, gulrajani2017improved}, \odima{in} robust machine learning~\cite{madry2018towards, pourbabaee2020robust, zhang2019theoretically, sagawa2019distributionally, hu2018does, sinha2018certifying}, in fair statistical inference~\cite{rezaei2020fairness,baharlouei2020renyi, zhang2018mitigating, xu2018fairgan, lowy2021fermi, adel2019one}, in reinforcement learning~\cite{dai2018sbeed}, distributed optimization and learning over networks~\cite{rabbat2004distributed,li2013designing,hajinezhad2016nestt,hong2017prox,hong2018gradient}, and for optimal resource allocation in multi-agent systems~\cite{scutari2008optimal,scutari2013joint}.
The common task arising in these applications, as well as in many others, is solving optimization problems of the general form
\begin{eq}
\label{opt:min-max}
\min_{x \in X} \max_{y \in Y} f(x,y),
\tag{$\mathsf{P}$}
\end{eq}
where~$X,Y$ are some convex sets in the corresponding high-dimensional Euclidean spaces and~$f$ is smooth in both variables, possibly in a heterogeneous manner. %
Min-max optimization has been an active area of research since the early works of Nash, von Neumann, and Morgenstern~\cite{nash1950equilibrium,neumann1928theorie,morgenstern1953theory}.
A very important subclass of such problems where~$f$ is {\em convex-concave}---i.e., $f(\cdot,y)$ is convex and $f(x,\cdot)$ is concave for all~$(x,y) \in X \times Y$---has been studied extensively from the algorithmic viewpoint starting from the seminal work of Nemirovski~\cite{nemirovski2004prox}, who proposed a first-order algorithm with an~$O(1/T)$ dimension-independent convergence guarantee, generalizing the extragradient algorithm from the early work of Korpelevich~\cite{korpelevich1977extragradient}.
This has been followed by many extensions (e.g.,~\cite{juditsky2011solving,he2015mirror,he2015semi,juditsky2016solving,cox2017decomposition}) and applications in machine learning, statistics, and signal processing; see~\cite{xiao2010dual,optbook2,nesterov2013first} and references therein. 
Convex-concave min-max optimization is still a very active area of research: see, e.g., the recent works~\cite{wang2020improved,zhang2021lower} establishing the matching upper and lower complexity bounds under strong convexity-concavity; the recent proliferation of works on last-iterate convergence~\cite{daskalakis2018last,abernethy2019last,golowich2020last,wei2020linear}; the work~\cite{juditsky2021well} extending the functionality of CVX~\cite{grant2014cvx} (a.k.a.~``disciplined convex programming'') to convex-concave min-max problems and \odima{variational inequalities with monotone operators (MVIs)}~\cite{kinderlehrer2000introduction}.

\vspace{0.1cm}

Meanwhile, many modern applications where~\cref{opt:min-max} appears fall beyond the convex-concave scenario. 
For example, in the standard formulation of GANs~$f$ is neither convex in~$x$ nor concave in~$y$ \cite{goodfellow2014generative}; this is also the case in the task of adversarially-robust deep neural network training~\cite{rice2020overfitting}; other applications of nonconvex-nonconcave min-max optimization can be found in the recent review article~\cite{razaviyayn2020nonconvex}. 
Problems in which~$f$ is concave in~$y$ but nonconvex in~$x$ arise in fair inference ``\`a la R\'enyi''~\cite{baharlouei2020renyi} and when minimizing the maximum of smooth functions~\cite{zhang2020single} (in the latter case~$f(x,\cdot)$ is even affine). 
Moreover,  machine learning  under distributional uncertainty~\cite{sinha2018certifying, sagawa2020distributionally,hu2018does},  power control for wireless communication~\cite{lu2020hybrid}, and special formulations of the task of learning  from multiple domains~\cite{qian2019robust} all lead to nonconvex-concave min-max optimization.
In the past couple of years, these new applications reignited a substantial interest across the optimization theory community for analyzing~\cref{opt:min-max} 
in the general {nonconvex-nonconcave} setup, the {nonconvex-concave} setup as an approachable stepping stone towards it, and in other ``intermediate'' scenarios, where~$f(\cdot,y)$ is nonconvex, and~$f(x,\cdot)$ is not concave but has some other special structure allowing to efficiently solve the nested maximization problem---that is, \odima{to} evaluate at any point~$x \in X$ the {\em \odima{max-function}}\,\footnote{Note also that with~$f(\cdot,y)$ nonconvex, the minimax theorem does not apply, and the order of~$\min$ and~$\max$ in~\cref{opt:min-max} becomes important. %
As such, there is a crucial difference between nonconvex-concave and convex-nonconcave instances of~\cref{opt:min-max}: 
the former ones are easier, as in them  the \odima{max-function}~$\vphi(\cdot)$, cf.~\cref{def:prim-fun} is merely hard to minimize, while in the convex-nonconcave case we typically cannot even {\em evaluate} it at a point.}
\begin{equation}
\label{def:prim-fun}
\vphi(x) := \max_{y \in Y} f(x,y).
\end{equation}
The nonconvex-concave case, first addressed in~\cite{nouiehed2019solving}, is by now quite well understood. 
Generally, we lose all hope of actually solving~\cref{opt:min-max}---or minimizing~$\vphi(\cdot)$---when~$f(\cdot,y)$ is nonconvex, as we then deal with a {\em nonconvex} minimization problem even for a singleton~$Y$.
Hence, one has to be satisfied here with approximating a {\em local} minimizer of~$\vphi(\cdot)$, or a {\em local} saddle point (also called a Nash equilibrium) in~\cref{opt:min-max}---i.e., a point~$(x^{\NE}, y^{\NE}) \in X \times Y$ such that
\[
f(x^\NE,y) \le f(x^\NE, y^\NE) \le f(x,y^\NE) \;\; \text{for all} \;\; (x,y) \in X \times Y \;\; \text{in a neighborhood of} \;\; (x^\NE, y^\NE).
\]
In fact, even these two tasks can be too ambitious, and most of recent studies have been focused on the tasks of approximating a {\em first-order stationary point} in~\cref{opt:min-max} up to accuracy~$\veps > 0$, as measured by the norm of either the proximal gradient of~$f$ or the gradient of the standard Moreau envelope of~$\vphi(\cdot)$. 
(The notion of the Moreau envelope, essential in the context of our work, shall be recalled in Section~\ref{sec:problem}.)
Painting with a broad brush, these two accuracy measures turn out to be (essentially) equivalent in the nonconvex-concave case, and settling on one of them is largely a matter of personal preference; see~\cite[Section~5]{ostrovskii2020efficient} for a technical discussion of this equivalence.
In this line of research, the works~\cite{ostrovskii2020efficient,thekumparampil2019efficient,lin2020near,zhao2020primal} 
demonstrated that an~$\veps$-first-order stationary point with respect to the Moreau envelope criterion, referred to as $\veps$-FOSP from now on, can be found in~$\wt O(\veps^{-3})$ queries of the gradient of~$f$, and in~$\wt O(\veps^{-2})$ queries when~$f(x,\cdot)$ is strongly concave.\footnote{We use the standard ``big-O'' notation:~$g = O(h)$ or~$g \lsim h$ both hide an absolute constant, i.e., means that~$g \le c h$ for some~$c > 0$ uniformly over all allowed pairs of values of~$g,h$; similarly,~$g = \wt O(h)$ is a shorthand for~$g = O(h \log(h+e))$.}
The latter of these estimates was recently shown to be worst-case optimal~\cite{han2021lower,zhang2021complexity,li2021complexity}, and it is widely believed---although, to the best of our knowledge, not yet proved---that the former estimate is also tight.
Furthermore, in the absence of concavity, but assuming access to an abstract {\em maximization oracle} evaluating~$\vphi(x)$, the authors of~\cite{jin2020local} observed that an algorithm earlier proposed in~\cite{davis2018stochastic} allows to compute an~$\veps$-FOSP in~$O(\veps^{-4})$ queries of the maximization oracle and of the~$x$-gradient.

\vspace{-0.2cm}

\paragraph{Nonconvex-nonconcave problems.}
By contrast, our understanding of the general nonconvex-nonconcave setup is rather fragmentary.
One delicate issue here is that the two viewpoints on local optimality in~\cref{opt:min-max}---%
the one focusing on finding a (local) minimizer~$x^*$ of~$\vphi(x) = \max_{y \in Y} f(x,y)$, historically due to Stackelberg~\cite{von1934marktform}, %
and the game-theoretic paradigm of Nash~\cite{nash1950equilibrium}, where one aims at finding a local Nash equilibrium---might result in quite different notions of (local) optima and stationary point when~$f(x,\cdot)$ is not concave~\cite{jin2020local}.
Indeed: consider, as an illustration, the problem
\begin{equation}
\label{eq:intro-hard}
\min_{x\in\R \vphantom{[-2,2]}} \max_{y \in [-2,2]} 
\left\{ 
f(x,y) := xy - \tfrac{1}{3}y^3
\right\}.
\end{equation} 
By computing~$\vphi(x)$ one can verify that~$x^* = 1$ is a unique local (hence also global) minimizer of~$\vphi(x)$, and~$\max_{y \in [-2,2]} f(1,y)$ is attained at~$y^* = -2$; in other words,~$(1,-2)$ is a unique solution to~\cref{eq:intro-hard}. On the other hand, a unique first-order Nash equilibrium, that is~$(x^{\FNE}, y^{\FNE})$ such that
\begin{equation}
\label{eq:first-order-nash}
\lang \gx f(x^{\FNE},y^{\FNE}), x-x^{\FNE} \rang \ge 0
\;\; \text{and} \;\;
\lang \gy f(x^{\FNE},y^{\FNE}), y-y^{\FNE} \rang \le 0 
\;\; \text{for all} \;\; 
(x,y) \in X \times Y,
\end{equation}
is the point~$(0,0)$. 
Finally,~\cref{eq:intro-hard} does not have any (local) Nash equilibrium,
and these conclusions remain valid if~$x$ restricted to a compact set (e.g.,~$X = [0,4]$).
Meanwhile,~$\inf_{x \in X} \vphi(x)$ is always finite in~\cref{opt:min-max}, and is attained whenever~$\vphi(x)$ is bounded from below on~$X$ (in particular, whenever~$X$ is compact), and must satisfy the necessary condition~$\partial \vphi(x^*) \ni 0$ in terms of the weak subdifferential. 
(To make the paper self-contained, we provide background on weak convexity and weak subdifferentials in~\cref{app:weakly-convex}; for the present non-technical discussion this is irrelevant.)
Together with the above example, this observation makes a case for the \odima{Stackelberg} approach as the one better suited for the general nonconvex-nonconcave scenario, and we adhere to this approach in our work. Another argument in favor of the Stackelberg approach is that in most of the modern applications of nonconvex-nonconcave min-max optimization we mentioned, the actual {\em practical} goal is to minimize the \odima{max-function}~$\vphi(\cdot)$, not to find a local Nash equilibrium.

\vspace{0.1cm}

The second challenge posed by nonconvex-nonconcave min-max optimization is an exceptional algorithmic difficulty of {\em finding} (approximately) even a {\em local} first-order stationary point or Nash equilibrium when no additional structure is imposed on~\cref{opt:min-max}.
\odima{In particular,~\cite{daskalakis2020complexity} shows the following results assuming that~$X,Y = [0,1]^d$ and~$f$ is~$L$-smooth,~$G$-Lipschitz, and takes values in~$[-B,B]$.}
\begin{itemize}

\item[$(i)$] 
The problem of exhibiting~$(x,y) \in [0,1]^{2d}$ such that~$\|\nabla f(x,y)\| \le \frac{1}{24}$ or detecting that no such point exists,\footnote{In what follows,~$\|\cdot\|$ stands for the standard Euclidean norm on a given Euclidean space (defined by the context).} is~\textsf{FNP}-complete (in~$d$) in the regime~$L = d$,~$G = \sqrt{d}$,~$B = 1$ \cite[Theorem~4.1]{daskalakis2020complexity}.\footnote{See~\cite{papadimitriou1994complexity,daskalakis2009complexity} for a technical background on the complexity classes~\textsf{FNP},~\textsf{PPAD}---in particular, in the context of~\cref{opt:min-max}.}

\item[$(ii)$] 
\odima{If~\eqref{opt:min-max} is replaced with~$\min_{x \in X} \max_{y \in Y} f(x,y) + I_{\cP}(x,y)$ where~$I_{\cP}$ is the indicator function of a convex polytope~$\cP \subseteq X \times Y$,} then an \emph{$(\eps,\odima{\delta})${-local Nash equilibrium,}} i.e.~$(x^*,y^*) \in \cP$ such that
\[
f(x^*,y) - \eps < f(x^*,y^*) < f(x,y^*) + \eps  \;\; \forall (x,y) \in \cP: \; \max\{\|x-x^*\|, \|y - y^*\|\} \le \odima{\delta,}
\]
is guaranteed to exist in the local regime~$\odima{\delta} \le \odima{\sqrt{{2\eps}/{L}}}$; however\odima{,} exhibiting it is~\textsf{PPAD}-hard already when~$\odima{\delta} \ge \sqrt{\odima{{\eps}/{L}}}$ and~$\max\{L, G, B, 1/\eps,1/\odima{\delta}\} = O(\textsf{poly}(d))$; see~\cite[Theorems~4.3--4.4]{daskalakis2020complexity}.\footnote{\odima{This result leverages a major restriction, assuming the coupled structure of the joint feasible set:~$\cP = \{(x,y) \in [0,1]^{2d}: \max\{|x_1-y_1|, |x_2-y_2|\} \le \delta \}$ with~$d \ge 2$; see~\cite[Eq.~(6.1)]{daskalakis2020complexity}. The uncoupled constraints~($\cP = X \times Y$) hardness result was very recently obtained in~\cite{bernasconi2026min}.
}
}
\end{itemize}
These two negative results demonstrate that approximating even a stationary point, or a local min-max point for a large enough neighborhood, is virtually impossible without imposing additional structure on~\cref{opt:min-max}.
As such, existing {\em positive} results rely on adding this structure in one way or another. In particular, a common methodology is to ``mimic'' the case of~\odima{an MVI}~\cite{nemirovski2004prox} by imposing pseudomonotonicity, the Minty condition, or contraction of the best response mappings~\cite{facchinei2007finite,PangRaza16,dang2015convergence,mertikopoulos2018mirror,rafique2021weakly,song2021optimistic,diakonikolas2021efficient,lee2021semi,liu2019decentralized,barazandeh2021solving}. 
However, these assumptions are restrictive and rarely satisfied in modern applications; see~\cite{razaviyayn2020nonconvex} for a detailed discussion. 
In addition, they are tied to the Nashian paradigm, wheras applications tend to call for the Stackelbergian one. 
Other interesting approaches rely upon restricting the nature and level of coupling between the variables~\cite{grimmer2020landscape} or the power of the maximizing player~\cite{fiez2021minimax}. 
Yet another recent line of research advocated an alternative to the local Nash equilibria---``greedy adversarial equilibria'' that are computationally feasibile, but at the expense of a certain loss of transparency and interpretability~\cite{mangoubi2021greedy,keswani2020convergent}. 
Finally, a growing body of literature is devoted to the asymptotic behavior of algorithms~\cite{wang2019solving,domingo2020mean,mazumdar2019finding,fiez2020gradient,hsieh2020limits}, and structural results for GANs and adversarial training~\cite{heusel2017gans,mescheder2017numerics,liang2018interaction,farnia2020gans}.

Going back to the hardness results~$(i)$--$(ii)$ established in~\cite{daskalakis2020complexity}, the latter of them hints at a possibility to control the complexity of~\cref{opt:min-max} through the size of a feasible set. 
Our work explores this possibility; let us now present the main ideas behind our approach.
\vspace{-0.2cm}

\paragraph{Outline of our approach.}
Our work is motivated by a trivial observation: finding an approximate first-order stationary point of~\cref{opt:min-max} becomes a computationally feasible task when~$Y$ is {\em a singleton}, as in this case we deal with a smooth minimization problem, so an approximate stationary point can be found by running projected gradient descent. 
Furthermore, one can hope that the computational tractability is preserved when~$Y$ is not singleton, but is {\em small.} %
\odima{Indeed, letting~$\fapx_k(x,\cdot)$ be the
$k^{th}$-order Taylor approximation of $f(x,\cdot)$ for some $k\ge0$. This approach requires assuming sufficient regularity of $f$, specifically higher-order smoothness in $y$ matching the order $k$ (detailed in Section 2). One can then advocate the following two-step approach.
}

\vspace{-0.2cm}
\begin{framed}
\begin{itemize}
\vspace{-0.2cm}
\item[\proofstep{1}.]
Prove that any~$\veps$-FOSP in the surrogate min-max problem is also an~$O(\veps)$-FOSP in~\cref{opt:min-max}, due to~$\fapx_k(x,\cdot)$ approximating~$f(x,\cdot)$ over the (small) set~$Y$ well enough for our purposes.%
\item[\proofstep{2}.]
Find~$\veps$-FOSP in the surrogate problem---and thus~$O(\veps)$-FOSP in~\cref{opt:min-max}---by exploiting the structure of~$\fapx_k(x,\cdot)$ such as, e.g., linearity for~$k = 1$ or quadratic structure for~$k = 2$.
\end{itemize}
\vspace{-0.4cm}
\end{framed}
A natural question immediately arising in connection with this strategy, namely with~\proofstep{1}, is: 
\vspace{-0.1cm}
\begin{quote}
{\em How small the diameter of~$Y$ has to be in order for the reduction in step~\proofstep{1} to be valid?}
\end{quote}
\vspace{-0.1cm}
One would expect this question to have a nontrivial answer. 
Indeed, 
it seems unlikely that Taylor approximation would work over a set of a constant diameter.
As such, one may expect that
\begin{equation}
\label{eq:general-D-bound}
\diam(Y)  \lsim \veps^{\odima{\mathsf{\gamma}(k)}}
\end{equation}
allows for the reduction in~\proofstep{1} to work \odima{when the~$k$-order Taylor approximation is used,} where~\odima{$\gamma(k)$ is a nonincreasing positive function (``the more smoothness, the better'')} and the hidden constant \odima{does not depend on~$\veps$ or~$\Dy$} (but might depend on the regularity parameters of~$f$). 
Later on\odima{,} we shall verify this hypothesis, proving explicit bounds of the form~\cref{eq:general-D-bound} for arbitrary~$k$, under the natural regularity assumptions on~$f$, and showing that they are not only sufficient, but also necessary---that is, the reduction in~\proofstep{1} may fail for~$\diam(Y)$ beyond the allowed threshold. Finally, we shall also implement step~\proofstep{2} of the strategy by designing efficient algorithms for solving the surrogate min-max problem.
\vspace{-0.2cm}

\paragraph{Applications.}

In several modern machine learning and signal processing applications, it is natural to assume that~$\diam(Y)$ is relatively small. A prominent example is the task of training neural networks robust against adversarial attacks. In this context, diam(Y) corresponds to the magnitude (perturbation budget) of the attack, which must be small for the attack to remain imperceptible~\cite{madry2018towards,  zhang2019theoretically}.

Furthermore, a recent approach termed sharpness-aware minimization (SAM)~\cite{foret2020sharpness} aims to improve generalization by avoiding sharp local minima of the training loss. This results in optimization problems of the form $\min_{w} \max_{||u||\le r}L(w+u),$ where $L(\cdot)$ is the training loss, $u$ is the perturbation of a model $w$. Crucially, the radius $r$---corresponding to our $\Dy$---must be small, as the analysis often relies upon local approximations of $L(\cdot)$ near $w$. More generally, designing systems that are robust against small perturbations of certain parameters naturally leads to min-max problems of form (P) with a small maximization domain $Y$.

\vspace{-0.2cm}

\paragraph{Paper organization and summary of contributions.}
In a nutshell, our work is concerned with implementing both steps~\proofstep{1}--\proofstep{2} of the proposed approach; combined together, they result in efficient algorithms for finding an~$\veps$-FOSP in~\cref{opt:min-max}. The rest of the paper is organized as follows.

In~\cref{sec:problem}\odima{,} we set off by stating our assumptions about~$f$. 
In a nutshell, we grant Lipchitzness of the~$x$-gradient~$\gx f$ and of the order-$k$ differential tensor~$\odima{\nabla_{\y^k}^k f :=} \nabla_{\y\cdots\y}^k f$, for given~$k \in \N \cup \{0\}$, 
in accordance with our intention of using the~$k$-order Taylor expansion of~$f(x,\cdot)$.
We then recall some mathematical background on~\cref{opt:min-max}, including the definitions of the Moreau envelope and approximate first-order stationary points, and finally define the reduction in~\proofstep{1} in a rigorous way.

In~\cref{sec:upper}\odima{,} we formulate, discuss, and prove our first result: the general bound of the form~\cref{eq:general-D-bound}, holding for {\em arbitrary}~$k \in \N \cup \{0\}$ under the {\em matching} regularity assumption (i.e., with the same~$k$).
Informally, we show (cf.~\cref{th:upper-bound}) that, for any~$k \in \N \cup \{0\}$, an arbitrary~$\veps$-FOSP in the counterpart of~\cref{opt:min-max} with~$\fapx_k$ instead of~$f$ as objective function, is also an~$\veps$-FOSP in~\cref{opt:min-max} whenever
\odima{\begin{equation}
\label{eq:diam-intro}
\diam(Y) \le C\max \left\{ (k+1) \left(\frac{\veps^2}{\lam \rho_k} \right)^{\frac{1}{k+1}}, \;\; \frac{\veps}{\mu} \, \right\},
\end{equation}}%
\odima{where~$\Lxx,\Lxy,\rho_k$ are the uniform over~$X \times Y$ Lipschitz constants of~$\gx f(\cdot,y), \gx f(x,\cdot), \nabla_{\y \dots \y}^k f(x,\cdot)$ respectively,
and~$C$ is a numerical constant.} 
In particular, this implies~\cref{eq:general-D-bound} with~\odima{$\gamma(k) = \frac{2}{\max\{k,1\}+1}$ once~$\veps$ is below a certain threshold level, whose value is defined by the regularity parameters~$\lambda, \mu,\rho_k$.}

In~\cref{sec:lower}\odima{,} we present our second series of results, showing that the bound~\cref{eq:diam-intro} is nearly tight. 
We do this by constructing specific instances of~\cref{opt:min-max}, carefully choosing the center~$\by \in Y$ of the Taylor approximation, and then exhibiting a point~$x \in X$ which is {\em stationary} in the approximated problem, while {\em not} being an~$\veps$-FOSP in~\cref{opt:min-max}. 

In~\cref{sec:algos}\odima{,} we carry out step~\proofstep{2} of the strategy. To this end, we suggest three algorithms based on Taylor approximations with~$k \in \{0,1,2\}$. For~$k = 0$ and~$k = 1$, the algorithms are based on gradient descent and \odima{gradient descent-ascent respectively}; 
the resulting oracle complexity is~$O(\veps^{-2})$, and~$\diam(Y) \lsim \veps$ is allowed, cf.~\cref{th:algo-0,th:algo-1}. 
For~$k = 2$, our algorithm allows for a larger diameter~$O(\veps^{2/3})$ as per~\cref{eq:diam-intro}. 
Yet, this improvement has a price:~$Y$ must be a Euclidean ball, we \odima{need} access to higher-order derivatives of~$f$, and \odima{the complexity deriorates} to~$O(\veps^{-{\sfrac{13}{3}}})$; see~\cref{th:algo-2}.%
\vspace{-0.2cm}

\paragraph{Notation.}
We use the standard~$O(\cdot)$ and~$\wt O(\cdot)$ notation: given two functions~$g,h > 0$ we use~$g = O(h)$ and~$g \lsim h$ as shorthands for saying that~$g \le c h$ for some~$c > 0$ uniformly over all simultaneously allowed pairs of values of~$g,h$; similarly,~$g = \wt O(h)$ is a shorthand for~$g = O(h \log(h+e))$. Symbol~$c$ denotes a numerical constant whose value is unimportant and might change from line to line. We use {\odima the} abridged notation~$\gx f(x,y),\gy f(x,y)$ for the partial gradients of~$f$ in the first and second argument evaluated at some~$(x,y) \in X \times Y$, and similarly for higher-order tensors of partial derivatives (see Section~\ref{sec:problem}). 
$\|\cdot\|$ stands for the standard Euclidean norm on a given Euclidean space (clear from context and \odima{identified with its dual}) when the argument is a vector, and for the induced operator norm when the argument is a~$k$-tensor. 
Other notation shall be introduced when necessary.

Most of our results are rather technical, and we defer the proofs to the appendix. An exception is made for~\cref{th:upper-bound}, as its proof is not so technical and illuminates the mechanism behind~\cref{eq:diam-intro}.

\section{Standing assumptions, definitions, and technical background}
\label{sec:problem}

We shall focus on~\cref{opt:min-max} assuming that~$X,Y$ are two convex sets with \odima{nonempty} interior\odima{s} in the corresponding Euclidean spaces~$\XX, \YY$;~$Y$ is compact and its Euclidean diameter is bounded by~$\Dy$,
\[
\Dy \ge \diam(Y) := \max_{y,y' \; \in \; Y} \|y'-y\|.
\]
Let~$\|\cdot\|$ be the (Euclidean) operator norms on the spaces of multilinear forms on~$\XX^{\otimes i} \times \YY^{\otimes j}$ for~$0 \le i \le 2$ and~$0 \le j \le k$, where~$k$ is a nonnegative integer.
We grant two assumptions on the regularity of~$f$.
\begin{assumption}[First-order smoothness in~$x$]
\label{ass:gradx}
The partial gradient~$\gx f(x,y)$ exists and is\odima{~$(\Lxx,\Lxy)$-Lipschitz on~$X \times Y$ for some~$\Lxx > 0$ and~$\Lxy \ge 0$: in other words, for any~$x,x' \in X$ and~$\forall y,y' \in Y,$}
\begin{equation}
\label{eq:gradx-lip} 
\|\gx f(x',y') - \gx f(x,y) \| \le \Lxx \|x' - x\| + \Lxy \|y' - y\|.
\tag{$\mathsf{A1}$}
\end{equation}

\end{assumption}
\begin{assumption}[$k$-order smoothness in~$y$]
\label{ass:tensy}
For a given~$k \in \N \cup \{0\}$, 
the tensor~$\Ty{k} f(x,y)$ %
of~$k^\textup{th}$-order partial derivatives of~$f$ in~$y$ exists and is\odima{~$(\cLykx, \cLyky)$-Lipschitz: for any~$x,x' \in X$ and~$y,y' \in Y$,}
\begin{equation}
\label{eq:tensy-lip}
\|\Ty{k} f(x',y') - \Ty{k} f(x,y) \| \le \cLykx \|x' - x\| + \cLyky \|y' - y\|. 
\tag{$\mathsf{A2}$} 
\end{equation}
Moreover, the tensor~$\Txy{k+1}{k} f$ incorporating the partial derivatives of order~$k$ in~$y$ and~$1$ in~$x$ (in this order) exists and is\odima{~$\cLykxx$-Lipschitz in~$x$ for some~$\cLykxx \ge 0$; in other words, for any~$x',x \in X$ and~$y \in Y$,}
\begin{equation}
\label{eq:tensy-xx}
\|\Txy{k+1}{k} f(x',y) - \Txy{k+1}{k} f(x,y) \| \le \cLykxx \|x' - x\|.
\tag{$\mathsf{A3}$}
\end{equation}

%
%
%
%
%
%
%
%

\end{assumption}

\odima{Assumption~\ref{ass:tensy} merits some discussion.
With~$k = 1$, we recover the classical smoothness assumption, Lipschitzness of the gradient of~$f$.
However, in our setup it makes sense to assume higher regularity in~$y$ (cf.~\eqref{eq:tensy-lip}) as this is the variable in which we perform Taylor expansions (cf.~\eqref{opt:min-max-apx}). 
Additionally,~\eqref{eq:tensy-xx} allows to control the mismatch between the~$x$-gradients of~$f(x,y)$ and~$\fapx_k(x,y)$ at any point~$x$, and implies that~$\gx \fapx_k(\cdot,y)$ is Lipschitz with an adjusted parameter; see~Lemma~\ref{lem:gx-lip}.}

\odima{Let us now more thoroughly discuss the cases~$k \in \{0,1,2\}$ in which we propose efficient algorithms.}
\begin{itemize}
\item
When~$k = 0$,~\eqref{eq:tensy-xx} follows from~\eqref{eq:gradx-lip} with~$\tau_0 = \lam$, and so does not give any extra restrictions. %
As for~\eqref{eq:tensy-lip}, it then requires that~$f$ is~$\rho_0$-Lipschitz in~$y$ and~$\sig_0$-Lipschitz in~$x$. 
In fact, such Lipschitzness conditions are not necessary for our approximation results to hold, although they can possibly improve the approximation bounds in~\cref{th:upper-bound} in the case~$k = 0$, 
which is of marginal practical interest. 
Meanwhile, the Lipschitzness in~$x$ is used in one of the algorithms proposed in~\cref{sec:algos}.
Even so, we have the variational bounds~$\sigma_0 \le \min\{\lam \, \diam(X), \mu\Dy\}$ provided that~$X$ contains in its interior a first-order stationary point of~$f(\cdot,y)$ for any~$y \in Y$. %

\item 
When~$k = 1$, condition~\eqref{eq:tensy-lip} reduces to the Lipschitzness of the partial gradient~$\gy f$, namely
\begin{equation}
\label{eq:grady-lip}
\|\gy f(x',y') - \gy f(x,y) \| \le \rho_1 \|y' - y\| + \sigma_1 \|x' - x\|.
\end{equation}
Moreover, due to~\eqref{eq:tensy-xx} we can assume, without loss of generality, that~$\sigma_1 = \mu$.
Meanwhile,~\eqref{eq:tensy-xx} is a second order condition: it implies that~$\hxx f$ is differentiable in~$y$ almost everywhere on~$X$, and allows to preserve weak convexity in~$x$ after the Taylor expansion in~$y$. 
Note that~\eqref{eq:tensy-xx} holds with~$\tau_1 = 0$ (in fact, with~$\tau_k = 0$ for~$k \ge 1$) \odima{when the objective is {\em bilinearly coupled} (BC):}
\begin{equation}
\label{eq:bilinear-coupling}
f(x,y) = g(x) + \lang Ax, y \rang + h(y).
\tag{$\mathsf{BC}$}
\end{equation}
%
%
%
%

\item
When~$k = 2$, condition~\eqref{eq:tensy-lip} reduces to the Lipschitzness of the partial Hessian~$\hyy f$, namely
\begin{equation}
\label{eq:hessy-lip}
\|\hyy f(x',y') - \hyy f(x,y) \| \le \rho_2 \|y' - y\| + \sigma_2 \|x' - x\|.
\end{equation}
Clearly, under~\eqref{eq:bilinear-coupling} for~$k \ge 2$ one has~$\sigma_k = 0$, and also~$\rho_k = 0$ if, in addition,~$h(y)$ is quadratic.
\end{itemize}
Finally, we grant a mild regularity assumption allowing to differentiate~$\Txy{k+1}{k}f$ in~$x$ under the integral (we use Lebesgue measures on~$\XX$ and~$\YY$); 
it is trivially satisfied if~$\Txxy{k+2}{k} f$ exists everywhere. %
\begin{assumption}
\label{ass:measurable}
Define the set-valued map~$x \mapsto Y_x' \subseteq Y$ where~$Y_x'$ is the set of~$y$ for which~$\Txxy{k+2}{k} f(x,y)$ does not exist. 
Its graph, i.e., the subset of~$X \times Y$ defined as~$\{(x, y): x \in X, \, y \in Y_x'\}$, is measurable.
\end{assumption}

\paragraph{Moreau envelope and the notion of FOSP.}
In this work we use a stationarity criterion based on the Moreau envelope of the \odima{max-function}. 
To define it formally, we first have to remind some standard definitions; readers familiar with the notion of Moreau envelope for weakly convex functions can safely skip this part, while those looking for more details may refer to~\cite{davis2018stochastic,davis2019stochastic} or~\cite[Sec.~5]{ostrovskii2020efficient}.
Recall that~$\vphi(x) := \max_{y \in Y} f(x,y)$ is called the {\em \odima{max-function}} for~\cref{opt:min-max}, cf.~\cref{def:prim-fun}. 
Typically~$\vphi$ is non-smooth (unless~$f(x,\cdot)$ is strictly concave), and so its gradient is not defined on~$X$.
However, under~\cref{ass:gradx}~$\vphi$ is {\em $\Lxx$-weakly convex}---that is, for any~$x \in X$ the function given by~$\vphi(\cdot) + \frac{\Lxx}{2} \|\cdot \|^2$ is convex, and~$(\bar\Lxx - \Lxx)$-strongly convex if regularization is with~$\frac{\bar\Lxx}{2} \|\cdot \|^2$ for~$\bar\Lxx > \Lxx.$
(In order to streamline the presentation while keeping the paper self-contained, we provide necessary background on weakly convex functions and weak subdifferentials in~\cref{app:weakly-convex}.)
\odima{This leads to the following definition.}

\begin{definition}[Moreau envelope]
\label{def:moreau}
Let~$\phi: X \to \R$ be~$\Lxx$-weakly convex. Given an~$\bar\Lxx > \Lxx$, the function %
\begin{equation}
\label{eq:moreau}
\phi_{\bar\Lxx}(x) := \min_{u \in X} 
\big\{
\phi(u) + \tfrac{\bar\Lxx}{2} \|u-x\|^2
\big\}
\end{equation}
is called the~{\em $\bar\Lxx$-Moreau envelope} of~$\phi$, and the unique solution~$x_{\phi/{\bar\Lxx}}^+(x)$ to~\cref{eq:moreau} is called the {\em proximal mapping} of~$x$ (corresponding to~$\phi$ with stepsize~${1}/{\bar\lam}$).
\end{definition}
It is well known (see, e.g.,~\cite[Lemma~2.2]{davis2019stochastic}) that~$\vphi_{\bar\Lxx}$ is~$C^1$-smooth when~$\bar\Lxx > \Lxx$, and moreover it has~$O(\Lxx)$-Lipschitz gradient whenever~$\bar\Lxx = (1+c)\Lxx$ for~$c > 0$.
The standard practice (see, e.g.,~\cite{davis2018stochastic}) is to simply choose~$\bar\Lxx = 2\Lxx$ and use the gradient norm~$\|\nabla \vphi_{2\Lxx}(\cdot)\|$ as the accuracy measure. Indeed, from the first-order optimality conditions for~\cref{eq:moreau} one can easily conclude (see, e.g.,~\cite[p.~4]{davis2019stochastic}) that
\begin{equation}
\label{eq:moreau-explicit}
\nabla\vphi_{2\lam}(x)= 2\Lxx (x - x^+) \in \partial (\vphi + I_X)(x^+),
\end{equation}
where~$x^+ = x^+_{\vphi/({2\lam})}(x)$ for arbitrary~$x \in X$,~$I_X$ is the indicator function of~$X$, and~$\partial(\cdot)$ is the Fr\'echet (or weak) subdifferential; see~\cref{app:weakly-convex-basic} for details. As a result, in the~$x$-unconstrained case ($X = \XX$), the inequality~$\|\nabla \vphi_{2\lam}(x)\| \le \veps$ for some~$x \in X$ implies the existence of a point~$x^+ = x^+(x)$ within~$O(\veps)$ distance from~$x$, such that~$\vphi$ has a subgradient at~$x^+$ with the norm at most~$\veps$. 
A similar result holds in the constrained case if the subgradient norm of~$\vphi$ is replaced with an appropriate inaccuracy measure based on projection onto~$X$; see~\cite[Prop.~5.1]{ostrovskii2020efficient} also given as~\cref{prop:moreau-to-primal} in our paper. 
These results motivate us to introduce the following definition.

\begin{definition}[Approximate first-order stationary points]
\label{def:accuracy}
The point~$x \in X$ is called~{\em $(\veps,\Lxx')$-first-order stationary ($(\veps,\Lxx')$-FOSP)} for~\cref{opt:min-max} 
if~$\|\nabla \vphi_{\lam'} (x) \| \le \veps$.
\end{definition}

\odima{Note that an~$(\veps,\lambda')$-FOSP is guaranteed to exist in~\eqref{opt:min-max} for any~$\veps > 0$ and~$\lambda' > \lambda$. 
Quantifying first-order stationarity in~\eqref{opt:min-max} in terms of the Moreau-envelope criterion is quite natural. 
Indeed: this criterion, intially proposed in the context of weakly convex optimization~\cite{davis2018stochastic,davis2019stochastic}, has been subsequently adopted in nonconvex-concave min-max problems~\cite{thekumparampil2019efficient,ostrovskii2020efficient,lin2020near}, a subclass of~\eqref{opt:min-max} that already lacks strong duality. 
Historically, the key motivation for its use in the nonconvex-concave context is the above-mentioned guarantee that, granted Assumption~\ref{ass:gradx}, any~$(\veps,2\lambda)$-stationary point is~$O(\veps)$-close to an~$\veps$-stationary point of~$\varphi(\cdot)$, but in reality this guarantee holds due to the weak convexity on~$\varphi(\cdot)$, and thus carries out to the nonconvex-concave setup.
On the other hand, while in the nonconvex-concave case one can convert an~$(\veps,2\lambda)$-FOSP into an~$(\ex,\ey)$-first-order Nash equilibrium (NE)~(informally, a point satisfying the two inequalities in~\eqref{eq:first-order-nash} approximately; see~\cite[Def.~2.1]{ostrovskii2020efficient}), and in this sense the two notions are equivalent, this is not the case anymore when~$f(x,\cdot)$ is nonconcave: as discussed in Section~\ref{sec:intro}, in some instances of~\eqref{opt:min-max} even {\em exact} first-order NE have no relation to even {\em local} saddle points (which might not even exist), nor to the critical points of~$\varphi(\cdot)$.}

\vspace{-0.4cm}
\paragraph{Taylor expansion and the surrogate problem.}
Let us formally define the Taylor expansion of~$f(x,\cdot)$ to be used from now on. %
Recall that our approach relies on replacing~$f(x,\cdot)$ with its~$k$-order Taylor approximation~$\fapx_k(x,\cdot)$ for a non-negative integer~$k$ and a fixed ``center'' point~$\by \in Y$:
\begin{equation}
\label{def:fk}
\fapx_k(x,y) = \sum_{j = 0}^k \frac{1}{j!}\Ty{j} f(x,\by) \, [(y-\by)^j].
\tag{$\mathsf{TE}$}
\end{equation}
Here~$T[y^j]$ is the diagonal evaluation of tensor~$T : \YY^{\otimes j} \to \R$ on~$y \in \YY$---that is,~$T[y^j] := T[y,...,y]$.
In particular, in the most practically important cases~$k \in \{0, 1, 2\}$,~\cref{def:fk} reduces to~$\fapx_{0}(x,y) = f(x,\by)$,
\begin{equation}
\label{eq:0-1-models}
\begin{aligned}
\fapx_{1}(x,y) = f(x,\by) + \lang \gy f(x,\by), y-\by \rang, \quad\quad 
\fapx_{2}(x,y) = \fapx_1(x,y) + \tfrac{1}{2} \langle y-\by, \hyy f(x,\by) (y-\by)\rangle.
\end{aligned}
\end{equation}
Note that for~$k \le 2$ we can efficiently maximize~$\fapx_k(x,\cdot)$: indeed, ~$\fapx_{0}(x,\cdot) \equiv f(x,\by)$ is constant,~$\fapx_{1}(x,\cdot)$ is affine, and~$\fapx_{k}(x,\cdot)$ for~$k = 2$ is a (nonconcave) quadratic that can be globally maximized using the gradient oracle~$\gy \fapx_2(x,\cdot)$ when~$Y$ is a ball~\cite{carmon2020first}. 
From now on, we focus on the {\em surrogate problem}
\begin{equation}
\label{opt:min-max-apx}
\min_{x \in X} \max_{y \in Y} \fapx_k (x,y)
\tag{${\mathsf{P}}_k$}
\end{equation}
and denote with~$\vphiapx(x) := \max_{y \in Y} \fapx_k(x,y)$ the corresponding \odima{max-function}, with~$k$ omitted for brevity.
%
%
%
%

%
\section{Upper bounds on the admissible diameter}
\label{sec:upper}

We are now in the position to rigorously formulate the question we first posed in the introduction:
\begin{framed}
\vspace{-0.4cm}
\begin{quote}
{\em For~$k \ge 0$, what~$\Dy$ allows to guarantee the existence of~$\Lxxp \lsim \Lxx$ and~$c > 0$ such that any~$(c\veps,2\Lxxp)$-FOSP for~\eqref{opt:min-max-apx}, %
regardless of the choice of~$\by$, is~$(\veps,2\Lxxp)$-FOSP for~\eqref{opt:min-max}?} %
\end{quote}
\vspace{-0.4cm}
\end{framed}
In this section, our goal is to answer this question, and such an answer will be given in~\cref{th:upper-bound}.
But before, let us clarify a subtle point about it: namely, note that we should expect~$c < 1$ and~$\Lxxp > \Lxx$.
Indeed, replacing~$f$ with its Taylor approximation is likely to cause some deterioration of accuracy as measured by the gradient norm of the Moreau envelope, so we cannot expect, say,~$(2\veps,2\Lxxp)$-FOSP for~\eqref{opt:min-max-apx} to also be an~$(\veps,2\Lxxp)$-FOSP for~\eqref{opt:min-max}. %
Similarly, considering~$\Lxxp = \Lxx$ would imply that the surrogate \odima{max-function}~$\vphiapx(\cdot)$ is~$\Lxx$-weakly convex (cf.~\cref{def:moreau}), but in fact such a guarantee is not available. 
Fortunately, it is not hard to prove (cf.~\cref{lem:gx-lip} below) that under the regularity assumptions granted in~\cref{sec:problem} and a weaker bound on~$\Dy$ than the one imposed in~\cref{th:upper-bound},
the function $\vphiapx(\cdot)$ is~$\Lxxp$-weakly convex with~$\Lxxp = (1+o(\veps))\Lxx$.

%

%
%

%

%
%
%
%
%
%
%
%

%
%

%
To prepare the ground for proving Theorem~\ref{th:upper-bound}, we first obtain the bounds for approximating~$f$ with~$\fapx_k$ uniformly over~$(x,y) \in X \times Y$ in terms of the function value, the~$x$-gradient, and~$x$-Hessian. 
\begin{lemma}
\label{lem:fval-err}
Grant~\cref{eq:tensy-lip} with~some~$k \ge 0$ and possibly~$\cLykx = \infty$.
Then for any~$x \in X,$~$y \in Y$ one has
\[
| f(x,y) - \fapx_k(x,y) |
\le
\frac{\cLyky \Dy^{k+1}}{(k+1)!}.
\]
\end{lemma}

\begin{lemma}
\label{lem:gx-err}
Let~$k \ge 0$ be given. Grant~\cref{eq:gradx-lip} if~$k = 0$ (possibly with~$\Lxx = \infty$) and grant~\cref{eq:tensy-lip} (possibly with~$\cLyky = \infty$, but with~$\Ty{k}f(\cdot,y)$ absolutely continuous~$\forall y \in Y$). 
\mbox{Then for any~$x \in X,$ $y \in Y$ one has}
\[
\| \gx f(x,y) - \gx \fapx_k(x,y) \| 
\;\le\;  \odima{\ind_{k = 0} \min\{\Lxy \Dy, 2\sigma_0\}  \; + \; \ind_{k > 0} \frac{2\cLykx \Dy^k}{k!}.}
\]
\end{lemma}

\odima{
\begin{lemma}
\label{lem:gx-lip}
Given~$k \ge 0$, grant~\cref{ass:gradx,ass:tensy,ass:measurable}.
Then~$\gx \fapx_k(\cdot, y)$ is~$\Lxxp_k$-Lipschitz,~$\forall y \in Y$, with
\begin{equation}
\label{eq:gx-lip}
\Lxxp_k
:= \Lxx +  \ind_{k > 0} \frac{2\cLykxx \Dy^k}{k!}.
\end{equation}
Moreover, the variation of~$\gx \fapx_k(x,\cdot)$ over~$Y$ satisfies
$
\displaystyle\sup_{x \in X; \; y,y' \in Y} \| \gx \fapx_k(x,y') - \gx \fapx_k(x,y) \| \le \ind_{k > 0} \, \Lxyp_k \Dy,
$\vspace{-0.2cm}
\begin{equation}
\label{eq:gx-vary}
\quad\Lxyp_k := \Lxy +  \ind_{k > 0} \frac{2\cLykx \Dy^{k-1}}{k!}.
\end{equation}
\end{lemma}}
Note that an immediate corollary of~\cref{eq:gx-lip} is that~$\vphiapx$ is~$\Lxxp_k$-weakly convex.
\cref{lem:fval-err,lem:gx-err,lem:gx-lip} are proved by integrating the remainder term of the Taylor expansion; however, in the case of~\cref{lem:gx-lip} some technicalities arise; while they could be easily resolved by imposing an extra order of regularity (namely by requiring that~$\Txxy{k+2}{k} f$ exists everywhere), 
we manage to avoid this condition \odima{by carefully applying}~\cref{ass:measurable}. 

\vspace{0.1cm}
Next we present our first main result: a general upper bound on the admissible diameter of~$Y$---i.e., the one allowing to replace~\cref{opt:min-max} with~\cref{opt:min-max-apx}  when searching for FOSPs. 

\begin{theorem}
\label{th:upper-bound} 
Given~$k \ge 0$, grant~\cref{ass:gradx,ass:tensy,ass:measurable}, and let~$x^*$ be~$(\tfrac{1}{6}\veps,2\Lxxp_k)$-FOSP for~\eqref{opt:min-max-apx}, with~$\Lxxp_k,\Lxyp_k$ as in~\eqref{eq:gx-lip}.
Then~$x^*$ is~$(\veps,2\Lxxp_k)$-FOSP---hence also~$(\veps,2\Lxx)$-FOSP---for~\eqref{opt:min-max} as long as
\odima{
\begin{equation}
\label{eq:upper-bound-k}
\min\left\{\,
\left(\frac{1}{50\,(k+1)!} \Lxxp_k \rho_k \Dy^{k+1} \right)^{1/2}, \;
\Lxyp_k \Dy %
\right\}
\le \frac{\veps}{24}.
\end{equation}
Moreover, in the case of~$k = 0$, the minimum can be complemented with the third term---namely,~$2\sigma_0$.}
\end{theorem}
%
%
%
\noindent We shall present the proof of \odima{this result} in~\cref{sec:upper-proof} and \odima{establish} its tightness in~\cref{sec:lower}.
But before doing all this, let us discuss the implications of this result. %

\subsection{Discussion of~\cref{th:upper-bound}}
\paragraph{\odima{Zero-order approximation~($k = 0$).}}
\odima {As can be seen from Theorem~\ref{th:upper-bound}, the case~$k = 0$ is somewhat special.
Recall that in this case, the~``\eqref{opt:min-max-apx}-to-\eqref{opt:min-max} reduction'' is valid provided that
\begin{equation}
\label{eq:sigma-0}
\min\left\{\sigma_0, \, \sqrt{\Lxx\rho_0 \Dy}, \, \Lxy \Dy\right\} \lsim \veps.
\end{equation}
When~$\sigma_0 \lsim \veps$, this condition satisfied for {\em any}~$\Dy$, and without all other Lipschitzness conditions in~\cref{ass:gradx,ass:tensy} (i.e., when~$\Lxx,\rho_0,\Lxy$ are infinite); in this ``pathological'' case we have that any~$x \in X$ is an~$O(\veps)$-FOSP.
Another interesting scenario is when~$f(x,\cdot)$ is Lipschitz-continuous (i.e.,~$\rho_1 < \infty$) and has a stationary point in the {\em interior of $Y$} for each~$x \in X$. 
From the variational bound~$\rho_0 \le 2\rho_1 \Dy$ it immediately follows that, in this case,~\cref{eq:upper-bound-k} with~$k = 0$ is {\em milder} than with~$k = 1$.
Note that this phenomenon is specific to~$0^{\textup{th}}$- vs.~$1^{\textup{st}}$-order approximation: if~$f^{(m)}(x,\cdot)$, for some~$m > 1$, is~$\rho_{m}$-Lipschitz and vanishes at some~$y \in Y$ for each~$x \in X$, then~$\rho_{m-1} \le 2\rho_{m} \Dy$; however, plugging this into~\cref{eq:upper-bound-k} with~$k = m$ we lose the factor~$\frac{(m+1)!}{m!} = m+1 \;[> 2]$ as compared to~\cref{eq:upper-bound-k} with~$k = m-1$.
}

%
%

%
%
%
%
%
%
%
%
%
%
%
%
%

\paragraph{\odima{Leading-term picture: validity.}}
\odima{
When~$k > 0$, the first term in~\cref{eq:upper-bound-k} depends on~$\Dy$ as~$\Dy^{\frac{k+1}{2}}$, modulo a higher-degree additive term (stemming from~$\Lxxp_k$) that can be neglected if~$\Dy$ is not too large (if~$\Dy \lsim k (\Lxx/\tau_k)^{1/k}$, so that~$\Lxxp_k \lsim \Lxx$ by Stirling's formula).
Likewise, the second term in~\eqref{eq:upper-bound-k} can be replaced with~$\Lxy \Dy$ if~$\Dy \lsim k (\Lxy/\sigma_k)^{1/(k-1)}$. 
Under these mild restrictions, condition~\eqref{eq:upper-bound-k} follows from}
\odima{
\begin{equation}
\label{eq:acc-bilinear}
\min\left\{\sqrt{\frac{1}{(k+1)!}\Lxx \rho_k \Dy^{k+1}}, \Lxy\Dy \right\} \lsim \veps, 
\quad
\text{\odima{that is}}
\quad
\Dy \lsim 
\max \left\{ (k+1) \left( \frac{\veps^{2}}{\Lxx  \rho_k} \right)^{\frac{1}{k+1}}, \frac{\veps}{\Lxy} \right\}.
\end{equation}
Note that the simplified condition~\eqref{eq:acc-bilinear} does not depend on the higher-order regularity parameters~$\sigma_k,\tau_k$. Moreover, the approximation of~\eqref{eq:upper-bound-k} with~\eqref{eq:acc-bilinear} is {\em exact} for BC objectives (cf.~\eqref{eq:bilinear-coupling}), in which case~$\tau_k = 0$, and~$\sigma_k = 0$ if~$k = 1$. 
In fact, in~\cref{sec:lower} we shall demonstrate that for such problems,~\cref{eq:acc-bilinear} is {\em tight} by constructing specific problem instances on which it is attained. Finally, note that even beyond the BC subclass, the condition~$\Dy \lsim k \min\{(\Lxx/\tau_k)^{1/k},  (\Lxy/\sigma_k)^{1/(k-1)} \}$ is enforced automatically under~\eqref{eq:acc-bilinear} as long as the {\em accuracy~$\veps$} is smaller than a constant threshold: as we verify in Appendix~\ref{app:aux},
\begin{equation}
\label{eq:eps-for-lower-order}
\veps \lsim_k 
\min \left\{ \left({\mu^{k}}/{\sigma_k}\right)^{\frac{1}{k-1}}, \left({\lam^{2k+1} \rho_k^{k}}/{\tau_k^{k+1}} \right)^{\frac{1}{2k}}\right\}
\end{equation}
suffices, where the~$\lsim_k$ notation hides a~$k$-dependent constant. 
Clearly, this is a very mild condition.}
\paragraph{\odima{Implications for~$k > 1$: coupling-independent behavior for small~$\veps$.}}
Neglecting the dependencies on~$\Lxx,\Lxy$ and~$\rho_k$, the second form of~\eqref{eq:acc-bilinear} amounts to~$\Dy = O(\veps)$ when~$k = 1$.
Meanwhile, when~$k > 1$, the coupling-independent threshold for~$\Dy$ shrinks with~$\veps \to 0$ as~$O(\veps^{\frac{2}{k+1}})$, i.e., slower than the coupling-dependent threshold~$O(\veps)$, and the former dominates the later when
\begin{equation}
\label{eq:high-accuracy}
\veps \lsim_k \left(\frac{\Lxy^{k+1}}{\Lxx  \rho_k} \right)^{\frac{1}{k - 1}}. 
\end{equation}
Thus, increasing the approximation order above~$k  = 1$ allows to gain in terms of the range of~$\Dy$ for which~\cref{th:upper-bound} applies, and~\cref{opt:min-max} can be replaced with~\cref{opt:min-max-apx} when searching for FOSPs. Moreover, in this ``high-accuracy'' regime the critical~$\Dy$ becomes {\em coupling-independent} for~$k > 1$, defined solely by the ``homogeneous'' parameters~$\lam, \rho_k$ and the target~$\veps$; meanwhile,~$\mu$ defines the moment of transition to this regime from the initial~$\Dy = O(\veps)$ one, as~$\veps$  is driven below the threshold in~\cref{eq:high-accuracy}.

It is worth noting that, granted sufficient regularity of~$f$ and assuming constant ($\veps$-independent) regularity parameters, the choice~$k = 2\log(1/\veps)-1$ gives {\em constant} critical diameter~$\Dy$. However, as discussed in Section~\ref{sec:algos}, the surrogate problem~\cref{opt:min-max-apx} becomes computationally intractable when $k \ge 3$.

\paragraph{Implications for the cases~$k=1$ and~$k = 0$.}
In the case of~$k = 1$, the previously discussed ``elbow effect'' (cf.~\eqref{eq:high-accuracy}) disappears. Here the critical diameter is~$\Dy = O(\veps)$ over the {\em whole range of~$\veps$,} with~$\mu$ appearing in the hidden constant factor~$1/\min\{\mu, \sqrt{\lam \rho_1}\}$. 
Thus, here we benefit from low interaction levels ($\mu \lsim \sqrt{\lam \rho_1}$) while not suffering from higher ones ($\mu \gsim \sqrt{\lam \rho_1}$), regardless of~$\veps > 0$.
Of course, beyond the BC class interaction does manifest in higher order, via~$\tau_1$ in~\cref{eq:eps-for-lower-order}.

Finally, in the case of~$k = 0$ the aforementioned ``elbow'' is ``in reverse\odima{:}'' we start with~$\Dy = O(\veps^{1/2})$ for large~$\veps$, and switch {to}~$\Dy = O(\veps)$ critical diameter as~$\veps$ passes the threshold~$\lam\rho_0/{\mu}$ which corresponds to~\cref{eq:high-accuracy} with~$k = 0$. 
In the scenario where~\cref{ass:tensy} holds simultaneously for~$k \in \{0,1\}$ and we can choose between approximations with these orders,
the gain for~$k = 1$ is marginal, only in the constant factor: namely,~$1/\min\{\mu,\sqrt{\lam \rho_1}\}$ instead of~$1/\mu$---and this effect only manifests on high interaction levels:~$\Lxy^2 \gsim \Lxx \rho_1.$
In fact, even this marginal comparative disadvantage of zeroth-order approximation disappears in the ``$y$ unconstrained as per FOSP'' scenario, where~$f(x,\cdot)$ has a stationary point inside~$Y$ for any~$x \in X$. %

%

%
%
%

%
%

\subsection{Proof of Theorem~\ref{th:upper-bound}}
\label{sec:upper-proof} 

The result follows by combining~\cref{prop:upper-uniform,prop:upper-coupled} which we formulate and prove next. 
These propositions correspond to the two choices for the minimum in~\cref{eq:upper-bound-k}, and we prove each of them under minimal assumptions; the full~\cref{ass:gradx,ass:tensy} are required to have both results simultaneously.

\begin{proposition}
\label{prop:upper-uniform}
For~$k \ge 0$ and~$\Lxxp_k$ given by~\cref{eq:gx-lip}, under the premise of Lemmas~\ref{lem:fval-err}--\ref{lem:gx-lip} one has
\[
\|\nabla \vphiapx_{2\Lxxp_k}(x) - \nabla \vphi_{2\Lxxp_k}(x)\| \le 
\sqrt{\tfrac{8}{(k+1)!} \Lxxp_k \rho_k \Dy^{k+1}}, \quad \forall x \in X.
\]
\end{proposition}
\begin{proof}
Clearly,~$\vphi(\cdot)$ is~$\Lxx$-weakly convex (cf.~\cref{ass:gradx}), hence also~$\Lxxp_k$-weakly convex.
Moreover, by~\cref{lem:gx-lip} the function~$\vphiapx(\cdot) = \max_{y \in Y} \fapx_k(\cdot,y)$ is also~$\Lxxp_k$-weakly convex.
Whence by~\cref{eq:moreau-explicit} we have %
\[
\nabla \vphi_{2\Lxxp_k}(x) = 2\Lxxp_k(x - x^+), 
\quad
\nabla \vphiapx_{2\Lxxp_k}(x) = 2\Lxxp_k(x - \hat x^+),
\] 
with~$x^+, \hat x^+$ being the associated proximal-point operators:
\[
\begin{aligned}
x^+ 
&=  \argmin_{x' \in X}  \{\vphi(x') + \Lxxp_k \|x'-x\|^2\},
\quad
\hat x^+ 
= \argmin_{x' \in X}  \{\vphiapx(x') + \Lxxp_k \|x'-x\|^2\}.
\end{aligned}
\]
As a result,~$\| \nabla \vphi_{2\Lxxp_k}(x) - \nabla \vphiapx_{2\Lxxp_k}(x)\| = 2\Lxxp_k \|\hat x^+ - x^+\|$, and we can focus on bounding~$\|\hat x^+ - x^+\|$. 
To this end, since the function~$\vphi(\cdot) + \Lxxp_k \|\cdot-x\|^2$ is~$\Lxxp_k$-strongly convex and minimized at~$x^+$, we have that
\begin{equation}
\label{eq:strong-cvxty-ineq}
\tfrac{1}{2} \Lxxp_k \|\hat x^+ - x^+ \|^2 \le \vphi(\hat x^+) + \Lxxp_k \|\hat x^+ - x\|^2 - \vphi(x^+) - \Lxxp_k \|x^+ - x\|^2.
\end{equation}
On the other hand, since the function~$\vphiapx(\cdot) + \Lxxp \|\cdot-x\|$ is~$\Lxxp$-strongly convex and minimized at~$\hat x^+$,
\[
\tfrac{1}{2} \Lxxp_k \|\hat x^+ - x^+ \|^2 \le \vphiapx(x^+) + \Lxxp_k \|x^+ - x\|^2 - \vphiapx(\hat x^+) - \Lxxp_k \|\hat x^+ - x\|^2.
\]
Adding the two inequalities gives
$
\Lxxp_k \|\hat x^+ - x^+ \|^2 \le \vphiapx(x^+) - \vphi(x^+) + \vphi(\hat x^+) -  \vphiapx(\hat x^+) \le 2\sup_{x \in X} |\vphiapx(x) - \vphi(x)|.
$
Finally, let~$x \in X$ be arbitrary and~$\hat y^* = \hat y^*(x)$ be such that~$\vphiapx(x) = \fapx(x,\hat y^*)$. Then~\cref{lem:fval-err} gives
\[
\vphiapx(x) - \vphi(x) 
\le \fapx_k(x, \hat y^*) - f(x, \hat y^*) 
\le \tfrac{1}{(k+1)!} \cLyky \Dy^{k+1}, \quad \forall x \in X.
\]
The same estimate holds for~$\vphi(x) - \vphiapx(x)$ which is bounded from above by~$f(x, y^*) - \fapx_k(x, y^*)$ with~$y^* = y^*(x)$ such that~$\vphi(x) = f(x,y^*)$), 
and thus for~$\sup_{x \in X} |\vphiapx(x) - \vphi(x)|$. The result follows.
\end{proof}

\begin{proposition}
\label{prop:upper-coupled}
Let~$k \ge 0$, \odima{and~$\Lxxp_k, \Lxyp_k$ be as in~\cref{eq:gx-lip}--\cref{eq:gx-vary}.}
Grant~\cref{ass:gradx} and the assumptions of~\cref{lem:gx-err,lem:gx-lip}. 
Then for any~$x^* \in X$ such that~$\|\nabla \vphiapx_{2\Lxxp_k}(x^*)\| \le \veps$, \odima{one has the following:}
\[
\odima{\tfrac{1}{4}\|\nabla \vphi_{2\Lxxp_k}(x^*) - \nabla \vphiapx_{2\Lxxp_k}(x^*)\|
\;\le\; \veps \;+\; \ind_{k > 0} \, \Lxyp_k \Dy \;+\; \ind_{k=0}\, \min\{\Lxy \Dy, 2\sigma_0\}.}
\]
\end{proposition}

\begin{proof}
First observe that~$\vphi(\cdot)$ and~$\vphiapx(\cdot)$ are~$\Lxxp_k$-weakly convex by~\cref{lem:gx-lip}. %
Hence,~$\vphi(\cdot) + \Lxxp_k\|\cdot-x^*\|^2$ and~$\vphiapx(\cdot) + \Lxxp_k \|\cdot-x^*\|^2$ are~$\Lxxp$-strongly convex, and
their corresponding minimizers~$x^+, \hat x^+$ satisfy (cf.~\cref{eq:strong-cvxty-ineq})
\begin{align}
\tfrac{1}{2} \Lxxp_k \|\hat x^+ - x^+ \|^2 
&\le  \Lxxp_k \|\hat x^+ - x^*\|^2 + \vphi(\hat x^+) - \vphi(x^+) - \Lxxp \|x^+ - x^*\|^2 \nn
&\le 4\Lxxp_k \|\hat x^+ - x^*\|^2 + \vphi(\hat x^+) - \vphi(x^+) - \tfrac{3}{4} \Lxxp_k \|\hat x^+ - x^+\|^2.
\label{eq:strong-cvxty-ineq-xhat}
\end{align}
Here in the final step we used the inequality
$
\|\hat x^+ - x^+\|^2 \le \tfrac{4}{3}\|x^+ - x^*\|^2 + 4\|\hat x^+ - x^*\|^2
$ 
which can be easily deduced from the triangle inequality.
Furthermore, by~\cref{prop:moreau-to-primal} we have
\[
\nabla \vphi_{2\Lxxp_k}(x^*) = 2\Lxxp_k(x^* - x^+), 
\quad
\nabla \vphiapx_{2\Lxxp_k}(x^*) = 2\Lxxp_k(x^* - \hat x^+).
\] 
Hence~$\| \nabla \vphiapx_{2\Lxxp_k}(x^*) - \nabla \vphi_{2\Lxxp_k}(x^*) \| = 2\Lxxp_k \|\hat x^+ - x^+\|$ so we can focus on bounding~$\|\hat x^+ - x^+\|$ using~\cref{eq:strong-cvxty-ineq-xhat}.
Now observe that~$x^*$, being an~$(\veps,2\Lxxp_k)$-FOSP for~\cref{opt:min-max-apx} with a~$\Lxxp_k$-weakly convex \odima{max-function}~$\vphiapx(\cdot)$, satisfies the premise of Proposition~\ref{prop:moreau-to-primal}, so that
\begin{equation}
\label{eq:moreau-charact-apx}
2 \Lxxp_k \|\hat x^+ - x^*\| \le \veps, 
\quad 
\min_{\xi \in \partial \vphiapx(\hat x^+)} \Sx(\hat x^+,\xi,2\Lxxp_k) \le \veps,
\end{equation}
cf.~\cref{eq:moreau-to-primal}, with functional~$\Sx(x, \xi, \lam')$ defined by 
$
\Sx^2(x, \xi, \lam') := 2\lam' \max_{u \in X} \big\{-\langle \xi, u - x \rangle - \tfrac{\lam'}{2} \|u-x\|^2 \big\}
$
for arbitrary~$x \in X$,~$\xi \in \XX,$ and~$\lam' > 0$.
By~\cref{eq:strong-cvxty-ineq-xhat} and the first bound in~\cref{eq:moreau-charact-apx} we get 
\[
\tfrac{1}{2} (\Lxxp_k \|\hat x^+ - x^+ \|)^2 
\le \veps^2 + \Lxxp_k \left( \vphi(\hat x^+) - \vphi(x^+) - \tfrac{3}{4} \Lxxp_k \|x^+ - \hat x^+\|^2 \right).
\]
Meanwhile, convexity of~$\vphi(\cdot) + \tfrac{1}{2}\Lxxp_k \|\cdot - \hat x^+\|^2$ implies that
$
\vphi(\hat x^+) - \vphi(x^+) - \tfrac{\Lxxp_k}{2} \|\hat x^+ - x^+\|^2 
\le \lang \xi^+, \hat x^+ - x^+  \rang
$
for any~$\xi^+ \in \partial \vphi(\hat x^+)$. 
Using this fact and letting~$\hat\xi_{X}^+ \in \Argmin_{\xi \in \partial \vphiapx(\hat x^+)} \Sx(\hat x^+,\xi,2\Lxxp_k)$, cf.~\cref{eq:moreau-charact-apx}, we get
\begin{align}
( \Lxxp_k \|\hat x^+ - x^+ \| )^2 
&\le 2\veps^2 + 2\Lxxp_k \left( -\lang \xi^+, x^+ - \hat x^+   \rang - \tfrac{1}{4} \Lxxp_k  \|\hat x^+ - x^+\|^2 \right) \nn
&\le 2\veps^2 + 2\Lxxp_k \left( -\lang \hat\xi_X^+, x^+ - \hat x^+   \rang - \tfrac{1}{4} \Lxxp_k \|\hat x^+ - x^+\|^2  + \|\hat x^+ - x^+\| \cdot \|\hat \xi_X^+ - \xi^+\| \right) \nn
&\le 2\veps^2 + 2\Sx^2(\hat x^+,\hat\xi_X^+,\tfrac{1}{2}{\Lxxp_k}) + 2 \Lxxp_k \|\hat x^+ - x^+\| \cdot \|\hat \xi_X^+ - \xi^+\|.
\label{eq:grad-diff-to-estimate}
\end{align}
We furthermore have~$\Sx(\hat x^+, \hat \xi_X^+, \tfrac{1}{2}\Lxxp_k) \le \veps$ due to~\cref{eq:moreau-charact-apx} and the well-known fact that~$\Sx(x, \xi,\Lxxp)$ is non-decreasing in~$\Lxxp > 0$. (This monotonicity property follows from the proximal Polyak-\L{}ojasiewicz inequality---see, e.g.,~\cite[Lem.~1]{karimi2016linear}.) 
Meanwhile, by a version of Danskin's theorem (cf.~\cref{lem:danskin} in the appendix),~$\hat \xi_X^+$ belongs to the closed convex hull of the set of active~$x$-gradients of~$\fapx_k$ at~$\hat x^+$:
\[
\hat\xi_X^+ \in \clconv\big(\big\{\gx \fapx_k(\hat x^+, y), \; y\in\textstyle\Argmax_{y \in Y} \fapx_k(\hat x^+, y) \big\}\big).
\]
Similarly, we can choose~$\xi^+ = \gx f(\hat x^+, y^*)$ for~$y^* \in \Argmax_{y \in Y} f(\hat x^+, y)$, 
whence by convexity of~$\|\cdot\|$:
\[
\|\hat \xi_X^+ - \xi^+\| \le \max\displaystyle_{y \in Y} \| \gx \fapx_k(\hat x^+, y) - \gx f(\hat x^+, y^*) \|.
\]
Thus, when~$k = 0$ we have~$\|\hat \xi_X^+ - \xi^+\| \le \min\{\Lxy\Dy, 2\sigma_0\}$ by~\cref{lem:gx-err}.
On the other hand, when~$k \ge 1$ we pick~$\bar y \in \Argmax_{y \in Y} \| \gx \fapx_k(\hat x^+, y) - \gx f(\hat x^+, y^*) \|$, then by~\cref{lem:gx-err} and the triangle inequality
\[
\begin{aligned}
\|\hat \xi_X^+ - \xi^+\| 
&\le \| \gx f(\hat x^+, \bar y) - \gx f(\hat x^+, y^*)\| + \| \gx \fapx_k(\hat x^+, \bar y) - \gx f(\hat x^+, \bar y)\| 
\le \odima{\Lxy \Dy + \tfrac{2}{k!} \cLykx \Dy^k = \Lxyp_k \Dy}. 
\end{aligned}
\]
Finally, plugging the obtained estimates for~$\Sx(\hat x^+, \hat \xi_X^+,\tfrac{1}{2}\Lxxp_k)$ and~$\|\hat \xi_X^+ - \xi^+\|$ into~\cref{eq:grad-diff-to-estimate} we have that
\[
(\Lxxp_k \|\hat x^+ - x^+ \|)^2 
\le 4\veps^2 + 2 \Lxxp_k \|\hat x^+ - x^+\|  \left( \ind_{k = 0} \, \min\{\Lxy\Dy, 2\sigma_0\} \,+\, \ind_{k > 0} \, \Lxyp_k \Dy \right)
\]
Solving this inequality for~$\Lxxp_k \|\hat x^+ - x^+\| = \tfrac{1}{2}\| \nabla \vphiapx_{2\Lxxp_k}(x^*) - \nabla \vphi_{2\Lxxp_k}(x^*) \|$ we conclude the proof.
\end{proof}

\section{Nearly matching lower bounds on the admissible diameter}
\label{sec:lower}

Our next goal is to prove that conditions
in Theorem~\ref{th:upper-bound} are tight in leading-order terms.
Namely, for any~$k \in \N \cup \{0\}$ we exhibit instances of~\eqref{opt:min-max} such that the corresponding approximate problem~\cref{opt:min-max-apx} has an {\em exact} FOSP which is {\em not}~$(\veps,2\lam)$-FOSP in~\cref{opt:min-max} for any accuracy~$\veps$ in the range
\begin{equation}
\label{eq:lower-bound-k}
\veps \lsim
\frac{1}{k+1}
\min\bigg\{ 
\Lxy \Dy,  %
\sqrt{\frac{\Lxx \cLyky \Dy^{k+1}}{k!}} 
\bigg\}.
\end{equation}
(For simplicity, in this section we let~$\Dy = \diam(Y)$; this is anyway the case in all instances to be exhibited.)
The objectives in these instances satisfy~\cref{ass:gradx},~\cref{ass:tensy} with appropriate~$k$, and~\cref{ass:measurable}.
Moreover, for~$k \ge 1$ we use bilinearly-coupled objectives (cf.~\cref{eq:bilinear-coupling}), so~\cref{ass:tensy} is satisfied with~$\tau_k = 0$ (hence,~$\Lxxp_k$ in~\cref{eq:upper-bound-k} simplifies to~$\Lxx$, cf.~\cref{eq:gx-lip}) and~$\sigma_k = \Lxy\ind\{k= 0\}$ (hence, the additive to~$\Lxy \Dy$ term in~\cref{eq:upper-bound-k} disappears). 
Thus, for~$k \ge 1$ our admissibility bound~\cref{eq:lower-bound-k} is tight over the BC subclass up to an~$O(1/k)$ factor.
For~$k = 0$, the bound~\cref{eq:lower-bound-k} only misses the~$\Dy$-independent term~$\sigma_0$ term compared to~\cref{eq:sigma-0}, and this distinction disappears if the regularity parameters~$\lam,\mu,\rho_0,\sigma_0$ are fixed and~$\Dy$ decreases with~$\veps$. 
In this sense, Theorem~\ref{th:upper-bound} is nearly tight.

Let us now specify our problem instances. Consider a family of functions on~$\R \times \R$ in the form
\begin{equation}
\label{eq:hard-quadratic}
F_{k,s,\lam,\mu,\rho}(x,y) := -\frac{\lam x^2}{2} + \mu xy + \frac{s\rho|y|^{k+1}}{(k+1)!}
\quad \text{for} \;\;\; k \in \N \cup \{0\}, \;\, s \in \{\pm 1, 0\}, \;\; \lam > 0, \;\; \mu \ge 0, \;\; \rho \ge 0.
\end{equation}
Note that~$F_{k,s,\lam,\mu,\rho}$ is BC, concave in~$x$, and convex or concave in~$y$ depending on~$s$.
We also consider
\begin{equation}
\label{eq:hard-sigmoid}
S_{\lam,\rho,\Dy}(x,y) := -\frac{\lam x^2}{4} + \frac{\rho y}{2} \left( \tanh\left(\sqrt{\frac{\Lxx}{\rho \Dy}} \, x\right) - 1 \right)
\quad \text{for} \;\;\; \lam > 0, \;\; \rho \ge 0, \;\; \Dy \ge 0.
\end{equation}
Each of our hard instances of the form~\cref{opt:min-max} is specified by choosing one of these functions as the objective,~$X \subseteq \R$, and~$Y = [a,a+\Dy]$ with some shift~$a \in \R$, while allowing for varying~$\Lxx,\Lxy,\rho,$ and~$\Dy$. %
In the next lemma (proved in~\cref{app:aux}) we establish the smoothness properties of these functions. %

\begin{lemma}
\label{lem:hard-properties}
For any~$k \in \N \cup \{0\}$,~$s = \pm 1$,~$\lam > 0$,~$\mu \ge 0,$~$\rho \in \R$, and~$\Dy \ge 0$, the following claims hold:
\begin{enumerate}
\item
Function $F_{k,s,\lam,\mu,\rho}$ satisfies~\cref{ass:gradx} on~$\R \times \R$ (and thus also on~$\R \times [a,a+\Dy]$ for any~$a \in \R$). %
Moreover, function~$S_{\lam,\rho,\Dy}$
satisfies~\cref{ass:gradx} on~$\R \times [0,\Dy]$ %
provided that~$\mu \ge \sqrt{2\lam \rho/\Dy}$. 
\item 
Let~$r = \frac{\mu \Dy}{2\lam}$. 
\cref{ass:tensy} with~$k = 0$,~$\rho_0 = \rho$, and~$\sigma_0 = \mu\Dy$ is satisfied by function~$F_{0,0,\lam,\mu,0}$ on~$[-r,r] \times [-\tfrac{1}{2}\Dy, \tfrac{1}{2} \Dy]$ if~$\mu \le \sqrt{2\lam \rho/\Dy}$, and by function~$S_{\lam,\rho,\Dy}$ on~$[-r,r] \times [0,\Dy]$ if~$\mu \ge \sqrt{2\lam \rho/\Dy}$. 
\item 
When~$k \ge 1$ and~$a \in \R$, function $F_{k,s,\lam,\mu,\rho}$ on~$\R \times [a, a+\Dy]$ satisfies~\cref{ass:tensy} with~$\cLyky = \rho$, $\cLykx = \mu \ind\{k = 1\}$, and~$\cLykxx = 0$---in other words,~\cref{ass:tensy} restricted to the objective class~\cref{eq:bilinear-coupling}.
\end{enumerate}
\end{lemma}
%
%

Due to significant differences in the statements and analyses, we separately consider the three cases~$k = 0,$ $k= 1,$ and~$k \ge 2$.
We begin with the case~$k = 0$, where we use $F_{0,0,\lam,\mu,\rho}$ or~$S_{\lam,\rho,\Dy}$ depending on the level of coupling (cf. claims 1 and 2 of~\cref{lem:hard-properties}) and obtain the following result.

\begin{proposition}%
\label{prop:lower-bound-0}
For~$\lam, \mu, \rho, \Dy > 0$, let~$X = [-\frac{\mu \Dy}{2\lam},\frac{\mu \Dy}{2\lam}]$,~$\fapx \equiv \fapx_0$, cf.~\cref{def:fk}, with~$f, Y, \by$ to be defined.
\begin{enumerate}
\item 
Let~$f = F_{0,0,\lam,\mu,0}$ and~$Y = [-\frac{1}{2}\Dy, \frac{1}{2}\Dy]$.
\odima{Then\,~$\exists (x^*, \by) \in X \times Y$:\,~$\vphiapx_{2\lam}'(x^*)=0$ and
$
|\vphi_{2\lam}'(x^*)| \ge \frac{\mu \Dy}{2}.
$}
\item 
Let~$f = S_{\lam,\rho,\Dy}$,~$Y = [0,\Dy]$,~$\mu \ge \sqrt{\frac{2\lam \rho}{\Dy}}$. 
\odima{Then~$\exists (x^*,\by) \in X \times Y$:~$\vphiapx_{2\lam}'(x^*)= 0$ and
$
|\vphi_{2\lam}'(x^*)| \ge \frac{\sqrt{\lam \rho \Dy}}{3}.
$}
\end{enumerate}
\end{proposition}
When combined with the first two claims of~\cref{lem:hard-properties},~\cref{prop:lower-bound-0} establishes~\cref{eq:lower-bound-k} for~$k = 0$. 
As a result, we verify tightness of~\cref{th:upper-bound} in the case of zeroth-order approximation.
Note, however, that Lemma~\ref{lem:hard-properties} is restricted to the regime~$\sigma_0 \approx \mu \Dy$; hence, the term~$\sigma_0$ in the left-hand side of~\cref{eq:upper-bound-k}---which is beneficial when~$\sigma_0 \ll \mu \Dy$---is not captured in the result we have just obtained.

Next we address the case of~first-order approximation. 
Here we consider instances of~\cref{opt:min-max} with objective given by~$F_{1,-1,\lam,\mu,\rho}$ or~$F_{1,1,\lam,\bar \mu,\bar\rho}$ for some~$\bar\mu \le \mu$ and~$\bar \rho \le \rho$, depending on the region of parameters (as controlled by the relative level of coupling~$\Lxy$ compared to the geometric mean~$\sqrt{\lam\rho_1}$ of the ``homogeneous'' Lipschitz constants, cf.~\cref{eq:tensy-lip} with~$k = 1$). Here we obtain the following result.

\begin{proposition}%
\label{prop:lower-bound-1}
For~$\lam,\mu,\rho,\Dy > 0$, set~$Y = [-\half\Dy, \half\Dy]$,~$\fapx \equiv \fapx_1$, cf.~\cref{def:fk}, with~$f,\by$ to be defined.
\begin{enumerate}
\item 
\odima{Let~$\mu \le \sqrt{\frac{\lam \rho}{2}}$ and~$f = F_{1,-1,\lam,\mu,\rho}$. Then~$\exists (x^*,\by) \in \R \times Y$ such that $\vphiapx_{2\lam}'(x^*)=0$ and
$
|\vphi_{2\lam}'(x^*)| \ge \frac{\mu \Dy}{3}.
$}
\item 
\odima{Let~$\mu \ge \sqrt{\frac{\lam \rho}{2}}$ and~$f = F_{1,1,\lam,\bar \mu,\bar\rho}$,} then~$\exists\, \bar\mu \le \mu$,~$\bar\rho \le \rho$,~$(x^*,\by) \in \R \times Y$:~$\vphiapx_{2\lam}'(x^*)=0$ and
$
|\vphi_{2\lam}'(x^*)| \ge \sqrt{\frac{\lam \rho}{8}}\Dy.
$
\end{enumerate}
\end{proposition}
By combining~\cref{prop:lower-bound-1} with \odima{claims} 1, 3 of~\cref{lem:hard-properties}, we establish~\cref{eq:lower-bound-k} for~$k = 1$, and thus verify tightness of~\cref{th:upper-bound} in the case of first-order approximation, cf.~\cref{eq:upper-bound-k}, without imposing any restrictions on the problem parameters.  
More precisely, our lower bound~$\veps \gsim \min\{\mu\Dy, \sqrt{\lam\rho \Dy^2}\}$ on the approximation accuracy matches the upper bound~\cref{eq:upper-bound-k} for bilinearly-coupled objectives~\cref{eq:bilinear-coupling}, and replaces~$\Lxxp_1 = \lam + 2\tau_1 \Dy$ with~$\lam$ in the general case; this is only a minor modification since~$\Dy = O(\veps)$ is anyway required in order for the guarantees in Theorem~\ref{th:upper-bound} and~\cref{prop:lower-bound-1} to be applicable.

It remains to cover the case of approximation with~$k \ge 2$. 
To this end, we \odima{construct} instances of~\cref{opt:min-max} with~$f = F_{1,-1,\lam,\bar\mu,\rho}$ or~$f = F_{1,1,\lam,\bar \mu,\rho}$ for some~$\bar\mu \le \mu$.
\odima{As in the previous case~$(k = 1)$, these two instances allow to cover, respectively, the regimes or weak/strong interaction between the variables.} %

\begin{proposition}
\label{prop:lower-bound-k}
For~$\lam,\mu,\rho,\Dy > 0$ and~$k \ge 2$, set~$\fapx = \fapx_k$, cf.~\cref{def:fk}, with~$f, Y, \by$ to be defined, and let
\begin{equation}
\label{eq:mu-crit-k}
\mucrit := \sqrt{\frac{\lam \rho \Dy^{k-1}}{k!}}.
\end{equation}
%
\begin{enumerate}
\item 
If~$\mu \le \mucrit$, then for~$Y = [0,\Dy]$ and~$f = F_{k,-1,\lam,\bar\mu,\rho}$ with some~$\bar \mu \le \mu$ one can find~$\by \in Y,$~$x^* \in \R$ such that~$\vphiapx_{2\lam}'(x^*)=0$ while
\[
|\vphi_{2\lam}'(x^*)| \ge \frac{\mu \Dy}{2k}.
\]
\item 
If~$\mu \ge \mucrit$, then for~$Y = [-\half\Dy,\half\Dy]$ and~$f = F_{k,1,\lam,\bar{\mu},\rho}$ with some~$\bar{\mu} \le \mu$ one can find~$\by \in Y$,~$x^* \in \R$ such that~$\vphiapx_{2\lam}'(x^*)=0$ while
\[
|\vphi_{2\lam}'(x^*)| \ge \frac{\mucrit \Dy}{2k}.
\]
\end{enumerate}
\end{proposition}
By combining~\cref{prop:lower-bound-k} with claims 1 and 3 of~\cref{lem:hard-properties}, we establish~\cref{eq:lower-bound-k} for~$k \ge 2$, 
Thus, we verify that for BC objectives, the bound in~\cref{th:upper-bound} is tight, up to the two terms in~\cref{eq:upper-bound-k} being divided by~$O(k)$ and~$O(\sqrt{k})$ correspondingly; 
removing this gap is left for future work.

\paragraph{On the proofs of~\cref{prop:lower-bound-0,prop:lower-bound-1,prop:lower-bound-k}.}
While our choice of objective functions used in~\cref{prop:lower-bound-0,prop:lower-bound-k} is quite natural (except, perhaps, for~\cref{eq:hard-quadratic} used in the case~$k = 0$), the actual proofs of these propositions, as given in~\cref{sec:lower-proofs}, require \odima{substantial} technical work and are way less straightforward than the proof of~\cref{th:upper-bound}.
In all three cases, the analysis relies upon carefully choosing the center~$\by$ of Taylor expansion and exhibiting~$x^*$ which is stationary for~$\vphiapx_{2\lam}$, yet such that~$| \vphi_{2\lam}'(x^*)|$ is large. The choices of~$\by, x^*$ depend on the problem parameters and approximation order; in particular, they are different in the two regimes of strong/weak coupling.
The analysis in Proposition~\ref{prop:lower-bound-k} is especially delicate, notably due to our ambitious goal of matching the bound of~\cref{th:upper-bound} up to a {\em polynomial in~$k$} gap.
(For example, the reader may verify that merely replacing~$Y = [0,\Dy]$ with~$[-\half\Dy,\half\Dy]$ in the case~$\mu < \mucrit$ would result in the extra factor of~$2^k$.)\\ 

We \odima{conjecture} that the extra~$O(1/k)$ factor in~\cref{prop:lower-bound-k}, as compared to~\cref{eq:upper-bound-k}, can be eliminated. However, this might require analyzing a different family of problems, as our analysis of~\cref{eq:hard-quadratic} seems to be tight. 
In any case, such an improvement is of moderate practical interest: approximation with~$k \gg 1$ does not lead to efficient algorithms for searching FOSPs: as we shall see next, such schemes rely on solving the nested maximization problem, which becomes a daunting task for~$k > 2$. 

\section{Efficient algorithms for the search of first-order stationary points}
\label{sec:algos}

Our goal in this section is to implement step~\proofstep{2} of the strategy outlined in the introduction. 
To this end, we exploit the guarantee in~\cref{th:upper-bound} by replacing the task of finding FOSPs in~\cref{opt:min-max} with that of finding FOSPs in~\cref{opt:min-max-apx} with~$k \in \{0,1,2\}$, and propose efficient algorithms for solving the latter task.
Overall, as~$k$ increases, the proposed algorithms require access to higher derivatives of~$f$ in~$y$ and their oracle complexity estimates also deteriorate (especially when transitioning from~$k = 1$ to~$k = 2$).
On the other hand, increasing~$k$ allows us to handle larger diameter of~$Y$ for the same~$\veps$. %

Note that, despite the fact that our ``approximation theory'' in~\cref{sec:upper,sec:lower} carefully handles the general case of arbitrary~$k$, we do not propose algorithms based on Taylor approximations of order~$k \ge 3$. This is because already for cubic approximation (i.e., when~$k = 3$), solving the nested maximization problem becomes a daunting task: as it is shown in~\cite[Theorem~4]{nesterov2003random}, maximization of a general trilinear form is NP-hard even when it is available explicitly (in its tensor representation).

%
%

\subsection{Algorithms based on the constant and linear approximations}
In the case~$k = 0$, we fix arbitrary~$\by \in Y$ and consider~\cref{opt:min-max-apx} with objective~$\fapx_0 (x,y) = f(x,\by)$, that is
\begin{equation}
\label{opt:min-max-apx-0}
\min_{x \in X} f(x,\by). 
\tag{${\mathsf{P}}_0$}
\end{equation}
This is a nonconvex {\em minimization} problem with a smooth objective~$\vphiapx(x) = f(x,\by)$, so we can find a near-stationary point via projected gradient descent. 
This approach is summarized in~\cref{alg:zeroth-order}. 
It produces a point~$x_T$ satisfying~$\|\nabla \vphiapx(x_T)\| \lsim \veps$ in~$O(1/\veps^2)$ iterations, with one projected gradient step (in~$x$) per iteration.  
Using~\cref{lem:fne-to-moreau} stated below, this implies~$\|\nabla \vphiapx_{2\lam}(x_T)\| \lsim \veps$ in terms of the approximate Moreau envelope, which then results in the desired guarantee~$\|\nabla\vphi_{2\lam}(x_T)\| \lsim \veps$ by applying~\cref{th:upper-bound} with~$k = 0$. 
We shall rigorously state these results later on in~\cref{th:algo-0}.

Another approach we advocate here is based upon focusing on the nonconvex-affine problem
\begin{equation}
\label{opt:min-max-apx-1}
\min_{x \in X} \max_{y \in Y} \fapx_1(x,y),
\tag{${\mathsf{P}}_1$}
\end{equation}
which corresponds to~\cref{opt:min-max-apx} with~$k = 1$ (the choice of~$\by \in Y$ is again arbitrary).
General nonconvex-concave problems can be solved by a simple gradient descent-ascent (GDA) scheme combined with quadratic regularization in~$y$, with iteration complexity~$O(\veps^{-5})$~\cite{nouiehed2019solving}.\footnotemark\; 
More elaborate algorithmic schemes based on the proximal-point method have~$O(\veps^{-3})$ iteration complexity~\cite{ostrovskii2020efficient,kong2019accelerated,thekumparampil2019efficient, zhao2020primal}.\footnotetext{This estimate follows from the complexity~$O(\ex^{-2}\ey^{-3/2})$ of finding an~$(\ex,\ey)$-approximate first-order Nash equilibrium by such method, see~\cite[p.~3]{ostrovskii2020efficient}, combined with~\cite[Proposition~5.5]{ostrovskii2020efficient} which verifies that the~$x$-component of such a point is an~$O(\ex)$-FOSP as long as~$\ey = O(\ex^2)$. 
Note that the well-known result~\cite[Proposition~4.12]{lin2019gradient} commonly used for such a reduction in recent works, is erroneous---see the discussion immediately following~\cite[Proposition~5.2]{ostrovskii2020efficient}.}
Here we propose a  GDA-type scheme in the form of~\cref{alg:first-order}.
With it we manage to guarantee~$O(\veps^{-2})$ iteration complexity, by exploiting the special properties of~\cref{opt:min-max-apx-1}:
\begin{itemize}
\item $\fapx_1(x,\cdot)$ is affine, so~$\fapx_1(x,\cdot) + \tfrac{\rho}{2} \|\cdot - \by\|^2$ is maximized via a single projected gradient ascent step.
\item \cref{opt:min-max-apx-1} has to be solved in the regime~$\veps \gsim \Dy \min\{\Lxy, {(\Lxxp_{1} \rho_1)}^{1/2}\}$ where~$\veps$-FOSPs for~\cref{opt:min-max-apx-1} translate to~$\veps$-FOSPs for~\cref{opt:min-max} via~\cref{th:upper-bound} (cf.~\cref{eq:upper-bound-k}~with~$k = 1$). 
This allows to regularize with~$\rho = \rho_1$ and results in~$\veps$-independent smoothness~$O(\Lxxp_1 + \Lxy^2/\rho_1)$ of the corresponding \odima{max-function}.
\end{itemize}

{\flushleft
\begin{minipage}{0.42\textwidth}
\begin{algorithm}[H]
\caption{FOSP search using~$\fapx_0(x,\cdot)$}
\label{alg:zeroth-order}
\begin{algorithmic}[1]
\Require{$x_0 \in X$;~$\by \in Y$;~$\gamx > 0$;~$T \in \N$}
\State $x^* = x_0$; \; $\veps^* = +\infty$
\For{$t \in \{0, 1, ..., T-1\}$} 
\State $\wt x_{t+1} = x_{t} - \gamx \gx f(x_{t},\by)$
\State $x_{t+1} = \proj_{X}[\wt x_{t+1}]$ %
\LineComment{{\em Maintain the best iterate}}
\State $\veps_{t}^2 = \| \gx f(x_{t},\by) \|^2 - \frac{1}{\gamx^2} \|\wt x_{t+1} - x_{t+1} \|^2$
\label{line:eps_t-0}
\vspace{-0.4cm}
\If{$\veps_t < \veps^*$} 
\State $x^* = x_{t}$; \; $\veps^* = \veps_t$
\EndIf
\EndFor 
\Ensure{$x^*$}
\end{algorithmic}
\end{algorithm}
\vspace{2.35cm}
\end{minipage}
\begin{minipage}{0.58\textwidth}
\flushright
\begin{algorithm}[H]
\caption{FOSP search using~$\fapx_1(x,\cdot)$}
\label{alg:first-order}
\begin{algorithmic}[1]
\Require{$x_0 \in X$;~$\by \in Y$;~$\gamx,\gamy > 0$;~$T\in\N$;~$\Coupled\in\{0,1\}$}
\State $x^* = x_0$; \; $y^* = y_{t}$; \; $\veps^* = +\infty$ %
\For{$t \in \{0, 1, ..., T-1\}$} 
\If{$\Coupled$} 
\State $y_{t} = \proj_{Y}[\by + \gamy \gy f(x_{t},\by)]$
\label{line:y_t-strong}
\State $\wt x_{t+1} = x_{t} - \gamx [\gx f(x_{t},\by) + \hxy f(x_{t},\by) (y_{t} - \by)]$
\label{line:x_t-strong}
\Else 
\State Choose~$y_{t} \in \Argmax_{y \in Y} \lang \gy f(x_{t},\by), y\rang$
\label{line:y_t-weak}
\State $\wt x_{t+1} = x_{t} - \gamx \gx f(x_{t},y_{t})$
\label{line:x_t-weak}
\EndIf
\State $x_{t+1} = \proj_{X}[\wt x_{t+1}]$ %
\label{line:x_t-proj}
\LineComment{{\em Maintain the best iterate}}
\State $\veps_{t}^2 = \frac{1}{\gamx^2} \| \wt x_{t+1} - x_{t} \|^2 - \frac{1}{\gamx^2} \|\wt x_{t+1} - x_{t+1} \|^2$
\label{line:eps_t-1}
\If{$\veps_t < \veps^*$} 
\State $x^* = x_{t}$; \; $y^* = y_{t}$; \; $\veps^* = \veps_t$ 
\EndIf
\EndFor 
\Ensure{$x^*$} %
\end{algorithmic}
\end{algorithm}
\end{minipage}
}

\cref{alg:first-order} admits~$O(\veps^{-2})$ iteration complexity estimate as in the case of~\cref{alg:zeroth-order}. %
That said, compared to~\cref{alg:first-order}, the new algorithm has two advantages: a slightly increased range of available accuracies in the strongly-coupled regime~$\mu^2 \ge \Lxxp_1 \rho_1$; a smaller leading factor in the complexity estimate -- depending on the (primal) suboptimality gap, rather than the full duality gap. %
However, these improvements come at a price: computing~$\gx \fapx_{1}(x,y)$ requires access to a partial Hessian-vector product oracle for~$f$, namely
\begin{equation}
\label{eq:hess-product-oracle}
(x,y) \mapsto \hxy f(x,\by)(y-\by),
\end{equation}
since
$
\gx \fapx_{1}(x_t,y_t) = \gx f(x_t,\by) + \hxy f(x_t,\by)(y_t-\by).
$
On the other hand, for weakly-coupled problems\odima{---}i.e., when~$\mu^2 \le \Lxxp_1 \rho_1$\odima{---}\cref{alg:first-order} uses a simplified approach: (i) the descent step is performed in the negative direction of~$\gx f(x_t,y_t)$ instead of~$\gx \fapx_1(x_t,y_t)$; 
(ii) the gradient ascent step is replaced by the full maximization of the linear model~$\fapx_1(x_t,\cdot)$. %
These two properties allow us to avoid Hessian-vector product~\cref{eq:hess-product-oracle}, and also to access~$Y$ through the (weaker) linear maximization oracle.

Next we state convergence guarantees for~\cref{alg:zeroth-order,alg:first-order} (see~\cref{sec:algos-proofs} for the proofs). %

\begin{theorem}
\label{th:algo-0}
Grant~\cref{ass:gradx}. Running~\cref{alg:zeroth-order} with~$\gamx=\frac{1}{\Lxx}$ and number of iterations
\begin{equation}
\label{eq:zeroth-order-complexity}
T \ge \frac{300 \lam [ \vphi(x_0) - \psi(\by) ]}{\veps^2},
\end{equation}
where~$\psi(y) := \min_{x \in X} f(x,y)$ is the dual function of~\cref{opt:min-max}, 
guarantees~$\|\nabla\vphiapx_{2\Lxx}(x^*)\| \le \veps / 6$. 
Moreover, we have~$\|\nabla\vphi_{2\Lxx}(x^*)\| \le \veps$\odima{---}i.e., in terms of initial problem~\cref{opt:min-max}\odima{---provided that}
$24\Lxy\Dy \le \veps.$
\end{theorem}

Note that the factor~$\vphi(x_0) - \psi(\by)$ in~\cref{eq:zeroth-order-complexity} is the duality gap for the point~$(x_0, \by)$; 
by weak duality, it is lower-bounded by the sum of the dual gap~$\max_{y \in Y}\psi(y) - \psi(\by)$ and the primal gap
\begin{equation}
\label{def:prim-gap}
\Gap := \vphi(x_0) - \min_{x \in X} \vphi(x).
\end{equation}
As we shall see next,~\cref{alg:first-order} admits a slightly different (and typically better) complexity estimate, in which the full duality gap is replaced with the primal gap, and~$\Lxx$ with~$O(\Lxxp_1 + \Lxy^2/\rho_1)$.
In addition, we relax the condition~$24\Lxy \Dy \le \veps$, as imposed in Theorem~\ref{th:algo-0},
by replacing~$\Lxy$ with~$O(\min\{\Lxy,(\Lxxp_1 \rho_1)^{1/2}\})$.

\begin{theorem}
\label{th:algo-1}
Grant~\cref{ass:gradx,ass:tensy,ass:measurable} for~$k = 1$, let~$\Lxxp_1 = \Lxx + 2\tau_1\Dy$ (cf.~\cref{eq:gx-lip}), and assume that %
\begin{equation}
\label{eq:first-order-diameter}
200 \min \{ \Lxy, {(\Lxxp_{1} \rho_1)}^{1/2} \} \Dy \le \veps.
\end{equation}
Running~\cref{alg:first-order} with~$\gamx = \frac{1}{3\Lxxp_1 + \mu^2/\rho_1}$,~$\Coupled = \ind\{\mu \ge {(\Lxxp_1 \rho_1)^{1/2}}\}$ and~$\gamy = \frac{1}{\rho_1}$ if~$\Coupled = 1$, for 
\begin{equation}
\label{eq:first-order-complexity}
T \ge \left(3 + \frac{\mu^2}{\Lxxp_1\rho_1}\right) \left( \frac{700\Lxxp_1 \Gap}{\veps^2}  + 1 \right)
\end{equation}
iterations, with~$\Gap$ being the initial gap as per~\cref{def:prim-gap}, results in~$x^* \in X$ for which~$\|\nabla\vphi_{2\Lxxp_1}(x^*)\| \le \veps$.
\end{theorem}

\begin{remark}
\label{rem:eps_t}
As a criterion for selecting the ``best'' iterate, in~\cref{alg:zeroth-order,alg:first-order} we use the quantity
\[
\begin{aligned}
\veps_{t}^2 
= \frac{1}{\gamx^2} \left( \| \wt x_{t+1} - x_{t} \|^2 - \|\wt x_{t+1} - x_{t+1} \|^2 \right),
\end{aligned}
\]
where~$\wt x_{t+1}$ is the result of the gradient descent step from~$x_{t}$ prior to projection (i.e.,~$x_{t+1} = \proj_{X}[\wt x_{t+1}]$).
In fact,~$\veps_{t} = \Sx(x_t,\tfrac{1}{\gamx}(x_t - \wt x_{t+1}),\tfrac{1}{\gamx})$, where~$\Sx$ is the functional used in the proof of~\cref{prop:upper-coupled}.
Using this criterion instead of the gradient norm (which is a weaker criterion, cf.~\cite[Theorem~4.3]{barazandeh2020solving}) allows to work with the Moreau envelope 
(cf., in particular,~\cref{lem:fne-to-moreau} in~\cref{sec:algos-proofs}), and seems to be necessary already in the nonconvex-concave setup (see~\cite[Proposition~5.5]{ostrovskii2020efficient}).
\end{remark}

\subsection{An algorithm based on the quadratic approximation}
\label{sec:algo-quad}

Finally, we propose a more sophisticated method in which we focus on the quadratic approximation
\begin{equation}
\label{opt:min-max-apx-2}
\min_{x \in X} \max_{y \in Y} 
\fapx_{2}(x,y).
\tag{${\mathsf{P}}_2$}
\end{equation}
(Recall that $\hat{f}_{2}(x,\cdot)$ is the quadratic approximation of $f(x,\cdot)$ in y with arbitrary choice of $\hat{y}\in Y$
cf.~\cref{eq:0-1-models}.) Compared to the previous ones, the approach we are about to present requires Y to be a Euclidean ball in $\mathbb{R}^{d}$. This restriction is necessary for computational tractability when solving the maximization subproblem. While $l_{\infty}$-norm constraints are common in applications like adversarial training, maximizing a potentially nonconcave (or even convex) quadratic function over the $l_{\infty}$-ball (hypercube) is generally NP-hard. It is known (see, e.g., \cite{de2008complexity}) that this problem does not admit an efficient PTAS unless P=NP. Similar intractability results hold for the $l_1$-ball~\cite{vavasis1991nonlinear}. The tractability under the $l_2$-constraint enables our algorithmic approach here. For simplicity, we shall assume that Y has diameter precisely~$\Dy$, and is origin-centered---i.e.,~$Y = B_d(\Dy)$ where \odima{we define}
$
B_d(\Dy) := \{y \in \R^d: \|y\| \le \textstyle\half\Dy\},
$
\odima{the scaled Euclidean ball.}
Note that centering~$Y$ in the origin is not a limitation:~\cref{ass:gradx,ass:tensy} are preserved under shifts of~$y$.
\odima{Our construction rests upon the following two observations.}

\odima{Firstly, we observe,} following~\cite{davis2018stochastic}, that an~$(\frac{\veps}{6},2\Lxxp_2)$-FOSP in~\cref{opt:min-max-apx-2} can be found by running~$O(\veps^{-4})$ iterations 
of a (projected) subgradient scheme on the associated to~\cref{opt:min-max-apx-2}
\odima{max-function}~$\vphiapx(x)$, %
which is~$\Lxxp_2$-weakly convex (cf.~\cref{lem:gx-lip}). 
By~\cref{th:upper-bound}, this also gives an~$(\veps,2\Lxxp_2)$-FOSP in~\cref{opt:min-max-apx-2}. 
Each iteration of the subgradient scheme amounts to alternating between a maximization step in~$y$, i.e., finding~$y^* = y^*(x) \in \Argmax_{y \in Y}\fapx_2(x,y)$, and a projected gradient descent step on~$\fapx_2(\cdot,y^*)$.
Moreover, the analysis in~\cite{davis2018stochastic} (cf. also~\cite[Theorem~31]{jin2020local}) shows that this complexity is preserved under objective value errors of up to~$O(\veps^2/\Lxxp_2)$ in the maximization step. 

\odima{Secondly, we observe} that, despite the corresponding objective~$\fapx_2(x,\cdot)$ being nonconcave, the maximization steps can be efficiently performed by running a first-order algorithm on~$\fapx_2(x,\cdot)$. To this end, we make use of the recent result of~\cite{carmon2020first}, who showed that the Krylov subspace of dimension~$\wt O(\Dy\delta^{-1/2})$ contains a~$\delta$-accurate maximizer of a nonconcave quadratic form on a Euclidean ball.
Krylov-type schemes can usually be efficiently implemented via a Lanczos-type method (see~\cite{gould1999solving}), with~$\wt O(\Dy\delta^{-1/2})$ matrix-vector products  to find a~$\delta$-accurate maximizer. %

\vspace{0.1cm}

Below we present~\cref{alg:second-order} which adapts the general subgradient scheme~\cite{davis2018stochastic} to the present situation, assuming access to an abstract maximization oracle~$\ApproxMax(\fapx_2(x,\cdot),Y,\delta)$ returning a~$\delta$-accurate maximizer of~$\fapx_2(x,\cdot)$ over~$Y$. 
Efficient implementation of this oracle, given in the form of~\cref{alg:carmon}, is discussed in~\cref{app:krylov}.
Observe that~\cref{alg:second-order} can be run in a simplified (``naive'') regime, wherein the descent step is performed using~$\gx f(\cdot,y_t)$ rather than~$\gx\fapx_2(\cdot,y_t)$, so there is no need of higher-order oracles used otherwise (cf.~\cref{line:grad-2-fapx-grad} of~\cref{alg:second-order}).
\begin{center}
\begin{algorithm}%
\caption{FOSP search based on~$\fapx_2(x,\cdot)$}
\label{alg:second-order}
\begin{algorithmic}[1]
\Require{$x_0 \in X$;~$\by \in Y$;~$\gamx > 0$;~$T \in \N$;~$\delta > 0$; $\Naive \in \{0,1\}$}
\For{$t \in \{0, 1, ..., T-1\}$} 
\State $y_t = \ApproxMax(\fapx_2(x_t,\cdot),Y,\delta)$ 
\Comment{{\em implemented in~\cref{alg:carmon} (cf.~\cref{app:krylov})}}
\label{line:max-2}
\If{$\Naive$}
\State $x_{t+1} = \proj_{X} \left[x_{t} - \gamx \gx f(x_{t},y_t) \right]$
\label{line:grad-2-naive}
\Else
\State $x_{t+1} = \proj_{X} [x_{t} - \gamx \gx \fapx(x_{t},y_t) ]$, 
\label{line:grad-2-approx}
\State where~$\gx \fapx(x_{t},y_t) = \gx f(x_{t},\by) + \hxy f(x_{t},\by) (y_t - \by) + \nabla^3_{\x\y\y} f(x_{t},\by) [\cdot, y_t - \by,  y_t - \by]$
\label{line:grad-2-fapx-grad}
\EndIf
\EndFor 
\Ensure{$x_s$, where~$s \in \{0,..., T-1\}$ is sampled uniformly at random}
\end{algorithmic}
\end{algorithm}
\end{center}

\vspace{-0.2cm}
We now present a convergence guarantee for~\cref{alg:second-order}.
\begin{proposition}
\label{prop:subgradient-scheme}
Grant~\cref{ass:gradx,ass:tensy,ass:measurable} for~$k = 2$, let~$\Lxxp_2 = \Lxx + \tau_2\Dy^2$ (cf.~\cref{eq:gx-lip}), and assume that %
\begin{equation}
\label{eq:second-order-diameter}
24 \min\left\{ \Lxy \Dy + \sigma_2 \Dy^2,  \; \sqrt{\tfrac{1}{300}\Lxxp_2 \rho_2 \Dy^{3}}  \right\} \le \veps.
\end{equation}
Furthermore, assume that~$f(\cdot,y)$ is~$\sigma_0$-Lipschitz for any~$y \in Y$.
Finally, for~$\delta > 0$ and any~$x \in X$, let~$\ApproxMax(\gy\fapx_2(x,\cdot),Y,\delta)$ output~$y_{\delta} = y_{\delta}(x) \in Y$ such that
$\fapx_2(x,y_{\delta}) \ge \max_{y \in Y} \fapx_2(x,y) - \delta.$
Then:
\begin{enumerate}
\item~\cref{alg:second-order} run with~$\Naive = 0$ and the choice of parameters (for fixed~$\prob \in (0,1)$)
\begin{align}
\label{eq:quad-algo-params}
\gamx &= \frac{1}{\sigma_0 + \sigma_2 \Dy^2} \sqrt{\frac{\Gap + \rho_2 \Dy^3}{\Lxxp_2 T}}, \quad
\delta = \frac{4\prob}{10^{4}} \cdot  \frac{\veps^2}{\Lxxp_2}, \quad 
T \ge \frac{6 \cdot 10^6}{\prob^2} \cdot \frac{\Lxxp_2 (\Gap + \rho_2\Dy^3)  (\sigma_0 + \sigma_2 \Dy^2)^2}{\veps^4},
\end{align}
for~$\Gap$ defined in~\cref{def:prim-gap}, with probability at least~$1-\prob$ outputs~$x_s \in X$ such that~$\|\nabla\vphi_{2\Lxxp_2}(x_s)\|\le \veps$. 
\item
Moreover, the output of~\cref{alg:second-order} run with~$\Naive = 1$ has the same property if~$24 \sigma_2 \Dy^2 \le \veps \sqrt{\prob}$.
\end{enumerate}
\end{proposition}
We prove~\cref{prop:subgradient-scheme} in~\cref{sec:algos-proofs}.
The first claim is proved by following the footsteps of~\cite[Theorem~31]{jin2020local} up to minor modifications: first, parameters~$\sigma_0$ and~$\Gap$ have to be adjusted for the use of~$\vphiapx$ instead of~$\vphi$; second, the bound in expectation is replaced with a fixed-probability bound. 
The argument proceeds by establishing~$O(\veps^{-4})$ complexity in terms of the surrogate~$\| \vphiapx_{2\Lxxp_2}(x_s)\|$, and then applying~\cref{th:upper-bound} under the high-probability event, which allows to control~$\| \vphi_{2\Lxxp_2}(x_s)\|$.

The second claim, pertaining to a simplified (``naive'') variant of the algorithm, where descent is performed in the direction of~$\gx f(\cdot,y_t)$ rather than that of~$\gx\fapx_2(\cdot,y_t)$, is proved by controlling the resulting perturbation of~$\|\vphiapx_{2\Lxxp_2}(x_t)\|$. This perturbation turns out to be~$O(p^{-1/2}\sigma_2 \Dy^2)$, so requiring that~$\sigma_2 \Dy^2 \lsim \veps\sqrt{p}$ %
suffices for the ``naive'' approach to work. 
Note that under~\cref{eq:second-order-diameter}, this requirement is very weak: it is either satisfied right away if the minimum in~\cref{eq:second-order-diameter} is attained on the first argument, or follows from~\cref{eq:second-order-diameter} when~$\Dy$ is smaller than an $\veps$-independent threshold---namely, when~$\sigma_2^2 \Dy \lsim {\Lxxp_2\rho_2}$.
Moreover, in the latter case---which is of main interest, as otherwise there is no advantage in using~\cref{alg:second-order} over~\cref{alg:zeroth-order,alg:first-order} anyway---the number of iterations as per~\cref{eq:quad-algo-params} becomes
\[
\frac{1}{\prob} \, O\left(\frac{\Lxxp_2 \Gap \sigma_0^2}{\veps^4} + \frac{\Lxxp_2 \Gap + \sigma_0^2}{\veps^2} + 1\right).
\]

\paragraph{Implementation of the max-oracle.}
Next we show how to implement~$\ApproxMax(\fapx_2(x,\cdot),Y,\delta)$ in the case where~$Y$ is a Euclidean ball. 
To this end, for~$g \in \R^d$ and a symmetric~$H \in \R^{d\times d}$, we let
\begin{equation}
\label{eq:quad-carmon}
\Psi_{H,g}(y) = \frac{1}{2} y^\top H y + g^\top y
\end{equation}
be the corresponding quadratic form, and  we aim at efficiently solving problems of the form
\begin{equation}
\label{opt:quadratic}
\max_{y \in B_d(\Dy)} \Psi_{H,g}(y)
\end{equation}
up to accuracy~$\delta > 0$ in objective value given access to~$g$ and the matrix-vector multiplication oracle~$y \mapsto Hy$.  
In order to accomplish this goal, we shall exploit the following result from~\cite{carmon2020first}.
\begin{proposition}[{\cite[Corollary~5.2]{carmon2020first}}]
\label{prop:carmon-perf}
Define
$\cK_{2m}(H, \{g,\xi\}) := \textup{span} \left( \{H^{j} g, H^{j} \xi  \}_{j \in \{0,...,m-1\}} \right)$, 
the joint Krylov subspace, where~$\xi \sim \Uniform(\mathds{S}^{d-1})$.
For any~$\Ry \ge 0$ and~$q \in (0,1)$, w.p.~$\ge 1-q$ one has
\[
\max_{\|y\| \le \Ry} \Psi_{H,g}(y)  
- \max_{y \, \in \, \cK_{2m}(H, \{g,\xi\}): \; \|y\| \le \Ry} \Psi_{H,g}(y)  
\le \frac{4\|H\|\Ry^2}{m^2} \left( 2 + \log^2\left( \frac{2\sqrt{d}}{q} \right) \right).
\]
\end{proposition}
\cref{prop:carmon-perf} immediately implies that whenever
$
m \ge \Dy \sqrt{\frac{\|H\|}{\delta} \left( 2 + \log^2\left(\frac{2\sqrt{d}}{q}\right) \right)},
$
the corresponding joint Krylov subspace~$\cK_{2m}(H,\{g,\xi\})$
with probability at least~$1-q$ contains a~$\delta$-suboptimal solution to~\cref{opt:quadratic} in terms of objective value. 
Now, since~$\fapx_2(x,y) = \Psi_{\hat H(x),\hat g(x)}(y)$ with
\begin{equation}
\label{eq:grad-and-hess}
\hat g(x) = \gy f(x,\by) \quad \text{and} \quad \hat H(x) = \hyy f(x,\by),
\end{equation}
we conclude that, granted~\cref{ass:tensy} with~$k = 1$ (more precisely, finiteness of~$\rho_1$, cf.~\cref{eq:tensy-lip}), any optimal solution to the problem 
\begin{align}
\label{eq:krylov-at-x}
&\max_{y \, \in \, \cK_{2\mbar}(\hat H(x), \{\hat g(x),\xi\}): \; \|y\| \le \half \Dy} \Psi_{H(x),g(x)}(y) 
&\text{with} \quad
\mbar = \left\lceil \min \left\{ \Dy \sqrt{\frac{\rho_1}{\delta} \left(2 + \log^2\left(\frac{2\sqrt{d}}{q}\right)  \right)}, \,  \frac{d}{2} \right\} \right \rceil 
\end{align}
implements the query~$\ApproxMax(\fapx_2(x,\cdot),B_{d}(\Dy),\delta)$ with probability at least~$1-q$. 
On the other hand, as discussed in~\cite{carmon2020first}, the computational burden of solving~\cref{eq:krylov-at-x} to machine precision is dominated by~$O(\mbar)$ calls of the oracle~$(x,y) \mapsto [\hat g(x), \hat H(x) y]$, inner products, and element-wise vector operations on~$\E_{\y}$ (typically~$y \mapsto \hat H(x) y$ is the most expensive of these operations).\footnote{Such an implementation is discussed in~\cite[Appendix A]{carmon2020first}, but somewhat informally, and no pseudocode of an algorithm is given. For this reason, in~\cref{app:krylov} we provide a formal algorithm (following the footsteps of~\cite{carmon2020first}) and analyze its complexity. Note that this can also be useful in the broader context of nonconvex quadratic optimization.}
Now, by recalling~\cref{prop:subgradient-scheme} and plugging in the value of~$\delta$ from~\cref{eq:quad-algo-params}, we arrive at the following result. %
\begin{theorem}
\label{th:algo-2}
Grant the premise of~\cref{prop:subgradient-scheme}.
Also, assume that~$Y = B_d(\Dy)$, and~$\gy f(x,\cdot)$ is~$\rho_1$-Lipschitz for all~$x$, i.e.,~$\|\gy f(x,y') - \gy f(x,y) \| \le \rho_1 \|y' - y\|$ \odima{for all}~$ x \in X$ and~$y,y' \in Y$. 
Choosing~$p \in (0,1)$ and~$q = (0,1-p)$, run~\cref{alg:second-order} with~$\Naive = 0$, parameters~$\gamx, \delta, T$ \odima{set as per}~\cref{eq:quad-algo-params}, and the oracle~$x \mapsto \ApproxMax(\fapx_2(x,\cdot),Y,\delta)$ implemented by running~\cref{alg:carmon} with~$g = \hat g(x)$,~$H = \hat H(x)$ (cf.~\cref{eq:grad-and-hess}),~$\Ry = \half\Dy$, and~$m = \lceil \min\{ M, d/2 \} \rceil$ with
\[
M = \frac{50 \, \Dy}{\veps} \sqrt{ \left( 2 + \log^2\left(\frac{2T\sqrt{d}}{q}\right)  \right) \frac{\rho_1 \Lxxp_2}{p}} 
\quad\quad
\left[ 
\stackrel{\cref{eq:second-order-diameter}}{\le} 
\;\;
\frac{80}{\min\{(\Lxxp_2 \rho_2 \veps )^{1/3}, 24\mu \}}\sqrt{\frac{\rho_1 \Lxxp_2}{p} \left( 1 + \log^2\left(\frac{T\sqrt{d}}{q}\right)  \right)}
\right].
\]

Then
the resulting point~$x_s \in X$ satisfies~$\|\nabla\vphi_{2\Lxxp_2}(x_s)\|\le \veps$ with probability at least~$1-(p+q)$, and is constructed by performing~$O(T)$ calls of the oracle
\begin{equation}
\label{quad-x-oracle}
\begin{aligned}
(x,y) 
\mapsto 
&\gx \fapx_2(x,y) 
\;\; \left[= \gx f(x,\by) + \hxy f(x,\by) (y - \by) + \nabla^3_{\x\y\y} f(x,\by) [\cdot, y - \by,  y - \by] \right]
\end{aligned}
\end{equation}
and projections onto~$X$, and~$O(MT)$ calls of the oracle
$(x,y) \mapsto ( \gy f(x,\by), \hyy f(x,\by) (y - \by) )$, inner products on~$E_{\y}$, and elementwise vector operations on~$E_{\y}$.

Moreover, if we in addition assume that~$24 \sigma_2 \Dy^2 \le \veps \sqrt{\prob}$ (cf. the second claim of~\cref{prop:subgradient-scheme}), then running~\cref{alg:second-order} with~$\Naive = 1$ produces~$x_s \in X$ with the same property while using the oracle~$(x,y) \mapsto \gx f(x,y)$ instead of~\cref{quad-x-oracle}. 
\end{theorem}

Comparing this result with~\cref{th:algo-0,th:algo-1}  we see that for~\cref{alg:second-order}\odima{,} the \odima{admissible} range of~$\Dy$ improves from~$O(\veps)$ to~$O(\veps^{2/3})$, but this happens at the price of a significant increase of complexity -- from~$O(\veps^{-2})$ to~$O(\veps^{-\sfrac{13}{3}})$. We leave open the questions of whether the latter complexity estimate can be improved, and whether one can smoothly interpolate between the two complexities.

\appendix

\section{Deferred proofs for Section~\ref{sec:upper}}
\label{app:aux}

\subsection{Proof of~\cref{lem:fval-err}}
Take arbitrary~$\by, y \in Y$ and~$x \in X$. 
Let~$y_t = (1-t)\by + t y$ for~$t \in [0,1]$, and define~$\psi_{x,y}: [0,1] \to \R$ by~$\phi_{x,y}(t) := f(x,y_t)$.
Clearly,~$f(x,y) = \phi_{x,y}(1)$. 
Moreover, \odima{for~$\phi_{x,y}(\cdot)$ the first~$k$ derivatives read}
\begin{equation}
\label{eq:psi}
\phi_{x,y}^{(j)}(t) = \Ty{j} f(x,y_t)[(y-\by)^j], \quad \odima{1} \le j \le k; 
\end{equation}
whence
$
\fapx_k(x,y) = \sum_{j = 0}^k \frac{1}{j!}\phi_{x,y}^{(j)}(0),
$
cf.~\cref{def:fk}.
\odima{Now: by~\cref{eq:tensy-lip} with~$x' = x$,} it holds that~$\phi_{x,y}^{(k)}$ is absolutely continuous on~$[0,1]$, 
hence its derivative exists almost everywhere on~$[0,1]$ and is given by
\begin{equation}
\label{eq:psikplus1}
\phi_{x,y}^{(k+1)}(t) = \Ty{k+1} f(x,y_t)[(y-\by)^{k+1}].
\end{equation}
Expressing the Taylor expansion remainder~$\phi_{x,y}(1) - \sum_{j = 0}^k \frac{1}{j!}\phi_{x,y}^{(j)}(0)$ in the integral form\odima{,} we get
\begin{equation}
\label{eq:taylor-rem}
f(x,y) - \fapx_k(x,y)
=
\int_{0}^1 \frac{(1-t)^k}{k!} \, \phi_{x,y}^{(k+1)}(t) \, dt 
=
\int_{0}^1 \frac{(1-t)^k}{k!} \, \Ty{k+1} f(x,y_t)[(y-\by)^{k+1}] \, dt.
\end{equation}
Whence we arrive at
\[
\begin{aligned}
|f(x,y) - \fapx_k(x,y)|
&\le  
\| y-\by\|^{k+1} \int_{0}^1 \frac{(1-t)^k}{k!} \, \|\Ty{k+1} f(x,y_t) \| \, dt  
\le 
\cLyky \Dy^{k+1}  \int_{0}^1 \frac{(1-t)^k}{k!} \, dt
= 
\frac{\cLyky \Dy^{k+1}}{(k+1)!}.\qed
\end{aligned}
\]

\subsection{Proof of~\cref{lem:gx-err}}
For~$k = 0$ the result is obvious: we have~$\gx \fapx_0(x,y) =  \gx f(x,\by)$, so~$\| \gx f(x,y) - \gx \fapx(x,\by) \|$ can be bounded via~\cref{ass:gradx} or via triangle inequality and Lipschitzness of~$f(\cdot,y)$. 
For~$k \ge 1$, fix~$\by$ and arbitrary~$y \in Y$ and~$x \in X$, and let~$y_t = (1-t)\by + t y$ for~$t \in [0,1]$. 
As in the proof of~\cref{lem:fval-err} we define~$\phi_{x,y}: [0,1] \to \R$ as~$\phi_{x,y}(t) := f(x,y_t)$, and observe that~\eqref{eq:psi}--\eqref{eq:taylor-rem} are still valid.
(Indeed, imposing~\cref{eq:tensy-lip} with~$\cLykx = \infty$ suffices for~$\phi_{x,y}^{(k)}$ to be absolutely continuous on~$[0,1]$, and hence for its derivative to exists almost everywhere on~$[0,1]$ and to be given by~\cref{eq:psikplus1}.) 
As a result, we have that
\begin{align}
f(x,y) - \fapx_k(x,y)
&=
\int_{0}^1 \frac{(1-t)^k}{k!} \, \phi_{x,y}^{(k+1)}(t) \, dt 
= 
-\frac{\phi_{x,y}^{(k)}(0)}{k!}  + \int_{0}^1 \frac{(1-t)^{k-1}}{(k-1)!} \, \phi_{x,y}^{(k)}(t) \, dt \nn
&= 
\int_{0}^1 \frac{(1-t)^{k-1}}{(k-1)!} \Big( \Ty{k} f(x,y_t) - \Ty{k} f(x,\by)\Big) [(y-\by)^k] dt,
\label{eq:taylor-rem-int}
\end{align}
where we first \odima{integrated by parts and then used}~\cref{eq:psi}. 
Taking the partial gradient in~$x$ \odima{results in}
\[
\lang \gx f(x,y) - \gx \fapx_k(x,y), u \rang
= 
\int_{0}^1 \frac{(1-t)^{k-1}}{(k-1)!} \left( \Txy{k+1}{k} f(x,\by)[(y-\by)^k; u] - \Txy{k+1}{k} f(x,y_t)[(y-\by)^k; u] \right)  dt,
\]
where~$u \in E_\x$ is arbitrary\odima{; here} we used the abridged notation~$[(y-\by)^k; u] := [(y-\by), .., (y-\by); u]$ for tensor evaluation,
and the right-hand side is well-defined by the premise of the lemma.
Taking supremum over the unit ball in~$\XX$ and combining Jensen's inequality with~\cref{eq:tensy-xx}, we arrive at
\[
\begin{aligned}
\| \gx f(x,y) - \gx \fapx_k(x,y) \| 
&\le 
\|y-\by\|^k \int_{0}^1 \frac{(1-t)^{k-1}}{(k-1)!} \left(\left\| \Txy{k+1}{k} f(x,\by)\right\| + \left\|\Txy{k+1}{k} f(x,y_t) \right\| \right) dt 
\le \frac{2\cLykx \Dy^k}{k!}.\qed
\end{aligned}
\]

\subsection{Proof of~\cref{lem:gx-lip}}
\odima{\paragraph{Lipschitzness in~$x$.}
In the case~$k = 0$, the result is immediate. 
Let~$k \ge 1$.
It suffices to show that}
\begin{equation}
\label{eq:hxx-err}
\|\hxx \fapx_k(x,y) - \hxx f(x,y) \| \le \frac{2\cLykxx \Dy^k}{k!}
\end{equation}
for all~$y \in Y$ almost everywhere on~$X$. 
Indeed,~\cref{eq:gradx-lip} with~$\mu = \infty$ is equivalent to~$\| \hxx f(x,y) \| \le \Lxx$ holding for all~$y \in Y$ almost everywhere on~$X$, 
whence~\cref{eq:hxx-err} would imply\odima{,} by the triangle inequality\odima{,} that~$\|\hxx \fapx_k(x,y)\| \le \Lxxp_k$ for all~$y \in Y$ almost everywhere on~$X$, which is equivalent to~\cref{eq:gx-lip}.
Hence it only remains to verify~\cref{eq:hxx-err}. 
This can be done via~\cref{eq:taylor-rem-int} (which is valid by continuity of~$\Ty{k} f(x,\cdot)$):
\begin{equation}
\label{eq:f-err}
f(x,y) - \fapx_k(x,y)
= 
\int_{0}^1 \frac{(1-t)^{k-1}}{(k-1)!} \Big( \Ty{k} f(x,y_t) - \Ty{k} f(x,\by)\Big) [(y-\by)^k] dt.
\end{equation}
Now observe that by~\cref{eq:tensy-xx}, for any~$y \in Y$ tensor~$\Txxy{k+2}{k} f(x,y)$ exists and satisfies~$\|\Txxy{k+2}{k} f(x,y)\| \le \cLykxx$ almost everywhere on~$X$. 
Fix~$y \in Y$,~$\by \in Y$, and~$x \in X$, and assume w.l.o.g. that~$x \in X$ is such that~$\Txxy{k+2}{k} f(x,y_t)$ exists (and hence~$\|\Txxy{k+2}{k} f(x,y_t)\| \le \cLykxx$) for all~$y_t \in [\by,y]$. Then, for any~$u,v \in \XX$,
\begin{equation}
\label{eq:bilinear-integral}
\lang u, \left( \hxx f(x,y) - \hxx \fapx_k(x,y) \right) v \rang 
=  
\int_{0}^1 \frac{(1-t)^{k-1}}{(k-1)!} \Big( \Txxy{k+2}{k} f(x,y_t) - \Txxy{k+2}{k} f(x,\by)\Big) \, [(y-\by)^k; u,v] \, dt.
\end{equation}
Whence, taking supremum over~$u,v$ on the unit sphere, by Jensen's inequality and~\cref{eq:tensy-xx} we arrive at
\begin{equation}
\label{eq:bilinear-integral-bounded}
\begin{aligned}
\| \hxx f(x,y) - \hxx \fapx_k(x,y) \|
&\le 
\left\| y-\by  \right\|^k \int_{0}^1 \frac{(1-t)^{k-1}}{(k-1)!} \left( \Big\| \Txxy{k+2}{k} f(x,y_t) \Big\| + \Big\| \Txxy{k+2}{k} f(x,\by)  \Big\| \right) dt 
\le \frac{2\cLykxx \Dy^k}{k!}.
\end{aligned}
\end{equation}
Finally, observe that these estimates remain valid, for any~$y \in Y$ and almost all~$x \in X$, even if~$\Txxy{k+2}{k} f$ is not guaranteed to exist ever\odima{y}where on~$X \times Y$. 
Indeed, for any~$x \in X$ define~$Y'_x$ as the set of all~$y' \in Y$ where~$\Txxy{k+2}{k} f(x,y')$ does not exist. 
Consider the graph of the set-valued map~$x \mapsto Y'_x$, i.e.
\[
\Gamma := \{(x,y'): \; x \in X, \, y' \in Y'_x \} \subset X \times Y.
\]
By~\cref{ass:measurable},~$\Gamma$ is~$(m_{X} \times m_{Y})$-measurable (here~$m_{X},m_{Y}$ are the Lebesgue measures on~$X$,$Y$).
Hence its restriction~$\Gamma^*$ on~$X \times [\by, y]$ is measurable with respect to the induced measure~$m^* = m_{X \times [\by,y]}$, and we can apply Fubini's theorem:
\[
m^*(\Gamma^*) 
= \int_{x \in X}  m_{[\by,y]} (Y_x') \, dx 
= \int_{y' \in [\by, y]} m_{X} ( X_{y'}) \, dy', \;\;
\text{where} \;\;
X_{y'} := \{ x \in X: y' \in Y'_x \}. 
\]
By~\cref{eq:tensy-xx} we have~$m_{X} ( X_{y'}) = 0$ for any~$y' \in Y$, whence~$m^*(\Gamma^*) = 0$ by the second representation of~$m^*(\Gamma^*)$, and therefore~$m_{[\by,y]} (Y_x') = 0$ for almost all~$x \in X$ (by the first representation of~$m^*(\Gamma^*)$). 
This shows that, for any choice of~$y,\by \in Y$, identity~\cref{eq:bilinear-integral} is valid for almost all~$x$ (the integrand exists almost everywhere on~$[0,1]$), and so the final estimate is preserved.
\qed

\vspace{-0.2cm}
\odima{
\paragraph{Lipschitzness in~$y$.}
We can assume w.l.o.g.~that~$k \ge 1$. Let us now show that for all~$(x,y) \in X \times Y$,
\begin{equation}
\label{eq:hxy-err}
\|\hxy \fapx_k(x,y) - \hxy f(x,y) \| \le \frac{2\cLykx \Dy^{k-1}}{(k-1)!}.
\end{equation}
To this end,~$\forall y,w \in Y$ let
$
\psi_{x,y,w}(t) := \lang \gy f(x,y_t), w \rang.
$
Then~$\psi_{x,y,w}^{(j)}(t) = \Ty{j+1} f(x,y_t)[(y-\by)^{j}; w]$ for~$j \ge 1$, therefore~$\langle \gy \fapx_k(x,y), w \rangle = \sum_{j = 0}^{k-1} \frac{1}{j!}\psi_{x,y,w}^{(j)}(0)$ and
\[
\lang \gy f(x,y) - \gy \fapx_k(x,y), w \rang 
= \int_{0}^1 \frac{(1-t)^{k-1}}{(k-1)!} \, \psi_{x,y,w}^{(k)}(t) \, dt 
=  \int_{0}^1 \frac{(1-t)^{k-2}}{(k-2)!} \, \left( \psi_{x,y,w}^{(k-1)}(t) - \psi_{x,y,w}^{(k-1)}(0) \right)\, dt;
\]
here the first integral exists by Assumption~\ref{ass:gradx}, and for the second one we integrated by parts. 
Thus,
\[
\lang u, \left(\hxy f(x,y) - \hxy \fapx_k(x,y) \right) w \rang 
=
\int_{0}^1 \frac{(1-t)^{k-2}}{(k-2)!} \Big( \Txy{k+1}{k} f(x,y_t) - \Txy{k+1}{k} f(x,\by)\Big) \, [(y-\by)^{k-1}; u, w] \, dt.
\]
Whence, by proceeding as in~\cref{eq:bilinear-integral-bounded} but this time using~\eqref{eq:tensy-lip} rather than~\eqref{eq:tensy-xx}, we arrive at~\eqref{eq:hxy-err}.
\qed
}

\subsection{Justification of~\cref{eq:eps-for-lower-order}}
\label{app:eps-for-lower-order-proof}
We first consider the case~$k > 1$. 
Recall that we have to show that condition
$
\min\big\{\Lxy\Dy, \sqrt{\Lxx \rho_k \Dy^{k+1}} \big\} \lsim_k \veps,
$
for suitable constant factors that might depend on~$k$, implies \cref{eq:upper-bound-k} under~\cref{eq:eps-for-lower-order}---i.e., provided that
\[
\veps \lsim_k \min\left\{ \left(\frac{\mu^{k}}{\sigma_k}\right)^{\frac{1}{k-1}}, \left(\frac{\lam^{2k+1} \rho_k^{k}}{\tau_k^{k+1}} \right)^{\frac{1}{2k}}\right\}.
\]
It suffices to show that~\cref{eq:eps-for-lower-order} implies~$\cLykx \Dy^k \lsim_k \mu \Dy$ when~$\mu \Dy \lsim \veps$, and~$\cLykxx \Dy^{k} \lsim_k \lam$ when~$\sqrt{\Lxx \rho_k \Dy^{k+1}}\lsim_k \veps$.
The first of these implications follows from the first part of~\cref{eq:eps-for-lower-order}: indeed, under its premise we have
\[
\mu \Dy \lsim \veps \lsim_k \left(\frac{\mu^{k}}{\sigma_k}\right)^{\frac{1}{k-1}},
\]
whence~$\cLykx \Dy^k \lsim_k \mu \Dy$ follows by taking power~$k - 1 > 0$. 
For the second implication, the premise gives
\[
\Lxx \rho_k \Dy^{k+1} \lsim_k \veps^2 \lsim_k \left(\frac{\lam^{2k+1} \rho_k^{k}}{\tau_k^{k+1}} \right)^{\frac{1}{k}},
\]
whence~$\Dy^{k+1} \lsim_k \left({\lam}/{\tau_k} \right)^{\frac{k+1}{k}},$ that is~$\cLykxx \Dy^{k} \lsim_k \lam$ by taking power~$\frac{k}{k+1} > 0$. 
Both implications are proved. 

Finally, in the case~$k = 1$ the first implication holds trivially, as~$\sigma_1 = \mu$~w.l.o.g.
On the other hand, our previous argument for the second implication applies here as well (since~$\frac{k}{k+1} > 0$).
\qed

\section{Proofs for Section~\ref{sec:lower}}
\label{sec:lower-proofs}

\subsection{Proof of~\cref{lem:hard-properties}}
\proofstep{1}.
The first claim is obvious for~$F_{k,s,\lam,\mu,\rho}$. For~$S_{\lam,\rho,\Dy}$, as~$0 \le \tanh'(x) \le 1$ and~$-1 \le \tanh''(x) \le 1$, %
\[
\begin{aligned}
\frac{\partial^2}{\partial x^2} S_{\lam,\rho,\Dy}(x,y)           \; &= \frac{\lam}{2} \, \left( -1 + \frac{y}{\Dy} \tanh''\middle(\sqrt{\frac{\Lxx}{\rho \Dy}} \, x\middle) \right) \in [-\lam, 0],\\
\frac{\partial^2}{\partial x \partial y} S_{\lam,\rho,\Dy}(x,y)  \; &= \frac{1}{2} \sqrt{\frac{\Lxx \rho}{\Dy}} \tanh' \left(\sqrt{\frac{\Lxx}{\rho \Dy}} \, x\right) 
\in \bigg[0, \frac{1}{2} \sqrt{\frac{\Lxx \rho}{\Dy}} \bigg] \subseteq \left[0,\frac{\mu}{2\sqrt{2}}\right].
\end{aligned}
\]
\proofstep{2}.
For the second claim, first note that~$\frac{\partial}{\partial y} F_{0,s,\lam,\mu,0} = \mu x$ and~$\frac{\partial}{\partial x} F_{0,s,\lam,\mu,0} = \mu y - \lam x$.
So if~$\mu \le \sqrt{2\lam \rho / \Dy}$, then we have on~$[-r,r] \times [-\Dy/2, \Dy/2]$ that
$
\left|\frac{\partial}{\partial y} F_{0,0,\lam,\mu,0}\right| \le \frac{\mu^2 \Dy}{2\lam} \le \rho
$
and 
$
\left|\frac{\partial}{\partial x} F_{0,0,\lam,\mu,0} \right| \le \mu \Dy.
$
On the other hand, if~$\mu \ge \sqrt{2\lam \rho/\Dy}$, we have on~$[-r,r] \times [0,\Dy]$ that
$
\frac{\partial}{\partial y} S_{\lam,\rho,\Dy} = \frac{\rho}{2} \left( \tanh\left(\sqrt{\frac{\Lxx}{\rho \Dy}} \, x\right) - 1 \right)  \in [-\rho, 0];
$
meanwhile,
$
\frac{\partial}{\partial x} S_{\lam,\rho,\Dy} = \frac{y}{2} \sqrt{\frac{\lam\rho}{\Dy}} \tanh' \left(\sqrt{\frac{\Lxx}{\rho \Dy}} \, x\right)  - \frac{\lam x}{2}
\le \frac{1}{2}\sqrt{\lam\rho\Dy} + \frac{\mu \Dy}{4} < \mu \Dy,
$
where we used~$-1 \le \tanh(x) \le 1$ and~$0 \le \tanh'(x) \le 1$.
The second claim is thus verified. 
The third claim is straightforward. \qed

\subsection{Proof of~\cref{prop:lower-bound-0}}
\proofstep{1}. 
For the first claim, let~$f(x,y) = -\frac{1}{2}\lam x^2 + \mu xy$,~$X = [-r,r]$,~$Y = [\Ry,\Ry]$ where~$\Ry = \Dy/2$ and
$r := {\mu \Ry}/{\lam}.$
Clearly,
$
\vphi(x) = -\frac{1}{2}\lam x^2 + \mu \Ry |x| = \lam (-\tfrac{1}{2}x^2 + r |x|),
$
whence
$
\vphi_{2\lam}(x) = \lam\min_{u \in X} \{(u - x)^2 - \frac{1}{2} u^2 + r|u| \}.
$
Let~$x^+(x)$ be the (constrained) minimizer.
The unconstrained minimizer is given by~$[2x]_r$, where
\begin{equation}
\label{def:soft-thresholding}
[z]_r := (\max\{|z|,r\}-r)\sign(z)
\end{equation}
is the soft\odima{-}thresholding operator. %
Now, observe that whenever~$|x| \le r$, one has~$|[2x]_r| \le r$, and therefore~$x^+(x) = [2x]_{r}$. 
Whence by~\cref{eq:moreau-explicit} (cf.~also~\cref{prop:moreau-to-primal} in appendix) for any~$x \in X$ we have
\begin{equation}
\label{eq:moreau-grad-0}
\vphi_{2\lam}'(x) = 2\lam(x - [2x]_r).
\end{equation}
On the other hand, we have that~$\vphiapx(x) = f(x,\by)$ and thus
$
\vphiapx_{2\lam}(x) = \lam \min_{u \in X} \{(u-x)^2 -\frac{1}{2} u^2 + u{\mu \by}/{\lam} \}.
$
Here the unconstrained minimizer is given by~$2x-\mu\by/\lam$, hence~$\hat x^+(x) = 2x-\mu\by/\lam$ for the actual (constrained) minimizer as long as~$|2x-\mu\by/\lam| \le r$.
Now, let us choose~$x^* = {r}/{2}$ and~$\by = {\Ry}/{2}.$
Clearly,~$2x^*-\mu\by/\lam = r/2$, therefore
$
\vphiapx_{2\lam}'(x^*) = 2\lam(x^* - \hat x^+(x^*)) = 0.
$
Meanwhile, due to~\cref{eq:moreau-grad-0} we have
\[
\vphi_{2\lam}'(x^*) = 2\lam(x^* - [2x^*]_r) = 2\lam(\tfrac{1}{2}r - [r]_r) = \lam r = \mu \Ry.
\]

\proofstep{2}.
Now consider
$
f(x,y) = S_{\lam,\rho,\Dy}(x,y),
$ 
cf.~\cref{eq:hard-sigmoid}, on~$[-r,r] \times [0,\Dy]$. 
Since~$|\tanh(\cdot)| \le 1$ on~$\R$, we have~$\vphi(x) = f(x,0) = -\frac{1}{4} \lam x^2$ and~$\vphi_{2\lam}(x) = \lam \min_{-r \le u \le r} \left\{ (u-x)^2 - \tfrac{1}{4} u^2 \right\}$, with  unconstrained minimizer given by~$\tfrac{4}{3}x$. 
Hence, as long as~$|x| \le 3r/4$, we have
$
\vphi_{2\lam}(x) = -\lam x^2/3
$
and
$
\vphi_{2\lam}'(x) = -2\lam x/3.
$

On the other hand, for the choice~$\by = 2\Dy/3$ we have
\[
\vphiapx(x) = f(x,\by) = -\frac{\lam x^2}{4} + \frac{\rho \Dy}{3} \left( \tanh\left(\sqrt{\frac{\Lxx}{\rho \Dy}} \, x \right)-1 \right).
\]
Since~$\vphiapx(\cdot)$ is differentiable on~$X = [-r,r]$, the set of its points on~$X$ with vanishing derivative coincides with such set for~$\vphiapx_{2\lam}(\cdot)$. 
Let us now find~$x^* \in [-r,r]$ for which~$\vphiapx'(x^*) = 0$, i.e., solutions to
\[
\sqrt{\frac{\lam}{\rho \Dy}} x^* \cosh^2\left(\sqrt{\frac{\lam}{\rho \Dy}}x^*\right) = \frac{2}{3}.
\]
The unique solution is~$x^* = c\sqrt{\rho\Dy/\lam}$ with~$c \in (0.51, \; 0.52)$. 
Using that~$\mu \ge \sqrt{2\lam \rho/\Dy}$, we \odima{conclude that
$
x^* < 0.52 \mu \Dy/(\sqrt{2} \lam) < 3 \mu \Dy/(8\lam) = {3r}/{4},
$
whence finally~$\vphi_{2\lam}'(x^*) = -2\lam x^*/3 < -\sqrt{\lam \rho \Dy}/3$.} 
\qed

\subsection{Proof of~\cref{prop:lower-bound-1}}
\proofstep{1}. 
Assume that~$\mu \le \sqrt{\lam\rho/2}$.
Recall that
$
f(x,y) = -\frac{1}{2}\lam x^2 + \mu xy - \frac{1}{2}\rho y^2.
$
Hence for~$\by = 0$ we have~$\fapx(x,y) = -\frac{1}{2}\lam x^2 + \mu xy$ and~$\vphiapx(x) = \lam (-\frac{1}{2} x^2 + r |x|)$
where~$r = {\mu \Ry}/{\lam}$ with~$\Ry = \half\Dy$ \odima{(as in~\proofstep{1} of the previous proof).}
\odima{As such,}~$\vphiapx_{2\lam}'(x) = 2\lam([2x]_r-x)$, and~$x^* = r$ \odima{is} a stationary point for~$\vphiapx_{2\lam}$ (cf.~\cref{eq:moreau-grad-0}).

Meanwhile, the maximum in
$
\vphi(x) = -\frac{1}{2} \lam x^2 + \max_{|y| \le \Ry} \{ \mu xy - \frac{1}{2} \rho y^2 \}
$
is effectively unconstrained---attained at~$\mu x/\rho$---whenever~$|x| \le  \rho \Ry/\mu$; 
thus, for such~$x$ we have
$
\vphi(x) = 
\frac{1}{2}\left({\mu^2}/{\rho} -\lam  \right) x^2
$
and hence
\[
\vphi'(x) = ({\mu^2}/{\rho} - \lam)x.
\]
Moreover, by the first-order optimality condition~$2\lam(x - x^+(x)) \in \partial\vphi'(x^+(x))$, cf.~\cref{eq:moreau}, we express the proximal mapping~$x^+(x) = x^+_{{\vphi}/{2\lam}}(x)$ as the solution to~$2\lam(x^+(x) - x) + \left({\mu^2}/{\rho} - \lam \right) x^+(x) = 0,$ that is
\[
x^+(x) = \frac{2\lam\rho  x}{\lam\rho + \mu^2}, 
\quad \text{whenever} \quad
\frac{2\lam\rho  |x|}{\lam\rho + \mu^2} \le \frac{\rho \Ry}{\mu}.
\]
In particular, the above expression is valid for~$x^* = r$: indeed, recalling that~$\mu \le \sqrt{\lam\rho/2}$, we obtain
\[
\frac{2\lam \rho r}{\lam\rho + \mu^2}  < 2r = \frac{2\mu \Ry}{\lam} \le \frac{\rho \Ry}{\mu}.
\] 
To verify the first claim of the proposition, it remains to observe that
\[
-\vphi_{2\lam}'(r) 
= 2\lam(x^+(r) - r) 
= 2\lam r \left(\frac{2\lam \rho}{\lam\rho+\mu^2} - 1\right) 
= 2\lam r \, \frac{\lam\rho - \mu^2}{\lam\rho + \mu^2}
\ge \frac{2\lam r}{3} = \frac{2\mu \Ry}{3},
\]
where the inequality is due to~$\mu^2 \le \lam\rho/2$.

\proofstep{2}.
We shall now prove the second claim. 
Define~$\bar \rho = \rho/4$,~$\bar \mu := \sqrt{\lam \rho/2} = \sqrt{2\lam\bar\rho}$, and
$\bar r := {\bar\mu \Ry}/{\lam}$,
so that~$f(x,y) = F_{1,1,\lam,\bar\mu,\bar\rho} = -\frac{1}{2}\lam x^2 + \bar\mu xy + \frac{1}{2}\bar\rho y^2$.
As such,
$
\vphi(x) = \lam \left(-\tfrac{1}{2}x^2 + \bar r |x| \right) + \tfrac{1}{2} \bar\rho \Ry^2
$
and
$
\vphi_{2\lam}(x) = \lam\min_{u \in \R} \{(u - x)^2 - \frac{1}{2}u^2 + r|u| \} + \frac{1}{2} \bar\rho \Ry^2.
$
This implies the same result as in~\cref{eq:moreau-grad-0}, namely
\begin{equation}
\label{eq:moreau-grad-1}
\vphi_{2\lam}'(x) = 2\lam(x - [2x]_{\bar r}).
\end{equation}
On the other hand,~$\fapx \equiv \fapx_1$ at any~$\by \in Y$ is given by
$
\fapx(x,y) 
= -\frac{1}{2}\lam x^2  + (\bar \mu x + \bar\rho \by)y - \frac{1}{2} \bar\rho \by^2;
$
in particular, for~$\by = \Ry$ we have~$\fapx(x,y) = -\frac{1}{2}\lam x^2  + (\bar \mu x + \bar\rho \Ry)y - \frac{1}{2}\bar\rho \Ry^2$ and
$
\vphiapx(x) = -\frac{1}{2}\lam x^2  + \Ry |\bar \mu x + \bar\rho \Ry| - \frac{1}{2}\bar\rho \Ry^2,
$
thus
\begin{equation}
\label{eq:moreau-shifted}
\vphiapx_{2\lam}(x) = \lam\min_{u \in \R} \left\{(u - x)^2 - \frac{u^2}{2} + \bar r\left|u + \frac{\bar\rho \Ry}{\bar \mu}\right| \right\} - \frac{\bar\rho \Ry^2}{2}.
\end{equation}
By the optimality condition, the minimizer is given by
$\hat x^+(x) = \left[ 2x + {\bar\rho \Ry}/{\bar \mu} \right]_{\bar r} - {\bar\rho \Ry}/{\bar \mu},$
so we arrive at
\begin{equation}
\label{eq:moreau-shifted-grad}
\vphiapx_{2\lam}'(x) = 2\lam\left(x + \frac{\bar\rho \Ry  }{\bar \mu} - \left[ 2x + \frac{\bar\rho \Ry}{\bar \mu} \right]_{\bar r} \right).
\end{equation}
We conclude that~$x^* = -{\bar\rho \Ry}/{\bar \mu}$ is stationary for~$\vphiapx_{2\lam}(\cdot)$: indeed, plugging in~$\bar\mu = \sqrt{2\lam \bar\rho}$ we have
$
-\vphiapx_{2\lam}'(x^*) = 2\lam \left[-{\bar\rho \Ry}/{\bar \mu} \right]_{\bar r} = 2\lam \left[-{\bar r}/{2} \right]_{\bar r} = 0.
$
Meanwhile, due to~\cref{eq:moreau-grad-1} we conclude that
\[
-\vphi_{2\lam}'(x^*) 
= 2\lam \left( \frac{\bar\rho \Ry}{\bar \mu} + \left[-\frac{2\bar\rho \Ry}{\bar \mu} \right]_{\bar r} \right) 
= 2\lam \left( \frac{\bar\rho \Ry}{\bar \mu} + \left[-\bar r\right]_{\bar r} \right) 
= \frac{2\lam \bar\rho \Ry}{\bar \mu} 
= \sqrt{2\lam \bar\rho\Ry^2}  
= \sqrt{\frac{\lam \rho\Dy^2}{8}}.
\quad\qed
\]

\subsection{Proof of~\cref{prop:lower-bound-k}}

\subsubsection{Case~$\mu \le \mucrit$}
Here~$Y = [0,\Dy]$ and
$
f(x,y) = F_{k,-1,\lam,\mu,\rho}(x,y) = -\frac{\lam x^2}{2} + \bar\mu x y - \frac{\rho |y|^{k+1}}{(k+1)!},
$
with some~$\bar \mu \le \mu$ yet to be chosen.
Let~$\by = 0$; then~$\fapx(x,y) = -\frac{1}{2} \lam x^2 + \bar\mu x y$ is maximized on~$\{0,\Dy\}$, \odima{thus}
$
\vphiapx(x) = -\frac{\lam x^2}{2}  + \bar\mu \Dy \max\{x,0\}.
$
\odima{As such,} the point
$
x^* = \frac{\bar\mu \Dy}{\lambda}
$
is stationary for~$\vphiapx(\cdot)$, and thus for~$\vphiapx_{2\lam}(\cdot)$ as well.
It remains to lower-bound~$|\vphi_{2\lam}'(x^*)|$. 
To this end, note that
$
\frac{\partial}{\partial y}f(x,y) = \bar\mu x - \frac{1}{k!}{\rho |y|^{k}\sign(y)},
$
so~$f(x,\cdot)$ has a unique unconstrained maximizer~$\bar y = \bar y(x)$ which is given as the solution to
$
\bar\mu x - \frac{1}{k!} \rho |y|^{k}\sign(y)= 0;
$
in other words,
$
\bar y(x) = \left({|x| \bar\mu k!}/{\rho} \right)^{1/k} \sign(x).
$
Clearly, we have that~$\vphi(x) = f(x,\bar y(x))$ for any~$x\in \R$ such that~$\bar y(x) \in [0,\Dy]$---in other words, when
\begin{equation}
\label{eq:strong-coupling-k-to-check-0}
0 \le \bar \mu x \le \frac{\rho \Dy^k}{k!}.
\end{equation}
As a result, for such~$x$ function~$\vphi(x)$ is differentiable, and we have
\[
\vphi(x) = -\frac{\lam x^2}{2} + \frac{k}{k+1} \left(\frac{x \bar\mu k!}{\rho}\right)^{{1}/{k}} \mu x,
\quad
\vphi'(x) = -\lam x + \bar\mu \left(\frac{x \bar\mu k!}{\rho}  \right)^{1/k}.
\]
Now, recall (cf.~\cref{eq:moreau-explicit}) that we have
$
2\lam(x - x^+(x)) \in \partial \vphi(x^+(x))
$
for the proximal mapping~$x^+(x)$ corresponding to~$\vphi$ with stepsize~$\frac{1}{2\lam}$. 
Therefore, we can compute~$x^+(x)$ for given~$x \ge 0$ by solving
\begin{equation}
\label{eq:strong-coupling-x-plus-eq}
x^+ +  \frac{\bar\mu}{\lam} \left(  \frac{x^+ \bar\mu k!}{\rho} \right)^{1/k} = 2 x
\end{equation}
for~$x^+ \ge 0$ (such a solution is clearly unique) and verifying that the solution satisfies
$
\bar\mu x^+ \le \frac{1}{k!}\rho \Dy^k
$
(cf.~\cref{eq:strong-coupling-k-to-check-0}).
It is clear that, for any~$x \ge 0$, the corresponding solution~$x^+(x)$ to~\cref{eq:strong-coupling-x-plus-eq} satisfies the bounds
\begin{equation}
\label{eq:strong-coupling-k-bracketing}
2x -  \frac{\bar\mu}{\lam} \left(  \frac{2x \bar\mu k!}{\rho} \right)^{1/k} \le x^+(x) \le 2x.
\end{equation}
To this end, let~$x^+ = x^+(x^*)$ for~$x^* = \frac{\bar\mu \Dy}{\lam}$, and~$\bar \mu = \frac{1}{2}\mu \, [\le \frac{1}{2} \mucrit]$. 
By~\cref{eq:mu-crit-k} and the upper bound in~\cref{eq:strong-coupling-k-bracketing},
\[
\bar\mu x^+ \le \frac{2 \bar\mu^2 \Dy}{\lam} %
\le \frac{1}{2} \frac{\rho \Dy^k}{k!}, 
\]
which verifies our characterization of~$x^+$ as the solution to~\cref{eq:strong-coupling-x-plus-eq} for chosen~$x^*$.
On the other hand, %
\[
\frac{\bar\mu}{\lam} \left(  \frac{2x^* \bar\mu k!}{\rho} \right)^{1/k} 
=  \frac{\bar\mu \Dy}{\lam} \left(\frac{2\bar\mu^2 k!}{\lam \rho \Dy^{k-1}} \right)^{1/k} 
\le \frac{x^*}{2^{1/k}},
\]
\odima{whence by~\cref{eq:strong-coupling-k-bracketing} we get~$x^+ \ge (2-2^{-1/k})x^*$ and, by~\cref{eq:moreau-explicit},
$
-\vphi_{2\lam}'(x^*) = 2\lam (x^+ - x^*) \ge (1 - 2^{-1/k} ) \mu\Dy > \frac{\mu \Dy}{2k}.
$
Here in the final step we used~$k \ge 2$ and the fact that the function~$t \mapsto (1-\frac{1}{t})^t$ increases on~$[1,+\infty]$.}
\subsubsection{Case~$\mu \ge \mucrit$}
Denote~$\Ry = \Dy/2$ and let
$
\bar r := \frac{\bar\mu \Ry}{\lam}
$
for some~$\bar \mu \le \mucrit \,[\le \mu]$ yet to be chosen.
Recall that here~$Y = [-\Ry, \Ry]$, 
\[
f(x,y) = F_{k,1,\lam,\bar\mu,\rho}(x,y)  = -\frac{\lam x^2}{2} + \bar\mu xy + \frac{\rho |y|^{k+1}}{(k+1)!}.
\]
Clearly
$
\vphi(x) = \lam (-\tfrac{1}{2}x^2 + \bar r |x| ) + \tfrac{\rho \Ry^{k+1}}{(k+1)!},
$
therefore
$
\vphi_{2\lam}(x) = \lam\min_{u \in \R} \{(u - x)^2 - \frac{1}{2} u^2 + \bar r|u| \} + \frac{\rho \Ry^{k+1}}{(k+1)!} 
$
and
\begin{equation}
\label{eq:moreau-grad-k}
\vphi_{2\lam}'(x) = 2\lam(x - [2x]_{\bar r}),
\end{equation}
cf.~\cref{eq:moreau-grad-1}. 
Now, note that~$g(y) = \frac{1}{(k+1)!}|y|^{k+1}$ is~$k$ times continuously differentiable with~$j$-th derivative
\begin{equation}
\label{eq:g-derivatives}
g^{(j)}(y) 
= \frac{|y|^{k+1-j}\sign(y)^j}{(k+1-j)!}, \quad j \le k.
\end{equation}
Whence by the binomial formula we conclude that, for any~$\by \in [0,\Ry]$, 
\[
\begin{aligned}
\fapx(x,y) 
&= -\frac{\lam x^2}{2} + \bar\mu xy + \frac{\rho \left(y^{k+1}-(y-\by)^{k+1}\right)}{(k+1)!}.
\end{aligned}
\]
From now on, we consider two cases depending on the parity of~$k$ (the case of odd~$k$ being harder).

\paragraph{Case of even~$k$.}
In this case we choose~$\bar\mu = \mucrit$,~$\by = \Ry$, and observe that the resulting function~$\fapx(x,\cdot)$ is convex on~$[-\Ry,\Ry]$. 
Indeed, in terms of the rescaled variable~$z = y/\Ry$ we have that~$\fapx(x, y) = h(x,z)$ with
$
h(x,z) =  -\frac{1}{2}\lam x^2 + \bar\mu \Ry x z + \rho \Ry^{k+1} p(z)
$
and function~$p: [-1,1] \to \R$ given by
\begin{equation}
\label{eq:p-poly}
p(z) = \frac{z^{k+1}-(z-1)^{k+1}}{(k+1)!}.
\end{equation}
Let us verify that~$p$ is convex on~$[-1,1]$.
Indeed: on one hand, for~$z \in [-1,1]$ one has~$\frac{z}{z-1} \le 1/2$, thus~$( \frac{z}{z-1} )^{k-1} \le 1$
using that~$k-1$ is odd; on the other hand,~$(z-1)^{k-1} \le 0$ for~$z \in [-1,1]$. 
As a result,
\[
p''(z) = \frac{z^{k-1} - (z-1)^{k-1}}{(k-1)!} \ge 0, \quad \forall z \in [-1,1]. 
\]
As such,~$p(\cdot)$ is convex on~$[-1,1]$;~$h(x,\cdot)$ is convex on~$[-1,1]$ and maximized at an endpoint, so that
\[
\begin{aligned}
\vphiapx(x) 
= \max_{z = \pm 1} h(x,z) 
&= \lam \left( -\frac{x^2}{2} + \max \left\{ \bar r x + \frac{\rho \Ry^{k+1}}{\lam (k+1)!}, -\bar r x + \frac{(2^{k+1} - 1) \rho\Ry^{k+1}}{\lam (k+1)!}  \right\} \right) \\
&= \lam \left( -\frac{x^2}{2} + \bar r \left|  x - \frac{(2^{k} - 1) \rho \Ry^{k}}{\bar \mu (k+1)!}  \right| \right) + \frac{2^{k} \rho \Ry^{k+1}}{(k+1)!}
\end{aligned}
\]
and \odima{by~\eqref{eq:moreau-shifted}--\eqref{eq:moreau-shifted-grad},}
$
\vphiapx_{2\lam}'(x) = 2\lam\left( x - \frac{(2^{k} - 1) \rho \Ry^{k}}{\bar \mu (k+1)!} - \left[ 2x - \frac{(2^{k} - 1) \rho \Ry^{k}}{\bar \mu (k+1)!} \right]_{\bar r} \right).
$
Now: by~\cref{eq:mu-crit-k},
$
\odima{x^{\textsf{even}}_{k}} := \frac{ (2^{k} - 1) \rho \Ry^{k}}{\bar\mu (k+1)!} 
$
\odima{is such that}
$
\odima{x^{\textsf{even}}_k} \le \frac{\rho \Dy^k}{\bar\mu (k+1)!} \le \frac{2\bar r}{k+1}.
$
But then
$
[x^{\textsf{even}}_k]_{\bar r} = 0
$
and~$0 \le [2x^{\textsf{even}}_k]_{\bar r} \le [2x^{\textsf{even}}_k]_{\frac{3}{2}x^{\textsf{even}}_k} = \half x^{\textsf{even}}_k$. 
\odima{As a result,}~$\vphiapx_{2\lam}'(x^{\textsf{even}}_k) = 0$ and
$
\vphi_{2\lam}'(x^{\textsf{even}}_k) 
= 2\lam(x^{\textsf{even}}_k - [2x^{\textsf{even}}_k]_{\bar r}) 
\ge \lam x^{\textsf{even}}_k 
= \frac{(2^{k} - 1) \lam \rho \Ry^{k}}{\mucrit (k+1)!}
= \frac{(1-2^{-k}) \mucrit \Dy}{k+1} %
\ge \frac{\mucrit \Dy}{2k}.
$

\paragraph{Case of odd~$k$.}
Note that here we have~$k \ge 3$ by the premise of the theorem. 
We choose~$\by = (1-\frac{1}{k}) \Ry$, so that $\fapx(x,y) = h(x,z)$ with
$
h(x,z) =  -\frac{1}{2}\lam x^2 + \bar\mu \Ry x z + \rho \Ry^{k+1} q(z),
$
for~$z = y/\Ry$ and~$q(z)$ given by
\begin{equation}
\label{eq:q-poly}
q(z) = \frac{1}{(k+1)!} \left(z^{k+1}-\left(z-\left(1-\frac{1}{k}\right)\right)^{k+1} \right).
\end{equation}
We choose~$\bar \mu$ as follows:
\begin{equation}
\label{eq:mu-bar-odd-k}
\bar \mu = \sqrt{\frac{2\lam \rho \Dy^{k-1}}{k!} \left(1 - \frac{1}{2^{k}}\right) \left(1-\frac{1}{k}\right)^{k-1}}\odima{.}
\end{equation}
Since~$k \ge 3$, we have that~$\frac{1}{\sqrt{2}} \mucrit \le \bar \mu \le \mucrit$ (cf.~\cref{eq:mu-crit-k}).
Now, let us show that
$
\odima{x^{\textsf{odd}}_k} := -\left(1-\frac{1}{k}\right) \bar r
$ 
is a stationary point for~$\vphiapx$ (and thus also for~$\vphiapx_{2\lam}$). 
\odima{To this end, we first observe that}
\begin{equation}
\label{eq:h-derivative-k-odd}
\begin{aligned}
h'_z(x^{\textsf{odd}}_k, z) 
= \bar \mu \Ry x^{\textsf{odd}}_k + \rho \Ry^{k+1} q'(z) 
= -\left(1-\frac{1}{k}\right) \frac{\bar \mu^2 \Ry^2}{\lam}
	+ \frac{\rho \Ry^{k+1}}{k!}  \left(z^{k}-\left(z-\left(1-\frac{1}{k}\right)\right)^{k}\right). 
\end{aligned}
\end{equation}
\begin{itemize}
\item 
Let us now find all stationary points of~$h(x^{\textsf{odd}}_k, \cdot)$ on~$\R$. %
Plugging~\cref{eq:mu-bar-odd-k} into the right-hand side of~\cref{eq:h-derivative-k-odd} and dividing over~$\frac{1}{k!} \Ry^{k+1} \left(1-\frac{1}{k}\right)^k > 0$,
we arrive at the equation
\begin{equation}
\label{eq:odd-k-stationary-eq}
w^k - (w - 1)^k = 2^k-1
\end{equation}
in terms of~$w = \frac{k}{k-1} z$, and guess two solutions:~$w_1^* = -1$ and~$w_2^* = 2$. 
Moreover,~$w \mapsto w^k - (w - 1)^k$ is a strictly convex function (to see this, note that~$k-2 > 0$ is odd, and the function~$u \mapsto u^{k-2}$ increases on~$\R$ modulo the sole fixed point~$u = 0$), therefore~\eqref{eq:odd-k-stationary-eq} has no other (real) solutions.

\item 
Since~$k \ge 3$, we have~$1 < \frac{k}{k-1} \le \frac{3}{2}$, so only one of the two solutions, namely~$w_1^* = -1$, belongs to the range~$\left[-\frac{k}{k-1}, \frac{k}{k-1} \right]$ of~$w = w(z)$ corresponding to~$z \in [-1,1]$.
As such,
$
z^{\textsf{odd}}_k = -\left(1 -\frac{1}{k}\right)
$
is a unique stationary point of~$h(x^{\textsf{odd}}_k, \cdot)$ on~$[-1,1]$. 
Moreover, since~$k-1$ is even, we have that
\[
h^{''}_{zz}(x^{\textsf{odd}}_k, z^{\textsf{odd}}_k)
= \frac{\rho \Ry^{k+1}}{(k-1)!}  \left((z^{\textsf{odd}}_k)^{k-1}-\left(z^{\textsf{odd}}_k-\left(1-\frac{1}{k}\right)\right)^{k-1}\right)  
= (1-2^{k-1}) \left( 1 - \frac{1}{k} \right)^{k-1} \frac{\rho \Ry^{k+1}}{(k-1)!} < 0,
\]
so~$z^{\textsf{odd}}_k$ is a local maximizer---and hence also a unique global maximizer---of~$h(x^{\textsf{odd}}_k,\cdot)$ on~$[-1,1]$. %

\item
Finally, by a version of Danskin's theorem (see, e.g.,~\cref{lem:danskin} in appendix) we have that
\[
\vphiapx'(x^\textsf{odd}_k) 
= h'_x(x^\textsf{odd}_k, z^{\textsf{odd}}_k) 
= \bar\mu \Ry z^{\textsf{odd}}_k - \lam x^\textsf{odd}_k = \lam(\bar r z^{\textsf{odd}}_k - x^\textsf{odd}_k) 
= 0.
\]
\end{itemize}
We have just verified that~$x^\textsf{odd}_k$ is a stationary point for~$\vphiapx(\cdot)$. 
On the other hand, observe that
\[
\begin{aligned}
-\vphi_{2\lam}'(x^\textsf{odd}_k) 
= 2\lam([2x^\textsf{odd}_k]_{\bar r}-x^\textsf{odd}_k) 
&= 2\lam \left( \left(1-\frac{1}{k}\right) \bar r + \left[-2\left(1-\frac{1}{k}\right) \bar r \right]_{\bar r} \right) 
= \frac{2\lam \bar r}{k} 
= \frac{\bar \mu \Dy}{k}
> \frac{\mucrit \Dy}{2k}.
\end{aligned}
\qed
\]

\section{Proofs for Section~\ref{sec:algos}}
\label{sec:algos-proofs}

The proofs of~\cref{th:algo-0,th:algo-1} rely on the following technical result extracted from~\cite{ostrovskii2020efficient}.

\begin{lemma}[cf.~{\cite[Proposition~5.5]{ostrovskii2020efficient}}]
\label{lem:fne-to-moreau}
%
%
For~$x \in X,$~$\xi \in \XX$,~$\bar\lam > 0$ define functionals~$\Sx,\Sy$ by
\begin{equation}
\label{def:strong-measure}
\begin{aligned}
\Sx^2(x,\xi,\bar\lam) 
	&:= 2\bar\lam\max_{u \in X} \big\{-\langle \xi, u-x \rangle - \tfrac{\bar\lam}{2} \|u-x\|^2 \big\}
	\quad \text{for \, $x \in X,$~$\xi \in \XX$,\,~$\bar\lam > 0$}, \\
\Sy^2(y,\eta,\bar\rho) 
	&:= 2\bar\rho\max_{v \in Y} \big\{-\langle \eta, v-y \rangle - \tfrac{\bar\rho}{2} \|v-y\|^2 \big\}
	\quad \,\text{for \; $y \in Y,$\,~$\eta \in \YY$,\,~$\bar\rho > 0$}.
\end{aligned}
\end{equation}
\begin{enumerate}
\item
Under~\cref{ass:gradx}, for~$\fapx \equiv \fapx_0$, at any~$x \in X$ (and regardless of the choice of~$\by \in Y$) we have
\begin{equation}
\label{eq:fne-to-moreau-0}
\|\nabla \vphiapx_{2\Lxx}(x)\| \le 2 \Sx(x,  \gx f(x,\by), \lam).
\end{equation}
\item
Under~\cref{ass:gradx,ass:tensy,ass:measurable} with~$k = 1$, for~$\fapx \equiv \fapx_1$, at any~$x \in X$ and~$y \in Y$ we have
\begin{equation}
\label{eq:fne-to-moreau-1-gen}
\|\nabla \vphiapx_{2\Lxxp_1}(x)\| 
\le
2 \left( \Sx(x, \gx \fapx(x, y), \Lxxp_1) + \sqrt{\tfrac{\Lxx}{\rho_1}} \Sy(y, -\gy \fapx(x, y), \rho_1) + \sqrt{\Lxx \rho_1} \Dy \right).
\end{equation}
Moreover,
for~$y^\circ [= y^\circ(x)] \in \Argmax_{y \in Y} \fapx(x,y)$ we have
$\|\nabla \vphiapx_{2\Lxxp_1}(x)\| \le 2 \Sx(x,  \gx \fapx(x,y^\circ), \Lxxp_1)$.
\end{enumerate}
\end{lemma}
\begin{proof}
Inequality~\cref{eq:fne-to-moreau-1-gen} follows from~\cite[Proposition~5.5]{ostrovskii2020efficient} after recovering the constant factors from the proof (see~\cite[Appendix A]{ostrovskii2020efficient}), bounding the second term in the right-hand side of the resulting inequality (cf.~\cite[Eq.~(5.6)]{ostrovskii2020efficient}) via Cauchy-Schwarz, and extracting the square root.
For~$y^\circ = y^\circ(x)$ the corresponding term vanishes, and we get~$\|\nabla \vphiapx_{2\Lxxp_1}(x)\| \le 2 \Sx(x,  \gx \fapx(x,y^\circ), \Lxxp_1)$ directly from~\cite[Eq.~(5.6)]{ostrovskii2020efficient}. 
We obtain~\cref{eq:fne-to-moreau-0} by the same argument, as in this case~$\fapx \equiv \fapx_0$ is formally maximized at~$\by$ for any~$x \in X$.
Note that the argument cannot be extended to~$k \ge 2$: in this case~$\fapx_k(x,\cdot)$ is not concave, so~\cite[Proposition~5.5]{ostrovskii2020efficient} cannot be applied anymore.
\end{proof}
%

\subsection{Proof of~\cref{th:algo-0}}
We first observe that, with~$\gamx = {1}/{\lam}$, the quantity~$\veps_t$ computed in~\cref{line:eps_t-0} of~\cref{alg:zeroth-order} is nothing else but~$\Sx(x_t,  \gx\fapx(x_t,\by), \lam)$. 
Indeed: 
\begin{equation}
\label{eq:grad-norm}
\begin{aligned}
\Sx^2(x_t,  \gx f(x_t,\by), \lam) 
&= \tfrac{1}{\gamx^2} \max_{x \in X} \big\{-2 \langle \gamx \gx f(x_t,\by), x - x_t \rangle - \|x-x_t\|^2 \big\} \\
&= \| \gx f(x_t,\by) \|^2 -  \tfrac{1}{\gamx^2} \min_{x \in X} \| x_{t} - \gamx \gx f(x_{t},\by) - x \|^2 
= \veps_{t}^2.
\end{aligned}
\end{equation}
Thus,~\cref{alg:zeroth-order} maintains~$x^*$ such that
$
\Sx(x^*,  \gx f(x^*,\by), \lam) = \min_{0 \le t \le T-1} \Sx(x_t,  \gx f(x_t,\by), \lam).
$
On the other hand, by the descent lemma (cf.~\cref{ass:gradx}), for each~$t \in \{0, ..., T-1\}$ we have that
\[
\begin{aligned}
f(x_{t+1}, \by) 
&\le f(x_{t}, \by) + \langle \gx f(x_{t}, \by), x_{t+1} - x_t \rangle + \tfrac{\lam}{2} \|x_{t+1} - x_t\|^2 
= f(x_{t}, \by) - \tfrac{1}{2\lam}\Sx^2(x_t,  \gx f(x_t,\by), \lam),
\end{aligned}
\]
where the equality follows from the definition of~$x_{t+1}$.
Whence, via telescoping and~\cref{eq:zeroth-order-complexity} we get
\[
\begin{aligned}
\Sx^2(x^*,  \gx f(x^*,\by), \lam) 
&= \min_{0 \le t < T} \Sx^2(x_t,  \gx f(x_t,\by), \lam) 
\le \tfrac{2\lam}{T} \hspace{-0.05cm} \left[ f(x_0, \by) \hspace{-0.05cm}-\hspace{-0.05cm} f(x_T, \by) \right]
\le \tfrac{2\lam}{T} \hspace{-0.05cm} \left[ \vphi(x_0) \hspace{-0.05cm}-\hspace{-0.05cm} \psi(\by) \right] 
\hspace{-0.05cm}<\hspace{-0.05cm} \tfrac{\veps^2}{144}.
\end{aligned}
\]
By the first claim of~\cref{lem:fne-to-moreau}, cf.~\cref{eq:fne-to-moreau-0}, this gives~$\|\nabla\vphiapx_{2\Lxx}(x^*)\| \le 2\veps^* \le \frac{\veps}{6}$.
Finally, \odima{condition~$24\Lxy \Dy \le \veps$} implies the premise of~\cref{th:upper-bound} with~$k = 0$; applying it we arrive at~$\|\nabla\vphi_{2\Lxx}(x^*)\| \le \veps$.
\qed

\subsection{Proof of~\cref{th:algo-1}}

\proofstep{1}.
We first consider the case~$\mu \le (\Lxxp_1 \rho_1)^{1/2}$, so that~$\Coupled = 0$, and~\cref{eq:first-order-diameter} reduces to
$200\mu \Dy \le {\veps}.$
As~$\fapx := \fapx_1$,~\cref{line:y_t-weak} of~\cref{alg:first-order} reads~$y_t \in \Argmax_{y \in Y} \fapx(x_{t},y)$;  
then the second claim of~\cref{lem:fne-to-moreau} (with~$y_t = y^\circ$) guarantees that
\begin{equation}
\label{e:first-orde-fne-to-moreau}
\|\nabla \vphiapx_{2\Lxxp_1}(x_t)\| \le 2 \Sx(x_t,  \gx \fapx(x_t,y_t), \Lxxp_1).
\end{equation}
Let us now upper-bound~$\Sx(x_t,  \gx \fapx(x_t,y_t), \Lxxp_1)$. 
To this end, for~$\delta > 0$ to be chosen later, consider
\begin{equation}
\label{def:freg}
\begin{aligned}
f_{\reg}(x,y) := f(x,y) - \tfrac{1}{2} \delta \|y - \by\|^2, \quad
\vphi_{\reg}(x) := \max_{y \in Y} f_{\reg}(x,y).
\end{aligned}
\end{equation}
Since~$f_{\reg}(x,\cdot)$ is~$\delta$-strongly concave, by Danskin's theorem (cf.~\cite[Lem.~24]{nouiehed2019solving})~$\vphi_{\reg}$ is differentiable, 
and $
\nabla \vphi_{\reg}(x) = \gx f(x,y_{\reg}(x))$ with~$y_{\reg}(x) := \argmax_{y \in Y} f_{\reg}(x,y)$,
and is~$\lam_{\reg}$-Lipschitz with
$
\lam_{\reg} := \lam + \frac{\Lxy^2}{\delta}.
$
Now, let us choose~$\delta = {\mu^2}/{\Lxxp_1},$ so that~$\lam_{\reg} = \lam + \Lxxp_1 \in [\Lxxp_1, 2\Lxxp_1]$. 
By the descent lemma\odima{,} %
\begin{align}
\label{eq:first-order-ex}
&\vphi_{\reg}(x_{t+1}) - \vphi_{\reg}(x_{t}) \nn
&\le  \langle \nabla \vphi_{\reg}(x_{t}), x_{t+1} - x_t \rangle + \tfrac{1}{2} \lam_{\reg} \|x_{t+1} - x_t\|^2 
= \langle \gx f(x_{t},y_{\reg}(x_t)), x_{t+1} - x_t \rangle + \tfrac{1}{2} \lam_{\reg} \|x_{t+1} - x_t\|^2 \nn
&\le \langle \gx f(x_{t},y_t), x_{t+1} - x_t \rangle + \tfrac{3}{4}\lam_{\reg} \|x_{t+1} - x_t\|^2 + \tfrac{1}{\lam_{\reg}}\| \gx f(x_{t},y_{\reg}(x_t)) - \gx f(x_{t},y_t) \|^2 \nn
&\le \langle \gx f(x_{t},y_t), x_{t+1} - x_t \rangle + \tfrac{3}{4}\lam_{\reg} \|x_{t+1} - x_t\|^2 + \tfrac{1}{\lam_{\reg}} \mu^2 \Dy^2 \nn
&\le \langle \gx f(x_{t},y_t), x_{t+1} - x_t \rangle + \tfrac{1}{2}\left({3\Lxxp_1} + {\mu^2}/{\rho_1}\right) \|x_{t+1} - x_t\|^2 + \tfrac{1}{\Lxxp_1^{\vphantom{A}}}\mu^2 \Dy^2 \nn
&= \min_{x \in X} \left\{\langle \gx f(x_{t},y_t), x - x_t \rangle + \tfrac{1}{2}\left({3\Lxxp_1} + {\mu^2}/{\rho_1}\right) \|x - x_t\|^2 \right\} + \tfrac{1}{\Lxxp_1^{\vphantom{A}}}\mu^2 \Dy^2 \nn
&= -\tfrac{1}{6\Lxxp_1 + 2\mu^2/\rho_1} \Sx^2 \left(x_t,  \gx f(x_t,y_t), {3\Lxxp_1} + {\mu^2}/{\rho_1} \right) + \tfrac{1}{\Lxxp_1^{\vphantom{A}}} \mu^2 \Dy^2 \nn
&\le -\tfrac{1}{8\Lxxp_1} \Sx^2 \left(x_t,  \gx f(x_t,y_t), {3\Lxxp_1^{\vphantom{A}}} + {\mu^2}/{\rho_1} \right) + \tfrac{1}{\Lxxp_1^{\vphantom{A}}} \mu^2 \Dy^2.
\end{align}
Here the second inequality is by Cauchy-Schwarz, the next one is via~\cref{ass:gradx}, 
and the subsequent identities are by the definitions of~$x_{t+1}$ and~$\gamx$ (cf.~\cref{line:x_t-proj} of~\cref{alg:first-order}). (Note that the factor~$\mu^2/\rho_1$ can be upper-bounded with~$\Lxxp_1$ by the standing assumption, but we avoid this in order for our estimates to have a similar form as in the strongly coupled case (cf.~\proofstep{2}) which is yet to be considered.)
We now proceed as in~\cref{eq:grad-norm}: by telescoping~\cref{eq:first-order-ex} we get
\begin{align}
&\Sx^2 \left(x^*,  \gx f(x^*, y^*), 3\Lxxp_1 + {\mu^2}/{\rho_1} \right) 
= \min_{0 \le t \le T-1} \Sx^2 \left(x_t,  \gx f(x_t,y_t), 3\Lxxp_1 + {\mu^2}/{\rho_1} \right) \nn
&\le \tfrac{8}{T} \Lxxp_1 \left( \vphi_{\reg}(x_0) - \vphi_{\reg}(x_T) \right)  + 8\mu^2\Dy^2 
\le \tfrac{8}{T} \Lxxp_1 \left( \vphi(x_0) - \vphi(x_T) \right)  + 12\mu^2\Dy^2. 
\label{eq:first-order-telescoping}
\end{align}
where in the final step we plugged in~$\delta = {\mu^2}/{\Lxxp_1}$.
As a result,
\begin{align*}
&\Sx^2(x^*,  \gx \fapx(x^*,y^*), \Lxxp_1) \\
&\le 
\Sx^2(x^*,  \gx \fapx(x^*,y^*), 4\Lxxp_1)  = 8\Lxxp_1 \max_{x \in X} \left\{-\langle \gx \fapx(x^*,y^*), x - x^* \rangle - 2\Lxxp_1 \|x - x^*\|^2 \right\} \\ 
&\le
8\Lxxp_1 \max_{x \in X} \left\{-\langle \gx f(x^*,y^*), x - x^* \rangle  - \tfrac{3}{2}\Lxxp_1  \|x - x^*\|^2\right\} + 4\| \gx f(x^*,y^*) - \gx \fapx(x^*,y^*) \|^2 \\
&=
\tfrac{4}{3} \Sx^2(x^*,  \gx f(x^*,y^*), 3\Lxxp_1) + 4\| \gx f(x^*,y^*) - \gx \fapx(x^*,y^*) \|^2 \\
&\le 
\tfrac{4}{3} \Sx^2(x^*,  \gx f(x^*,y^*), 3\Lxxp_1) + 16\mu^2\Dy^2 
\le 
32 \left( \tfrac{\Lxxp_1}{3T}[\vphi(x_0) - \vphi(x_T)]  + \mu^2\Dy^2 \right)
\le
\left( \tfrac{11}{2100}  + \tfrac{33}{40000} \right) \veps^2 
<
\tfrac{1}{144} \veps^2. 
\end{align*}
Here the first estimate relies on the proximal Polyak-\L{}ojasiewicz (PL) inequality that ensures that~$\Sx(x,\xi,\bar\lam)$ is non-decreasing in~$\bar\lam$ (cf.~\cite[Lemma~1]{karimi2016linear}); the second is by Cauchy-Schwarz, the third is by~\cref{lem:gx-err}, the fourth is by~\cref{eq:first-order-telescoping}, and the last one is by~\eqref{eq:first-order-complexity} \odima{and} \odima{since~$200\mu \Dy \le {\veps}$ (cf.~\eqref{eq:first-order-diameter}).}
Finally, returning to~\cref{e:first-orde-fne-to-moreau} we get~$\|\nabla\vphiapx_{2\Lxxp_1}(x^*)\| \le  \frac{\veps}{6}$, whence~\cref{th:upper-bound} implies that~$\|\nabla\vphi_{2\Lxxp_1}(x^*)\| \le  \veps$. 

\proofstep{2}. 
We now consider the case~$\mu \ge (\Lxxp_1 \rho_1)^{1/2}$, where~\cref{alg:first-order} is run with~$\Coupled = 1$. 
By~\cref{eq:first-order-diameter},
\begin{equation}
\label{eq:first-order-strong-coupling-condition}
200\, \Dy \sqrt{\Lxxp_1\rho_1}\le \veps.
\end{equation}
Casting the update in~\cref{line:y_t-strong} of~\cref{alg:first-order} as
$
y_t = \argmax_{y \in Y} \{ \langle \gy f(x_{t},\by), y - \by \rangle - \frac{\rho_1}{2} \|y - \by\|^2 \}
$
gives
\[
\lang \gy f(x_t,\by) - \rho_1 (y_t - \by), y - y_t\rang \le 0, \quad \forall y \in Y,
\]
from to the first-order optimality condition.
As a result, for any~$0 \le t \le T-1$ we have
\begin{align*}
\Sy^2(y_t, -\gy \fapx(x_t, y_t), \rho_1) 
&= \Sy^2(y_t, -\gy f(x_t, \by),  \rho_1) 
\le 2\rho_1^2\max_{y \in Y} \left\{ - \tfrac{1}{2} \|y-y_t\|^2 + \langle y_t-\by, y-y_t \rangle \right\} 
\le \rho_1^2 \Dy^2. 
\end{align*}
where in the final step we maximized over the whole~$\YY$.
Whence by~\cref{lem:fne-to-moreau} (cf.~\cref{eq:fne-to-moreau-1-gen}) we get %
\begin{equation}
\label{eq:first-order-acc-incomplete}
\begin{aligned}
\|\nabla \vphiapx_{2\Lxxp_1}(x_t)\| 
&\;\le\; 2 \Sx(x_t,  \gx \fapx(x_t,y_t), \Lxxp_1) + 4 \Dy \sqrt{\Lxx \rho_1}. 
\end{aligned}
\end{equation}
Our next goal is to estimate~$\Sx(x_t,  \gx \fapx(x_t,y_t), \Lxxp_1)$ via a telescoping argument. 
To this end, let
\begin{equation}
\label{eq:freg-apx}
\begin{aligned}
\wt f_\reg(x,y) := \fapx_1(x,y) - \tfrac{1}{2} \rho_1 \|y - \by\|^2, \quad
\wt\vphi_{\reg}(x) := \max_{y \in Y} \wt f_{\reg}(x,y).
\end{aligned}
\end{equation}
Recall that~$\gx \fapx_1(\cdot,y)$ is~$\Lxxp_1$-Lipschitz for any~$y \in Y$ (cf.~\cref{lem:gx-lip}); moreover,~$\wt f_{\reg}(x,\cdot)$ is~$\rho_1$-strongly concave. 
Therefore by Danskin's theorem (cf.~\cite[Lemma~24]{nouiehed2019solving}),
$
\nabla \wt\vphi_{\reg}(x) = \gx \fapx(x, \wt y_{\reg}(x))
$
with~$\wt y_{\reg}(x) := \argmax_{y \in Y} \wt f_{\reg}(x,y)$, and is~$\wt\lam_{\reg}$-Lipschitz with
$
\wt\lam_{\reg} := \Lxxp_1 + {\Lxy^2}/{\rho_1}.
$
Moreover, we have
\[
y_t = \wt y_{\reg}(x_t),
\]
as seen by looking at~\cref{line:y_t-strong} of~\cref{alg:first-order} again.
Whence, by the descent lemma (cf.~\cref{eq:first-order-ex}) we get
\begin{align*}
\wt\vphi_{\reg}(x_{t+1}) - \wt\vphi_{\reg}(x_{t}) 
&\le \langle \gx \fapx(x_{t},y_t), x_{t+1} - x_t \rangle + \tfrac{1}{2} \wt\lam_{\reg} \|x_{t+1} - x_t\|^2 \\
&\le \langle \gx \fapx(x_{t},y_t), x_{t+1} - x_t \rangle + \tfrac{1}{2} \left(3\Lxxp_1 + {\mu^2}/{\rho_1} \right) \|x_{t+1} - x_t\|^2 \\
&= \min_{x \in X} \left\{\langle \gx \fapx(x_{t},y_t), x - x_t \rangle + \tfrac{1}{2}\left({3\Lxxp_1} + {\mu^2}/{\rho_1}\right) \|x - x_t\|^2 \right\}\\
&\le -\tfrac{1}{6\Lxxp_1 + 2\mu^2/\rho_1} \Sx^2 \left(x_t,  \gx \fapx(x_t,y_t), 3\Lxxp_1 + {\mu^2}/{\rho_1} \right).
\end{align*}
Here the penultimate line follows by recasting the update in~\cref{line:x_t-strong} of~\cref{alg:first-order} as
\[
\wt x_{t+1} 
\left[= x_{t} - \gamx \gx f(x_{t},\by) - \gamx \hxy f(x_{t},\by) (y_{t} - \by) \right]
= x_{t} - \gamx \gx \fapx(x_{t},y_t).
\] 
Moreover, for the same reason we have (cf.~\cref{eq:grad-norm})
\begin{equation}
\begin{aligned}
\veps_{t}^2 
&= \| \gx \fapx(x_t,y_t) \|^2 -  \tfrac{1}{\gamx^2} \min_{x \in X} \| x_{t} - \gamx \gx \fapx(x_{t},y_t) - x \|^2 
= \Sx^2\left(x_t,  \gx \fapx(x_t,y_t), 3\Lxxp_1+{\mu^2}/{\rho_1} \right). 
\end{aligned}
\end{equation}
Whence, proceeding as in~\cref{eq:first-order-telescoping} we get
\begin{align*}
\Sx^2 \left(x^*, \gx \fapx(x^*, y^*), 3\Lxxp_1 + {\mu^2}/{\rho_1} \right) 
&\le \tfrac{1}{T} \left(6\Lxxp_1  + {2\mu^2}/{\rho_1} \right) \left(\wt\vphi_{\reg}(x_0) - \wt\vphi_{\reg}(x_T) \right)\nn
&\stackrel{(i)}{\le} \tfrac{2}{T}\left(3\Lxxp_1  + {\mu^2}/{\rho_1} \right) \left( \vphiapx(x_0) - \vphiapx(x_T) + \tfrac{1}{2} \rho_1 \Dy^2 \right)\nn
&\stackrel{(ii)}{\le} \tfrac{2}{T}\left(3\Lxxp_1  + {\mu^2}/{\rho_1} \right) \left( \vphi(x_0) - \vphi(x_T) + \tfrac{3}{2} \rho_1 \Dy^2 \right).
\end{align*}
Here for~$(i)$ we used the two inequalities~$\wt\vphi_{\reg}(x_0) - \wt\vphi_\reg(x_T) \le \vphiapx(x_0) - \max_{y \in Y} \{\fapx_1(x_T,y) - \frac{1}{2} \rho_1 \|y - \by\|^2\}$ 
and~$ \vphiapx(x_T) \le \max_{y \in Y} \{\fapx_1(x_T,y) - \frac{1}{2} \rho_1 \|y - \by\|^2\} + \frac{1}{2} \rho_1 \Dy^2$ (cf.~\cref{eq:freg-apx}); for~$(ii)$ we applied~\cref{lem:fval-err} to the right-hand side of~$|\vphiapx(x) - \vphi(x)| \le \max_{y \in Y} |f(x,y) - \fapx_1(x,y)|$. 
Returning to~\cref{eq:first-order-acc-incomplete} we now estimate
\begin{align*}
\|\nabla \vphiapx_{2\Lxxp_1}(x^*)\| 
&\le 2 \Sx\left(x_t,  \gx \fapx(x_t,y_t), 3\Lxxp_1 + {\mu^2}/{\rho_1} \right) + 4\Dy\sqrt{\Lxx \rho_1} \\
&\le \sqrt{\tfrac{8}{T}\left(3\Lxxp_1  + {\mu^2}/{\rho_1} \right) \left( \vphi(x_0) - \vphi(x_T) + \tfrac{3}{2} \rho_1 \Dy^2 \right)} + 4 \Dy\sqrt{\Lxx \rho_1} \\
&\stackrel{\cref{eq:first-order-strong-coupling-condition}}{\le} 
\sqrt{\tfrac{8}{T}\left(3\Lxxp_1  + {\mu^2}/{\rho_1} \right) \Big( \vphi(x_0) - \vphi(x_T) + \tfrac{3}{80000 \Lxxp_1^{\vphantom{A^a}}} \veps^2 \Big)} + \tfrac{1}{50} \veps \\
&\le
\sqrt{\tfrac{8}{T}\left(3\Lxxp_1  + {\mu^2}/{\rho_1} \right) \left( \vphi(x_0) - \vphi(x_T) \right)} + \tfrac{3}{100} \veps \sqrt{\tfrac{1}{3\Lxxp_1^{\vphantom{A^a}}T} \left( 3\Lxxp_1  + \mu^2/\rho_1 \right)} + \tfrac{1}{50} \veps \\
&\stackrel{\cref{eq:first-order-complexity}}{\le}
\Big( \sqrt{\tfrac{8}{700}}  + \tfrac{2\sqrt{3}}{100} + \tfrac{1}{50} \Big) \veps 
< \tfrac{1}{6}\veps.
\end{align*}
Finally, by applying~\cref{th:upper-bound} we conclude that~$\|\nabla \vphi_{2\Lxxp_1}(x^*)\| \le \veps$, as required.
\qed

\subsection{Proof of~\cref{prop:subgradient-scheme}}
We first observe that~$\fapx(\cdot, y) \equiv \fapx_2(\cdot,y)$ is~$(\sigma_0+\sigma_2\Dy^2)$-Lipschitz for any~$y \in Y$ as can be seen from
\begin{equation}
\label{eq:quad-grad-bound}
\| \gx \fapx(x,y) \| \le \| \gx f(x,y) \| + \|\gx \fapx(x,y) -  \gx f(x,y) \| \le \sigma_0 + \sigma_2 \Dy^2,
\end{equation}
where the last step is by~\cref{lem:gx-err}. Moreover, by~\cref{lem:gx-lip}~$\gx\fapx(\cdot,y)$ is~$\Lxxp_2$-Lipschitz, and thus~$\vphiapx$ is~$\Lxxp_2$-weakly convex.
These two observations allow us to adapt the analysis of the projected subgradient method from~\cite[Theorem 31]{jin2020local} (initially carried out in~\cite{davis2018stochastic}) to~\cref{alg:second-order} in order to establish that~$x^*$ is an approximate FOSP for~\cref{opt:min-max-apx-2} -- and after that allude to~\cref{th:upper-bound}.

For brevity, let~$\fapx \equiv \fapx_2$ and~$\vphiapx = \max_{y \in Y} \fapx(x,y)$. 
Observe that, for any~$x \in X$ and iterate~$(x_t,y_t)$, \vspace{-0.2cm}
\begin{align}
\vphiapx(x)
\ge \fapx(x,y_t)  
&\ge \fapx(x_t,y_t) + \langle \gx \fapx(x_t,y_t), x - x_t \rangle - \tfrac{1}{2}\Lxxp_2 \| x - x_t \|^2 \notag\\ 
&\ge \vphiapx(x_t) - \delta + \langle \gx \fapx(x_t,y_t), x - x_t \rangle - \tfrac{1}{2} \Lxxp_2 \| x - x_t \|^2,
\label{eq:quad-fun-decrease}
\end{align}
where the last step is by definition of~\ApproxMax.
Moreover, let~$x_t^+ = \argmin_{x \in X} \{ \vphiapx(x) + \Lxxp_2 \|x - x_t\|^2 \}$ be the proximal mapping of~$x_t$, so that (cf.~\cref{eq:moreau-explicit})
\begin{equation}
\label{eq:quad-moreau-step}
\nabla\vphiapx_{2\Lxxp_2}(x_t^{\vphantom+})= 2\Lxxp_2 (x_t^{\vphantom+} - x_t^+). 
\end{equation}

\proofstep{1}. 
Let~\cref{alg:second-order} be run with~$\Naive = 0$, i.e.,~$x_{t+1} = \proj_{X} [x_{t} - \gamx \gx \fapx(x_{t},y_t)]$; cf.~\cref{line:grad-2-approx}. 
Then
\begin{align}
\label{eq:quad-chained-simple}
	\vphiapx_{2\Lxxp_2}(x_{t+1}) 
&\stackrel{\phantom{\cref{eq:quad-grad-bound}}}{\le}
	\vphiapx(x_t^+) + \Lxxp_2 \|x_{t+1}^{\vphantom+} - x_t^+\|^2 \notag\\
&\stackrel{\phantom{\cref{eq:quad-grad-bound}}}{\le}
	\vphiapx(x_t^+) + \Lxxp_2 \|x_{t} - \gamx \gx \fapx(x_{t},y_t) - x_t^+\|^2 \notag\\
&\stackrel{\cref{eq:quad-grad-bound}}{\le}
	\vphiapx(x_t^+) + \Lxxp_2 \|x_{t}^{\vphantom+} - x_t^+\|^2 + 2 \gamx \Lxxp_2  \langle \gx \fapx(x_{t},y_t), x_t^+ - x_t^{\vphantom+} \rangle 
	+ \gamx^2 \Lxxp_2  (\sigma_0 + \sigma_2 \Dy^2)^2 \notag\\
&\stackrel{\phantom{\cref{eq:quad-grad-bound}}}{=} 
	\vphiapx_{2\Lxxp_2}(x_t) + 2 \gamx\Lxxp_2  \langle \gx \fapx(x_{t},y_t), x_t^+ - x_t^{\vphantom+} \rangle 
	+ \gamx^2 \Lxxp_2  (\sigma_0 + \sigma_2 \Dy^2)^2 \notag\\
&\stackrel{\cref{eq:quad-fun-decrease}}{\le}
	\vphiapx_{2\Lxxp_2}(x_t) + 2 \gamx\Lxxp_2 \left(\vphiapx(x_t^+) - \vphiapx(x_t^{\vphantom+}) + \delta + \tfrac{1}{2} \Lxxp_2 \| x_t^+ - x_t^{\vphantom+} \|^2 \right) 
	+  \gamx^2 \Lxxp_2 (\sigma_0 + \sigma_2 \Dy^2)^2,
\end{align}
where the second line relies on the projection lemma.
Repeating this for~$t \in \{0,...,T-1\}$ we get
\[
\vphiapx_{2\Lxxp_2}(x_T)
\le 	
\vphiapx_{2\Lxxp_2}(x_0) + 2 \gamx \Lxxp_2 \sum_{t = 0}^{T-1}\left(\vphiapx(x_t^+) - \vphiapx(x_t^{\vphantom+}) + \delta + \tfrac{1}{2}\Lxxp_2 \| x_t^+ - x_t^{\vphantom+} \|^2 \right) 
	+ \gamx^2 T \Lxxp_2 (\sigma_0 + \sigma_2 \Dy^2)^2,
\]
whence, by rearranging and dividing over~$2\gamx\Lxxp_2 T$, 
\begin{equation}
\label{eq:quad-telescoped-simple}
\frac{1}{T} \sum_{t = 0}^{T-1} \vphiapx(x_t^{\vphantom+}) - \vphiapx(x_t^+) - \frac{1}{2} \Lxxp_2 \| x_t^+ - x_t^{\vphantom+} \|^2 
\le 
\frac{\vphiapx_{2\Lxxp_2}(x_0) - \vphiapx_{2\Lxxp_2}(x_T)}{2\gamx \Lxxp_2 T} + \frac{1}{2} \gamx(\sigma_0 + \sigma_2 \Dy^2)^2
+ \delta.
\end{equation}
On the  other hand, by~$\Lxxp_2$-strong convexity of~$\vphiapx(\cdot) + \Lxxp_2 \|\cdot - x_t \|^2$ and the definition of~$x_t^+$ we get
\begin{align}
\vphiapx(x_t^{\vphantom+}) - \vphiapx(x_t^+) - \tfrac{1}{2} \Lxxp_2 \| x_t^+ - x_t^{\vphantom+} \|^2 
&= 
\vphiapx(x_t^{\vphantom+}) + \Lxxp_2 \| x_t - x_t \|^2 - (\vphiapx(x_t^+) + \Lxxp_2 \| x_t^+ - x_t^{\vphantom+} \|^2) + \tfrac{1}{2} \Lxxp_2 \| x_t^+ - x_t^{\vphantom+} \|^2 \notag\\
&=
\vphiapx(x_t^{\vphantom+}) + \Lxxp_2 \| x_t - x_t \|^2 - \min_{x \in X} \left\{ \vphiapx(x_t^+) + \Lxxp_2 \| x_t^+ - x_t^{\vphantom+} \|^2 \right\} + \tfrac{1}{2} \Lxxp_2 \| x_t^+ - x_t^{\vphantom+} \|^2 \notag\\
&\ge
\Lxxp_2\| x_t^+ - x_t^{\vphantom+} \|^2 
\stackrel{\cref{eq:quad-moreau-step}}{=} \tfrac{1}{4\Lxxp_2} \big\|\nabla\vphiapx_{2\Lxxp_2}(x_t^{\vphantom+})\big\|^2.
\label{eq:quad-Moreau-grad-simple}
\end{align}
Plugging this into~\cref{eq:quad-telescoped-simple} we arrive at
\begin{align}
\label{eq:quad-exp-bound}
\E [\|\nabla\vphiapx_{2\Lxxp_2}(x_s^{\vphantom+})\|^2]
&\le 
\tfrac{2}{\gamx  T} \Lxxp_2 \left( \vphiapx_{2\Lxxp_2}(x_0) - \vphiapx_{2\Lxxp_2}(x_T) \right) + 2 \Lxxp_2 \gamx (\sigma_0 + \sigma_2 \Dy^2)^2 
+ 4\Lxxp_2 \delta \nn
&\stackrel{(a)}{\le}
\tfrac{2}{\gamx  T} \Lxxp_2 \left( \vphiapx(x_0) - \vphiapx(x_T) \right) + 2 \Lxxp_2 \gamx (\sigma_0 + \sigma_2 \Dy^2)^2 
+ 4\Lxxp_2 \delta \nn
&\stackrel{(b)}{\le}
\tfrac{2}{\gamx  T} \Lxxp_2 (\Gap + \rho_2 \Dy^3) + 2 \Lxxp_2 \gamx (\sigma_0 + \sigma_2 \Dy^2)^2 
+ 4\Lxxp_2 \delta;
\end{align}
here for~$(a)$ we used that~$\vphiapx_{2\Lxxp_2}(x_0) = \min_{x \in X} \{ \vphiapx (x) + \Lxxp_2 \|x - x_0 \|^2 \} \le \vphiapx (x_0)$; for~$(b)$ we used~\cref{lem:fval-err}.
Whence by Markov's inequality, for any~$\prob \in (0,1)$ with probability at least~$1-\prob$ one has
\[
\begin{aligned}
\|\nabla\vphiapx_{2\Lxxp_2}(x_s^{\vphantom+})\|^2
&\le 
\tfrac{4}{\prob} 
\left( \tfrac{1}{\gamx  T}\Lxxp_2 (\Gap + \rho_2 \Dy^3) + \gamx \Lxxp_2 (\sigma_0 + \sigma_2 \Dy^2)^2 
+ 2\Lxxp_2 \delta \right) \\
&\le 
\tfrac{8}{\prob} 
\Big( (\sigma_0 + \sigma_2 \Dy^2) \sqrt{\tfrac{1}{T}\Lxxp_2 (\Gap + \rho_2 \Dy^3)} + \Lxxp_2 \delta \Big) 
< 
\tfrac{1}{150}\veps^2,
\end{aligned}
\]
where for the second line we substituted~$\gamx$ from~\cref{eq:quad-algo-params} (note that this choice of~$\gamx$ balances the two terms) and then performed an explicit calculation by plugging in~$T$ and~$\delta$ from~\cref{eq:quad-algo-params}. 
Finally, under this event~\cref{eq:second-order-diameter} allows to apply~\cref{th:upper-bound} (replacing~$\veps$ with~$\veps/6$) and conclude that~$\|\nabla\vphi_{2\Lxxp_2}(x_s)\| \le \veps$.

\proofstep{2}. Conforming with the second claim we intend to prove, from now on we shall assume that
\begin{equation}
\label{eq:quad-extra-asm}
24 \sigma_2 \Dy^2 \le \veps\sqrt{\prob},
\end{equation}
and consider the iterates of~\cref{alg:second-order} with~$\Naive = 1$---i.e.,~$x_{t+1} = \proj_{X} [x_{t} - \gamx \gx f(x_{t},y_t)]$; cf.~\cref{line:grad-2-naive}.
First of all, let us correct~\cref{eq:quad-chained-simple} for the discrepancy between~$\gx f(x_{t},y_t)$ and~$\gx \fapx(x_{t},y_t)$: to this end, 
\begin{align*}
	\vphiapx_{2\Lxxp_2}(x_{t+1}) 
&\le 
	\vphiapx(x_t^+) + \Lxxp_2 \|x_{t+1}^{\vphantom+} - x_t^+\|^2 \\
&\stackrel{(a)}{\le}
	\vphiapx(x_t^+) + \Lxxp_2 \|x_{t} - \gamx \gx f(x_{t},y_t) - x_t^+\|^2 \\
&\stackrel{(b)}{\le}
	\vphiapx_{2\Lxxp_2}(x_t) + 2\Lxxp_2 \gamx \langle \gx f(x_{t},y_t) , x_t^+ - x_t^{\vphantom+}   \rangle + \Lxxp_2 \gamx^2 \sigma_0^2  \\
&\stackrel{(c)}{\le}
	\vphiapx_{2\Lxxp_2}(x_t) + 2\Lxxp_2 \gamx \left( \langle \gx \fapx(x_{t},y_t) , x_t^+ - x_t^{\vphantom+}   \rangle  + \sigma_2 \Dy^2 \| x_t^+ - x_t^{\vphantom+} \| \right)  + \Lxxp_2\gamx^2 \sigma_0^2  \\
&\stackrel{(d)}{\le} 
	\vphiapx_{2\Lxxp_2}(x_t) + 2\gamx\Lxxp_2 \left( \langle \gx \fapx(x_{t},y_t) , x_t^+ - x_t^{\vphantom+}   \rangle  + \tfrac{1}{2} \Lxxp_2 \| x_t^+ - x_t^{\vphantom+} \|^2 + \tfrac{1}{2\Lxxp_2}\sigma_2^2 \Dy^4 \right)   + \Lxxp_2 \gamx^2 \sigma_0^2  \\
&\stackrel{\cref{eq:quad-fun-decrease}}{\le}
	\vphiapx_{2\Lxxp_2}(x_t) + 2\gamx\Lxxp_2 \left( \vphiapx(x_t^+) - \vphiapx(x_t^{\vphantom+}) + \delta + \Lxxp_2 \| x_t^+ - x_t^{\vphantom+} \|^2 + \tfrac{1}{2\Lxxp_2} \sigma_2^2 \Dy^4 \right)   + \Lxxp_2 \gamx^2 \sigma_0^2.
\end{align*}
Here~$(a)$ is by the projection lemma;~$(b)$ by the Lipschitz assumption;~$(c)$ by Cauchy-Schwarz and~\cref{lem:gx-err};~$(d)$ by~$uv \le (u^2 + v^2)/2$ for~$u,v \in \R$.
Whence we arrive at a counterpart of~\cref{eq:quad-telescoped-simple}: 
\begin{equation}
\label{eq:quad-telescoped-complex}
\frac{1}{T} \sum_{t = 0}^{T-1} \vphiapx(x_t^{\vphantom+}) - \vphiapx(x_t^+) - \Lxxp_2 \| x_t^+ - x_t^{\vphantom+} \|^2 
\le 
\frac{\vphiapx_{2\Lxxp_2}(x_0) - \vphiapx_{2\Lxxp_2}(x_T)}{2\gamx \Lxxp_2 T} + \frac{\gamx \sigma_0^2}{2}
+ \delta + \frac{\sigma_2^2 \Dy^4}{2\Lxxp_2}. 
\end{equation}
On the other hand, in the same spirit as for~\cref{eq:quad-Moreau-grad-simple} we get
\begin{equation}
\label{eq:quad-Moreau-grad-complex}
\begin{aligned}
\vphiapx(x_t^{\vphantom+}) - \vphiapx(x_t^+) - \Lxxp_2 \| x_t^+ - x_t^{\vphantom+} \|^2 
&= 
\vphiapx(x_t^{\vphantom+}) + \Lxxp_2 \| x_t - x_t \|^2 - (\vphiapx(x_t^+) + \Lxxp_2 \| x_t^+ - x_t^{\vphantom+} \|^2) \\
&=
\vphiapx(x_t^{\vphantom+}) + \Lxxp_2 \| x_t - x_t \|^2 - \min_{x \in X} \left\{ \vphiapx(x_t^+) + \Lxxp_2 \| x_t^+ - x_t^{\vphantom+} \|^2 \right\} \\
&\ge
\tfrac{1}{2} \Lxxp_2\| x_t^+ - x_t^{\vphantom+} \|^2 
\stackrel{\cref{eq:quad-moreau-step}}{=} \tfrac{1}{8\Lxxp_2}\big\|\nabla\vphiapx_{2\Lxxp_2}(x_t^{\vphantom+})\big\|^2.
\end{aligned}
\end{equation}
Plugging~\cref{eq:quad-Moreau-grad-complex} into~\cref{eq:quad-telescoped-complex} and proceeding in the same way as in~\cref{eq:quad-exp-bound}, we conclude that w.p.~$\ge 1-\prob$,
\[
\begin{aligned}
\|\nabla\vphiapx_{2\Lxxp_2}(x_s^{\vphantom+})\|^2
&\le 
\tfrac{8}{\prob} 
\left( \tfrac{1}{\gamx T} \Lxxp_2 (\Gap + \rho_2 \Dy^3) + \gamx \Lxxp_2 \sigma_0^2 
+ 2\Lxxp_2 \delta + \sigma_2^2 \Dy^4\right) \\
&\le 
\tfrac{16}{\prob} 
\Big( \sigma_0 \sqrt{\tfrac{1}{T} \Lxxp_2 (\Gap + \rho_2 \Dy^3)} + \Lxxp_2 \delta \Big) + \tfrac{8}{\prob}\sigma_2^2 \Dy^4
\stackrel{\cref{eq:quad-extra-asm}}{\le}
\left(\tfrac{1}{75} + \tfrac{1}{72}\right) \veps^2
< \tfrac{1}{36}\veps^2.
\end{aligned}
\]
In order to verify the claim, it remains to apply~\cref{th:upper-bound}. 
\qed

\section{Background on weak convexity and the Moreau envelopes}
\label{app:weakly-convex}

\subsection{Characterization of weakly convex functions}
\label{app:weakly-convex-basic}
Let~$\XX$ be a finite-dimensional Euclidean space with inner product~$\lang \cdot, \cdot \rang$  and norm~$\|x\| = \sqrt{\lang x, x \rang}$. Moreover, we identify with~$\XX$ with its dual space~$\XX^*$ (in particular,~$\|\cdot\|^* = \|\cdot\|$). Also we let~$\bar \R = (-\infty, +\infty]$ and tacitly assume that all arising convex functions are lsc and {proper} (not identically~$+\infty$).
Recall that any such function~$\phi: \XX \to \bar\R$ admits a variational representation as the pointwise supremum over a non-empty family of affine functions, and this is in fact a criterion (see, e.g.,~\cite[Thm.~8.13]{rockafellar2009variational}). 
The set of all affine minorants of~$\phi$ at given~$x \in \XX$ generates the corresponding set of linear functionals, 
called the {\em subdifferential of~$\phi$ at~$x$} and denoted~$\partial \phi(x)$---i.e.,
\begin{equation}
\label{eq:subdifferential-cvx}
\partial \phi(x) := \{ \xi \in \XX: \; \phi(x') \ge \phi(x) + \langle \xi, x'-x \rangle \;\; \forall x' \in \XX\}. 
\end{equation}
This set is non-empty at any interior point of {\em effective domain}~$\dom(\phi) := \{x \in \XX: \phi(x) < +\infty\}$, and its elements are called {\em subgradients of~$\phi$ at~$x$}. 
(For simplicity, we assume~${\rm int}(\dom(\phi))$ is non-empty.)

\begin{definition}
\label{def:weak-convexity}
$\phi: \XX \to \bar\R$ is called {\em$\lam$-weakly convex} (for~$\lam \ge 0$) if the function
$
\phi(\cdot) + \frac{\lam}{2} \|\cdot\|^2
$
is convex.
\end{definition}

From the above characterization of (proper, lsc) convex functions in terms of subdifferentials it immediately follows that~$\phi$ is~$\lam$-weakly convex if and only if the function
$
\phi_{\lam,x}(\cdot) := \phi(\cdot) + \tfrac{\lam}{2} \|\cdot-x\|^2
$
is convex for any~$x \in \XX$. 
Indeed, by~\cref{def:weak-convexity}~$\phi_{\lam,0}$ is convex; meanwhile, for any~$x \in \XX$ we have 
\[
\phi_{\lam, x}(u) = \phi_{\lam, 0}(u) + \tfrac{\lam}{2} \left( \|u - x\|^2 - \|u\|^2 \right).
\]
By applying~\cref{eq:subdifferential-cvx} to~$\phi_{\lam,0}$ we see that~$\phi_{\lam, x}$ has subdifferential
$
\partial \phi_{\lam,x}(u) = \partial \phi_{\lam,0}(u) - \lam x
$
(the difference being in the Minkowski sense), and thus is convex with~$\dom(\phi_{\lam,x}) = \dom(\phi)$. 
It is then natural to define the {\em weak subdifferential} of~$\phi$ at~$x$ by
$
\partial\phi(x) = \partial\phi_{\lam,x}(x). 
$
Equivalently,~$\partial\phi(u) = \partial\phi_{\lam,x}(u) + \lam (x-u)$ for any~$ u,x \in \XX$---or, in other words,
\begin{equation}
\label{eq:subdifferential-weak}
\partial \phi(x) = \{ \xi \in \XX: \; \phi(x') \ge \phi(x) + \langle \xi, x'-x \rangle - \tfrac{\lam}{2} \|x'-x\|^2, \;\; \forall x,x' \in \XX\}. 
\end{equation}
Thus,~$\partial \phi(x)$ belongs to the Fr\'echet (local) subdifferential at~$x$, defined as the set of~$\xi \in \XX$ satisfying
\begin{equation}
\label{eq:subdifferential-frechet}
\phi(x') \ge \phi(x) + \langle \xi, x'-x \rangle + o(\|x'-x\|) \quad \text{as} \;\; x' \to x.
\end{equation}
cf.~\cite[Def.~8.3]{rockafellar2009variational}. Generally, the property in~\cref{eq:subdifferential-frechet} is much weaker than that in~\cref{eq:subdifferential-weak}; 
however, for weakly convex functions the weak and Fr\'echet subdifferentials coincide (see~\cite[Theorem~12.17]{rockafellar2009variational}). 

\paragraph{Hessian-based criterion.}
Recall Alexandrov's theorem (\cite{alexandrov1939application}, see also~\cite[Thm.~2.1]{rockafellar1999second}): a convex~$\phi: \XX \to \bar\R$ is twice differentiable almost everywhere
(and hence~$\nabla^2 \phi \succeq 0$) in the interior of its domain.
Under~$C^1$ regularity it is a criterion: if~$\phi: \XX \to \bar\R$ is continuously differentiable on~${\rm int}(\dom(\phi))$, and~$\nabla^2 \phi \succeq 0$ almost everywhere on~${\rm int}(\dom(\phi))$, then~$\phi$ is convex. 
From Definition~\ref{def:weak-convexity} we see that this criterion extends to~$\lam$-weakly convex functions if we replace~$\nabla^2 \phi \succeq 0$ with~$\nabla^2 \phi \succeq -\lam I$.

\subsection{Upper envelope as a weakly convex function}

Let~$X \subseteq \XX$ be convex with nonempty interior. If~$\phi \to X$ is~$\lam$-smooth, i.e., such that~$\nabla \phi$ exists and
\[
\|\nabla \phi(x') - \nabla \phi(x) \| \le \Lxx \|x' - x\| \quad \forall x,x' \in X,
\]
then~$\phi$ is~$\lam$-weakly convex. (This follows from the Hessian-vased criterion above, as~$\nabla \phi(x)$ exists almost everywhere on~$X$ and its eigenvalues are in~$[-\lam,\lam]$.) 
More generally, the upper envelope of a family of such functions is~$\lam$-weakly convex. That is, given~$f: X \times Y$ (in the setup of~\cref{opt:min-max}) satisfying~\cref{ass:gradx} with~$\lam < \infty$, the \odima{max-function}~$\vphi(x) := \max_{y \in Y} f(x,y)$ is~$\lam$-weakly convex.
Indeed, note that~$f_{\lam,x'}(\cdot, y) = f(\cdot, y) + \frac{\lam}{2}\| \cdot - x'\|^2$ is convex for any~$y \in Y$ (cf.~\cref{def:weak-convexity}), therefore 
\[
\vphi(x) + \tfrac{\lam}{2} \|x-x'\|^2 = \max_{y \in Y} f_{\lam,x'}(x, y)
\]
is convex, i.e.,~$\vphi$ is~$\lam$-weakly convex. 
Moreover,~$\partial \vphi(x)$ can be expressed explicitly through~$\gx f(x,y)$.
\begin{lemma}
\label{lem:danskin}
Let~$f: X \times Y \to \R$ be continuous and satisfy Assumption~\ref{ass:gradx} with~$\lam < \infty$,~$X$ be convex, and~$Y$ be convex and compact.
Then~$\partial \vphi(x)$ is the closed convex hull of the set of active gradients:
\[
\partial \vphi(x) = \clconv \{ \gx f(x,y^*), \forall y^* \in \textstyle\Argmax_{y \in Y} f(x,y) \}.
\]  
\end{lemma}
\begin{proof}
We apply Danskin's theorem (\cite[Sec.~B.5]{bertsekas1997nonlinear}) to the function
$
\vphi_{\lam,x'}(x) = \max_{y \in Y} f_{\lam,x'}(x, y),
$
\[
\begin{aligned}
\partial \vphi_{\lam,x'}(x) 
&= \clconv \{ \gx f_{\lam,x'}(x,y^*), y^* \in \textstyle\Argmax_{y \in Y} f_{\lam,x'}(x,y) \}, \\
&= \clconv \{ \gx f(x,y^*) + \lam(x-x'), y^* \in \textstyle\Argmax_{y \in Y} f(x,y) \}.\\
\end{aligned}
\]
This holds because each~$f_{\lam,x'}(\cdot,y)$ is convex, and the set of maximizers for~$f_{\lam,x'}(x,y)$ (over~$y \in Y$) is the same as for~$f(x,y)$.
It remains to note that~$\partial \vphi(x) = \partial \vphi_{\lam,x'}(x)-\lam(x-x')$ by~\cref{def:weak-convexity}.
\end{proof}

\subsection{Characterization of~$(\veps,2\lam)$-FOSP in terms of the \odima{max-function}}

The following property of~$(\veps,2\lam)$-FOSP for~\cref{opt:min-max} has been leveraged in the proof~\cref{prop:upper-coupled}. 
It extends~\cite[Lem.~2.2]{davis2019stochastic} to the case~$X \ne \XX$ with guarantee~\cref{eq:moreau-to-primal} in terms of~$\partial\vphi$ rather than~$\partial(\vphi + I_X)$.

\begin{proposition}[{\cite[Proposition~5.1]{ostrovskii2020efficient}}]
\label{prop:moreau-to-primal}
Let~$\phi: X \to \R$ be~$\Lxx$-weakly convex, then it holds that
\[
\nabla \phi_{2\Lxx}(x) = 2\Lxx (x - x^+(x)), \;\; \text{where} \;\; x^+(x) = x^+_{\frac{1}{2\Lxx}\phi}(x) \;\; \text{for any} \;\; x \in X.
\] 
where~$\phi_{2\Lxx}$ is the~$2\lam$-Moreau envelope of~$\phi$ (cf.~\cref{def:moreau}).
Moreover, if~$\|\nabla \phi_{2\Lxx}(x)\| \le \veps$, then
\begin{equation}
\label{eq:moreau-to-primal}
\min_{\xi \in \partial \phi(x^+)} \Sx(x^+,\xi,2\Lxx) \le \veps,
\end{equation}
where we define the functional
$
\Sx^2(x, \xi, \bar\lam) := 2\bar\lam\max_{u \in X} \big\{-\langle \xi, u - x \rangle - \tfrac{\bar\lam}{2} \|u-x\|^2 \big\}
$
for~$x, \xi \in \XX$,~$\bar\lam > 0$.
\end{proposition}

\section{Solving~\cref{eq:krylov-at-x}: quadratic optimization on a joint Krylov subspace}
\label{app:krylov}

Following ~\cite[Appendix~A]{carmon2020first}, in this section we describe a procedure for maximizing a general quadratic function~$\Psi_{H,g}(y)$, cf.~\cref{eq:quad-carmon}, over the intersection of~$B_d(2\Ry)$ and the joint Krylov subspace
\begin{equation}
\label{def:krylov}
\cK_{2m} = \cK_{2m}(H, \{g,\xi\}) := \textup{span} \left( \{H^{j} g, H^{j} \xi  \}_{j \in \{0,...,m-1\}} \right)
\end{equation}
in~$O(m)$ computations of the matrix-vector product~$y \mapsto Hy$ and elementwise vector operations on~$E_\y$ (including the inner product).
This procedure is given in pseudocode in~\cref{alg:carmon}. 
It consists of two steps which we are now about to outline and discuss. 

\paragraph{Finding the right basis for~$\cK_{2m}$.}
In the first stage of the algorithm (up to~\cref{line:wtH}) we construct an orthonormal basis of~$\cK_{2m}$ in which the associated linear operator (corresponding to~$H$ in the initial basis), when restricted to~$\cK_{2m}$, is represented by a block tridiagonal matrix~$\wt H$ with blocks of size~$2$ (which pertains to the two vectors~$\{g,\xi\}$ generating~$\cK_{2m}$). 
In other words, we construct~$\wt H  \in \R^{2m \times 2m}$ and~$Q \in \R^{d \times 2m}$ such that
$
\wt H = Q^{\top} H Q,
$
where~$Q$ has orthonormal columns (i.e.,~$Q^\top Q = I_{2m}$) and
\[
\wt H = 
\begin{bmatrix} 
\alpha_1^{\vphantom\top} & \beta_2^{\top} & & \\ 
\beta_2 & \ddots & \ddots & \\ 
&\ddots & \ddots & \beta_m^{\top} \\
& & \beta_m & \alpha_m
\end{bmatrix},
\]
where each~$\alpha_i \in \R^{2 \times 2}$ is symmetric and each~$\beta_i \in \R^{2 \times 2}$ is upper-triangular (so~$\wt H$ is pentadiagonal). 
In the large-scale context, where~$H$ is accessed through the~$y \mapsto H y$ oracle, the one can construct such a decomposition in a computationally feasible manner via the block Lanczos process~\cite{cullum1974block,golub1977block}: 
\begin{enumerate}
\item 
Form~$q_1 \in \R^{d \times 2}$ as the Gram-Schmidt orthonormalization of~$\{g,\xi\}$. 
\item 
Compute~$\alpha_1 = q_1^\top H q_1$ and put~$\beta_1 = 0_{2 \times 2}$. (We let~$0_{a \times b}$ be the~$a \times b$ zero matrix.)
\item 
For~$t \in \{1, m-1\}$ iterate:
\[
\begin{aligned}
q_{t+1}' 
	= Hq_t^{\vphantom \top} - q_t^{\vphantom \top} \alpha_t^{\vphantom \top} - q_{t-1}^{\vphantom \top} \beta_t^{\top}, \quad\quad
(q_{t+1}, \beta_{t+1}) 
	= \text{QR}(q_{t+1}'), \quad\quad
\alpha_{t+1}^{\vphantom \top} 
	= q_{t+1}^{\top} H q_{t+1}^{\vphantom \top}.
\end{aligned}
\]
\end{enumerate}
Here~$\text{QR}(A)$ is the QR decomposition of a squared matrix~$A$---i.e., the ordered pair~$(U,R)$ of square matrices such that~$A = UR$,~$U$ is orthogonal, and~$R$ is upper-triangular. To see that the above process is sound, observe that it incrementally builds up~$q_t, \alpha_t, \beta_t$ that satisfy the matrix relations
\[
H [q_1 \; \cdots \; q_t] = [q_1 \; \cdots \; q_t \, | \, q_{t+1}] 
\begin{bmatrix} 
\alpha_1^{\vphantom\top} & \beta_2^{\top} & & \\ 
\beta_2 & \ddots & \ddots & \\ 
&\ddots & \ddots & \beta_t^{\top} \\
& & \beta_t & \alpha_t \\
\hline 
& & & \beta_{t+1}
\end{bmatrix} \;\; \text{for} \;\; t \in \{1, ..., m-1\}, 
\quad \text{and} \quad
H q_m = [q_1 \; \cdots \; q_m] 
\begin{bmatrix}  
0 \\ \vdots \\ 0 \\ \beta_{m}^\top \\ \alpha_m 
\end{bmatrix}
\] 
while ensuring that~$\beta_t$'s are upper-triangular and~$q_t$'s form an orthonormal system. These relations amount to the identity~$HQ = Q \wt H$ with~$Q = [q_1 \; \cdots \; q_m]$---and thus to~$Q^\top H Q = \wt H$ as desired.
The first part of~\cref{alg:carmon} (until~\cref{line:wtH}) implements the block Lanczos process with an explicitly rendered QR factorization step. 

\paragraph{Passing to the equivalent problem and solving it.}
Once~$\wt H$ and~$Q$ have been constructed, we compute~$\wt g = Q^\top g$ and the top eigenvalue~$\omega_0$ of~$\wt H$ (cf.~\cref{line:lammin}). 
The first of these two steps amounts to~$O(m)$ inner products, and the second one can be done in~$O(m)$ arithmetic operations since~$\wt H$ is pentadiagonal.
Using these data, we recast the initial maximization problem (cf.~\cref{eq:krylov-at-x}) by changing the variable~$y$ to~$z = Q^\top y \in \R^{2m}$---i.e., we pass to
\begin{equation}
\label{eq:krylov-transformed}
\max_{z \in B_{2m}(2\Ry)} \tfrac{1}{2}  z^\top \wt H z +  \wt g^\top z.
\end{equation}
The new problem is equivalent to~\cref{eq:krylov-at-x} in the following sense: if~$z$ is~$\delta$-suboptimal in~\cref{eq:krylov-transformed}, then~$y = Q z$ is~$\delta$-suboptimal in~\cref{eq:krylov-at-x}. 
Thus we can focus on~\cref{eq:krylov-transformed} and then recover a candidate solution to~\cref{eq:krylov-at-x} by multiplying with~$Q$ (which has essentially the same computational cost as forming~$\wt g = Q^\top g$ from~$g$). 

The final part of~\cref{alg:carmon} (after~\cref{line:lammin}) consists of solving~\cref{eq:krylov-transformed}. Here we proceed as follows. 
\begin{enumerate}
\item
We first test if~$\wt H$ is negative-semidefinite (by inspecting~$\sign(\omega_0)$) and, if yes, whether~$B_{2m}(2\Ry)$ contains an {\em unconstrained} maximizer~$z_0$---an optimal solution to the linear system~$\wt H z + \wt g = 0$. 
If the answers to both questions are positive, we in fact found a maximizer on~$B_{2m}(2\Ry)$, so we output~$Q z_0$ and terminate.
Note that here if~$\wt H$ is negative-definite, then~$z_0$ exists, is unique, and is given by~$-\wt H^{-1} \wt g$; otherwise (i.e., when~$\omega_0 = 0$),~$\wt H$ is rank-deficient, so~$\wt H z + \wt g = 0$ is solvable whenever~$\wt g$ is in the range of~$\wt H$, and then the least-norm solution is given by~$z_0 = -\wt H^{\pinv} \wt g$, where~$\wt H^{\pinv}$ is the generalized inverse of~$\wt H$ (i.e.,~$\wt H^\pinv$ has the same eigenspaces as~$\wt H$ and reciprocal nonzero eigenvalues). Note that~$\wt H^{\pinv}$ can be computed in~$O(m)$ arithmetic operations (a.o.). %
\item 
If the above test failed (and thus we have not terminated), then~$\Psi_{\wt H, \wt g}$ has a unique maximizer at the boundary of~$B_{2m}(2\Ry)$, and this maximizer also solves the strictly concave problem
\[
z_{\omega} = \argmax_{z \in B_{2m}(2\Ry)} \tfrac{1}{2}  z^\top (\wt H - \omega I) z +  \wt g^\top z
\]
for the unique value of~$\omega > \max\{\underline \omega, 0\}$ such that~$\|z_{\omega}\| = \Ry$  (see, e.g.,~\cite[Corollary 7.2.2]{conn2000trust}). 
Thus, we can proceed by a root-finding method of choice: evaluation of the mapping~$\omega \mapsto \|z_{\omega}\|$ amounts to solving a well-defined linear system with a pentadiagonal matrix and thus can be done in~$O(m)$ a.o.;  thus, we find a high-accuracy maximizer---exact one in floating point arithmetic---in~$O(m)$ a.o. 
In particular, a good option is practice is to seek for the root of~$\psi(\omega) - {1}/{\Ry}$ with~$\psi(\omega) = 1/\|z_{\omega}\|$ 
by Newton's method as discussed in~\cite[Chapter~7]{conn2000trust}.
\end{enumerate}

\begin{algorithm}
\caption{Approximate maximization of a quadratic~$\Psi_{H,g}(y) = \frac{1}{2}y^\top H y +  g^\top y$ on a ball~$Y = B_d(2\Ry)$} 
\label{alg:carmon}
\begin{algorithmic}[1]
\Require $g \in \R^d$, oracle~$y \mapsto Hy$, $\Ry > 0$, $m \in \N$
\LineComment{{\em Restrict search to the joint Krylov subspace~$\cK_{2m}(H, \{g,\xi\})$, cf.~\cref{def:krylov}, by using~\cref{prop:carmon-perf}}.} %
\item[]
\LineComment{{\em Blockwise Lanczos iterations for~$\cK_{2m} = \cK_{2m}(H, \{g,\xi\})$ with blocks of size~$2$ (\cite[Appendix~A]{carmon2020first}); results in a block tridiagonal matrix~$\wt H = Q^\top H Q \in \R^{2m \times 2m}$ with~$Q \in \R^{d \times 2m}$:~$Q^\top Q = I$.}}
\LineComment{{$q_t \in \R^{d \times 2}$, $\alpha_t, \beta_t \in \R^{2 \times 2}$}}
\item[]
\State 
$q_0 = 0_{d \times 2}$ 
\Comment{{\em $d \times 2$ zero matrix}}
\State \label{line:gram-schmidt-init}
Draw $\xi \sim \Uniform(\mathds{S}^{d-1})$; \;\;
$u = \frac{1}{\|g\|} g$; \;\;
$w = \xi - \langle \xi, u \rangle u$; \;\;
$v = \frac{1}{\|w\|}w$
\State \label{line:gram-schmidt-init-cont}
$q_1 = [u,v]$; \; $\alpha_1 = q_1^{\top} H q_1^{\vphantom \top}$
\Comment{{\em $q_1$ is an orthonormalization of~$[g,\xi]$}}
\item[]
\For{$t \in \{1, ..., m-1\}$}
\State
$q_{t+1}' = Hq_t - q_t \alpha_t - q_{t-1}^{\vphantom \top} \beta_t^{\top}$
\Comment{{\em $\beta_1$ never used since~$q_0 = 0$}} 
\LineComment{{\em Compute $(q_{t+1}^{\vphantom'},\beta_{t+1}^{\vphantom'}) = \textup{QR}(q_{t+1}')$ via Gram-Schmidt:}}  
\State 
$[u',v'] = q_{t+1}'$ \Comment{{\em extract columns from~$q_{t+1}'$}}
\State
$u = \frac{1}{\|u'\|} u'$; \;\;
$w = v' - \lang v',u \rang u$; \;\;
$v = \frac{1}{\|w\|} w$
\State 
$q_{t+1} = [u,v]$; \;\;
$\beta_{t+1} = \begin{bmatrix} \lang u',u \rang & \lang v',u \rang \\ 0 & \lang v',v \rang \end{bmatrix}$
\Comment{{\em $q_{t+1}$ is an orthonormalization of~$q_{t+1}' = q_{t+1}^{\vphantom'} \beta_{t+1}^{\vphantom'}$}}
\State
$\alpha_{t+1}^{\vphantom \top} = q_{t+1}^{\top} H q_{t+1}^{\vphantom \top}$
\EndFor
\item[]
\State 
$Q = \begin{bmatrix} q_1 \; \cdots \; q_m \end{bmatrix}$ 
\Comment{{\em columns of~$Q $ form an orthonormal basis for~$\cK_{2m}$}}
\State
$\wt H = 
\begin{bmatrix} 
\alpha_1^{\vphantom\top} & \beta_2^{\top} & & \\ 
\beta_2 & \ddots & \ddots & \\ 
&\ddots & \ddots & \beta_m^{\top} \\
& & \beta_m & \alpha_m
\end{bmatrix}$
\label{line:wtH}
\Comment{{\em $\wt H = Q^\top H Q \in \R^{2m \times 2m}$ is block tridiagonal with~$2 \times 2$ blocks}}
\item[]
\State $\wt g = Q^\top g$
\State $\omega_0 = \lammax(\wt H)$
\label{line:lammin}
\item[]
\LineComment{{\em Eliminate the interior case: $\Psi_{H,g}$ restricted to~$\cK_{2m}$ is concave and maximized inside the ball:}} 
\If{$\omega_0 < 0$} 
\Comment{{\em $\wt H$ is full-rank, so the unconstrained maximizer~$z_0$ exists and is unique}}
\State $z_0 = -\wt H^{-1} \wt g$
\EndIf
\If{$\omega_0 == 0$ \textbf{and} $\wt H (\wt H^2)^{\pinv} \wt H  \wt g == \wt g$} 
\Comment{{\em $\wt H$ is rank-deficient, but~$\wt g$ is in its range}} 
\State
$z_0 = -\wt H^{\pinv} \wt g$
\EndIf
\If{$z_0$ is defined \textbf{and} $\|z_0\| \le \Ry$}
\State \Return $Q z_0$
\EndIf
\item[]
\LineComment{{\em If we have not terminated yet, the maximizer is on the boundary (\cite[Corollary 7.2.2]{conn2000trust})}} 
\State Find $\omega > \max\{\omega_0, \, 0 \}$ such that the norm of~$z_{\omega} = -(\wt H - \omega I)^{-1} \wt g$ is~$\Ry$ %
\State 
\Return $Qz_{\omega}$
\end{algorithmic}
\end{algorithm}

\paragraph{Time and memory expenses.}
By carefully inspecting~\cref{alg:carmon} one may verify that the Lanczos process takes~$O(1)$ matrix-vector products  and elementwise vector operations in~$E_{\y}$ per iteration, so~$O(m)$ such operations in total. Computation of~$\wt g$ takes~$O(m)$ inner products, and those of~$\omega_0$,~$\wt H^{\pinv}$, and~$(\wt H - \omega I)^{-1}$ each take~$O(m)$ a.o.
Finally, as we previously discussed, a root finder performs~$O(1)$ evaluations of~$\omega$ in the floating-point model of computation.
Hence, the total {\bf computational cost} of~\cref{alg:carmon} in a.o. is
\begin{equation}
\label{eq:time-cost}
O(m (d_\y + \tHVP)),
\end{equation}
where~$\tHVP$ is the cost of computing~$y \mapsto Hy$, and~$d_\y = \dim(E_{\y})$.
The {\bf memory expenses} are
\begin{equation}
\label{eq:memory-cost}
O(md_\y + \mHVP),
\end{equation}
where the first term allows to explicitly store~$Q$ (which we need in order to report the final solution $Q z_0$ or~$Q z_{\omega}$), and the second term is the memory needed for computing~$y \mapsto H y$ (i.e.,~$\mHVP \lsim d_\y^2$ or less if~$H$ has \odima{a} special structure). 
Note also that we can replace the term~$md_\y$ in~\cref{eq:memory-cost} with~$m+d_\y = O(d_\y)$ if~\cref{alg:carmon} in the situation where there is no need to report the {\em solution} in~$E_{\y}$ (e.g., if only the optimal value is of interest). 
Indeed, apart from reporting the final solution, we only need~$Q$ when computing~$\wt g = Q^\top g = [q_{1}^\top g; \; ...; \; q_{m}^\top g ]$---but this can be done incrementally during the Lanczos process, without ever having to memorize more than~$O(1)$ columns of~$Q$ at once.

\section*{Statements and Declarations}

\subsection*{Funding and/or Conflicts of interests/Competing interests}
\textit{Funding.} This work was supported by the NSF CAREER Award CCF2144985 and the AFOSR Young Investigator
Program Award FA9550-22-1-0192.

\vspace{0.1cm}

\noindent\textit{Competing and Conflicts of Interests.} The authors have no relevant financial or non-financial interests to disclose.

\subsection*{Other Declarations}

\noindent\textit{Author Contributions.} 
All authors contributed to the design, analysis, and writing of the paper.  All authors read and approved the final manuscript.

\vspace{0.1cm}

\noindent\textit{Availability of Data and Materials.} 
Data sharing is not applicable to this article as no new data were created or analyzed in this study.

\vspace{0.1cm}

\noindent\textit{Code Availability.} 
Not applicable.

\vspace{0.1cm}

\noindent\textit{Ethics Approval.} 
Not applicable. No human or animal subjects were involved in this research.

\vspace{0.1cm}

\noindent\textit{Consent to Participate.}
Not applicable.

\vspace{0.1cm}

\noindent\textit{Consent for Publication.}
Not applicable.

\bibliographystyle{unsrt}
\bibliography{references}

\end{document}